\documentclass[12pt]{amsart}
\pdfoutput=1
\usepackage{microtype}
\usepackage{etex}
\usepackage[active]{srcltx}
\usepackage{calc,amssymb,amsthm,amsmath,amscd, eucal,ulem}
\usepackage{alltt}
\usepackage{mathtools}
\usepackage{bbding}
\RequirePackage[dvipsnames,usenames]{color}

\input{mabliautoref.sty}
\input{xy}
\xyoption{all}
\input{kmacros3.sty}
\usepackage{tikz}
\usepackage{amsfonts, mathrsfs}
\usepackage[cal=boondox, calscaled=1.05]{mathalfa}
\usepackage{calligra}
\usepackage{amssymb}
\usepackage{stmaryrd} %power series bracket
\usepackage{dcpic, pictexwd}
\usepackage[left=1in,top=1in,right=1in,bottom=1in]{geometry}
\usepackage{bm}
\usepackage{verbatim}
\usepackage{upgreek}
\usepackage{systeme}
\makeatletter
\let\@wraptoccontribs\wraptoccontribs
\makeatother
\numberwithin{equation}{theorem}

\DeclareMathOperator{\ssHom}{ \sH\!\!\;\!\!\text{\calligra{\Large om}\,}}

\newcommand{\kay}{\mathcal{k}}
\newcommand{\el}{\mathcal{l}}

\renewcommand{\:}{\colon}

\newcommand{\p}{\mathfrak{p}}
\newcommand{\q}{\mathfrak{q}}

\DeclareMathOperator{\APic}{APic}
\DeclareMathOperator{\ADiv}{ADiv}
\DeclareMathOperator{\depth}{depth}

\DeclareMathOperator{\val}{val}

\DeclareMathOperator{\Ass}{Ass}

\DeclareMathOperator{\Norm}{N}

\DeclareMathOperator{\colim}{colim}

\DeclareMathOperator{\asr}{\#} %asymptotic splitting rank
\DeclareMathOperator{\length}{\lambda}

\newcommand{\Gal}{\textnormal{Gal}}

\DeclareMathOperator{\DIV}{Div}
\DeclareMathOperator{\Cl}{Cl}
\DeclareMathOperator{\Branch}{Branch}

\usepackage{setspace}
\usepackage{hyperref}
\usepackage{enumerate}
\usepackage{graphicx}
\usepackage[all,cmtip]{xy}

\usepackage{verbatim}

\theoremstyle{theorem}

\renewcommand{\sO}{\mathcal{O}}
\renewcommand{\sF}{\mathcal{F}}
\renewcommand{\sE}{\mathcal{E}}
\renewcommand{\sC}{\mathcal{C}}
\renewcommand{\sD}{\mathcal{D}}
\renewcommand{\sI}{\mathcal{I}}
\renewcommand{\sH}{\mathcal{H}}
\renewcommand{\sM}{\mathcal{M}}
\renewcommand{\sG}{\mathcal{G}}
\renewcommand{\sp}{\upbeta}

\usepackage{marvosym}
\newcommand{\eg}{\emph{e.g.}~}
\newcommand{\ie}{\emph{i.e.}~}
\newcommand{\cf}{\emph{cf.}~}

\begin{document}

\title[$F$-signatures, splitting primes, and test modules under finite covers]{On the behavior of $F$-signatures, splitting primes, and test modules under finite covers}

\author[J.~Carvajal-Rojas]{Javier Carvajal-Rojas}
\address{\'Ecole Polytechnique F\'ed\'erale de Lausanne\\ SB MATH CAG\\MA C3 615 (B\^atiment MA)\\ Station 8 \\CH-1015 Lausanne\\Switzerland \newline\indent
Universidad de Costa Rica\\ Escuela de Matem\'atica\\ San Jos\'e 11501\\ Costa Rica}
\email{\href{mailto:javier.carvajalrojas@epfl.ch}{javier.carvajalrojas@epfl.ch}}
\author[A.~St\"abler]{Axel St\"abler}
\address{Universit\"at Leipzig\\ Institut f\"ur Mathematik\\Augustusplatz 10\\
04109 Leipzig\\Germany} 
\email{\href{mailto:staebler@uni-leipzig.de}{staebler@math.uni-leipzig.de}}


\thanks{Carvajal-Rojas was supported in part by the NSF CAREER Grant DMS \#1252860/1501102 and the ERC-STG \#804334. St\"abler was supported in part by SFB-Transregio 45 Bonn-Essen-Mainz financed by Deutsche Forschungsgemeinschaft.}


\begin{abstract}
We give a comprehensive treatment on how $F$-signatures, splitting primes, splitting ratios, and test modules behave under finite covers. To this end, we expand on the notion of transposability along a section of the relative canonical module as first introduced by K.~Schwede and K.~Tucker.
\end{abstract}
\maketitle
%\tableofcontents
\section{Introduction}

This work is concerned with natural invariants attached to \emph{Cartier modules} over rings of positive characteristic. More precisely, it investigates their behavior under finite covers of the base ring. The prototypical Cartier module over an $\bF_p$-algebra $R$ is a pair $(M,\varphi)$ where $M$ is a a finitely generated $R$-module and $\varphi\: M \to M$ is a $p^{-1}$-linear map, \ie $r\varphi(m)=\varphi(r^pm)$ for all $r\in R$ and $m\in M$. For instance, $M$ can be taken to be $R$ and $\varphi$ to be a Frobenius splitting, \ie a splitting of the Frobenius endomorphism $F\: R \to R$ sending $r \mapsto r^p$. Another prominent example is the Cartier operator $\kappa \: \omega_R \to \omega_R$ associated to the canonical module of $R$ of suitable rings $R$. More general Cartier modules are obtained by replacing $\varphi$ by a so-called Cartier algebra of $p^{-1}$-linear maps. See \autoref{Sec.CartierModules} and \autoref{sec.F-Stuff} for further details on Cartier modules and their invariants.
 
 The concept of \emph{$F$-regularity} is central in the study of Cartier modules. To many, the most important open problem in positive characteristic commutative algebra is the $F$-regularity of splinters. Roughly speaking, it is conjectured that $F$-regularity of an $\bF_p$-algebra $R$ is equivalent to the property that all finite covers of $R$ are split, in this case, $R$ is called a splinter. In particular, this poses the problem of describing the behavior of the $F$-regularity of Cartier modules under finite covers of the ground-ring. There have been several occurrences of this in the literature \cite{SchwedeTuckerTestIdealFiniteMaps} being the most prominent work. The same theme lies at the core of \eg \cite{CarvajalSchwedeTuckerEtaleFundFsignature,CarvajalFiniteTorsors,JeffriesSmirnovTransformationRule}.

The main results in the aforementioned papers are formulas (referred to as transformations rules) for the basic objects measuring the $F$-regularity of (certain) Cartier modules under finite maps. There is, however, no unified treatment of the phenomena expressed in those formulas. Further, it is often not clear what the key ingredients and underlying reasons for those formulas to work are. For instance, is normality an essential hypothesis? This is a truly important question as any birational geometer in positive characteristic can tell. In fact, by removing normality from the hypothesis, those formulas can be valuable tools for proving normality when this one is unknown (\eg cyclic covers).

In this work, we provide a formalism to treat those questions in a rather simple way. For instance, the desired transformation rules are exhibited as formal consequences of Grothendieck duality for finite covers. For such formalism to work the best, we need to work in a rather general framework. However, we do it seeking naturality in the proofs rather than generality in the statements. Concretely, we generalize the transformation rules for the $F$-signature in \cite{CarvajalSchwedeTuckerEtaleFundFsignature,CarvajalFiniteTorsors} and the main results in \cite{SchwedeTuckerTestIdealFiniteMaps} regarding the behavior of test ideals under finite covers. These generalizations are achieved by using the formalism of Cartier modules and functors $f^*$, $f^!$ introduced by M.~Blickle and the second named author \cite{BlickleStablerFunctorialTestModules} as well as the formalism of \emph{transposability along sections of the relative canonical module} introduced by Schwede--Tucker \cite{SchwedeTuckerTestIdealFiniteMaps,SchwedeTuckerExplicitlyExtending}. We develop the theory of transposability in \autoref{sec.Transposability} allowing us to drop normality from our hypothesis and express our results in a cleaner way.

Before stating our results, let us provide some background. Let $R$ be an $F$-finite noetherian $\bF_p$-algebra and $\sC$ be a Cartier $R$-algebra. To a Cartier $\sC$-module $M$, we may associate different objects that are useful in understanding its \emph{simplicity}. For example, the \emph{test module} $\uptau(M,\sC) \subset M$ is the smallest among certain Cartier submodules; see \cite{BlickleStablerFunctorialTestModules}. If $R$ is a local domain and $M=R$, its \emph{splitting prime} $\sp(R,\sC) \subset R$ is either nonproper or the unique largest proper Cartier submodule of $R$ and turns out to be a prime ideal. The ideals $\uptau(R,\sC)$, $\sp(R,\sC)$ bound the non-simplicity of $R$ as a Cartier module as follows. If $0 \subsetneq \mathfrak{r} \subsetneq R$ is a Cartier submodule of $R$ then:
\[
\uptau(R,\sC)\subset \mathfrak{r} \subset \sp(R,\sC).
\]
In particular, $R$ is a simple Cartier $\sC$-module if and only if $\uptau(R,\sC)=(1)$, or if and only if $\sp(R,\sC)=(0)$. In fact, if $\sp(R,\sC)$ is proper, the Cartier submodules of $R$ form a finite sub-lattice; in the sense of \cite[8]{JacobsonBasicAlgebraI}, of radical ideals of $R$ which is bounded (\ie it has a greatest and a least element); see \cite{EnescuHochsterTheFrobeniusStructureOfLocalCohomology,SharpGradedAnnihilatorsOfModulesOverTheFrobeniusSkewPolynomialRing,SchwedeFAdjunction,KumarMehtaFiniteness,SchwedeTuckerNumberOfFSplit,HunekeWatanabeUpperBoundMultiplicityFpireRIngs}. Following \cite{SchwedeCentersOfFPurity}, we call the prime ideals in this lattice \emph{centers of $F$-purity for $(R,\sC)$} and $\sp(R,\sC)$ is the \emph{maximal center of $F$-purity} of $(R,\sC)$. These notions have been traditionally studied in tight closure and $F$-singularity theory. For example, $(R,\sC)$ is said to be \emph{$F$-pure} if $\sp(R,\sC) \subset R$ is proper and \emph{$F$-regular} if $R$ is simple as a $\sC$-module. If $R$ is Cohen--Macaulay and $M=\omega_R$, analogous considerations lead to the notions of $F$-injectivity and $F$-rationality. By definition, a general Cartier module $(M,\sC)$ is $F$-regular if $\uptau(M,\sC)=M$ which is a subtler notion of simplicity.

Suppose that $(R,\sC)$ is $F$-pure, so $R/\sp(R,\sC)$ is a simple Cartier module under the induced action of $\sC$. To $(R,\sC)$ we associate a number $r(R,\sC) \in (0,1]$ measuring such $F$-regularity. This number is called the \emph{splitting ratio} of $(R,\sC)$ and the larger it is the ``milder'' the singularities of the $F$-regular pair $(R/\sp(R,\sC),\sC)$ are. When $\sp(R,\sC)=0$, the splitting ratio becomes the $F$-\emph{signature} of $(R,\sC)$ and is denoted by $s(R,\sC)$. 

Let $f\: \Spec S \to \Spec R$ be a finite cover, \ie a dominant finite morphism (between $F$-finite and noetherian schemes). How does the $F$-regularity of $(M,\sC)$ pull-back along $f$? Following \cite{BlickleStablerFunctorialTestModules}, we see that a precise, natural way to formulate this question is: How does the $F$-regularity of $(M,\sC)$ and $(f^!M, f^*\sC)$ compare to each other? See \autoref{sec.Preliminaries} for details on how this Cartier-theoretic pullback is defined. Our main result in this direction is the following:

\begin{theoremA*}[\autoref{theo.TraceTauSurjective}]
Let $f\: \Spec S \to \Spec R$ be a finite cover between $F$-finite noetherian schemes. Let $\sC$ be a Cartier $R$-algebra and $M$ be a Cartier $\sC$-module. Then, 
\begin{equation}
\label{eqn.TransRuleTestModules}
\Tr_{M}\big(f_\ast \uptau (f^! M, f^*\sC ) \big) = \uptau(M,\sC),
\end{equation}
where $\Tr_M: f_*f^{!}M \to M$ is the Grothendieck trace map (evaluation at $1$).
In particular, $\Tr_M$ is surjective if $(M,\sC)$ is $F$-regular. Conversely, if $\Tr_M$ is surjective and $(f^! M, f^*\sC)$ is $F$-regular, then $(M,\sC)$ is $F$-regular.
\end{theoremA*}

%Using Grothendieck duality for $f$, we can answer those in a straightforward fashion:
%\begin{equation} \label{eqn.TransRuleTestModules}
%\Tr_{M}\big(f_\ast \uptau(f^! M, f^*\sC) \big) = \uptau(M,\sC),
%\end{equation}
%where $\Tr_M: f_*f^{!}M \to M$ is the Grothendieck trace map (evaluation at $1$). See \autoref{theo.TraceTauSurjective}.
This has the salient application of telling us why, for an $F$-regular Cartier module $(M,\sC)$, the trace maps $\Tr_M: f_*f^{!}M \to M$ are surjective for all $f$, which for $M=R$ specializes to saying that $R$ is a splinter. See \autoref{subsubsec.RelationshipSplinters} for more on this.

In practice, very often the Cartier module $M$ of interest is divisorial, \ie $M=R(D)$ for some (generalized) divisor $D$ on $\Spec R$ (\eg $D=K_R$ some canonical divisor). In that case, our interest is on $S(f^*D)$ rather than $f^!\big(R(D)\big)=S(f^*D-K_{S/R})$. However, it is then not clear what the right question should be as $S(f^*D)$ has no natural Cartier $f^*\sC$-module structure. Our solution to this is inspired by \cite{SchwedeTuckerTestIdealFiniteMaps}: we consider a natural transformation $f^{\dagger} \to f^!$ where $f^{\dagger}$ is the functor from $R$-modules to $S$-modules obtained by pullback followed by $\mathbf{S}_2$-ification (we require some technical $\mathbf{G}_i + \mathbf{S}_j$ conditions on $R$ and $S$ for this to work). For example, $f^{\dagger}\big(R(D)\big)=S(f^* D)$. The choice of such natural transformation is tantamount to the choice of global section of $\omega_{S/R}$, which is the choice of a map $T \in \Hom_R(S,R)$. Then, we define a Cartier $\sC$-module $M$ to be \emph{$T$-transposable} so that $f^{\dagger} M \to f^! M$ is a morphism of Cartier $f^*\sC$-modules. In that case, under suitable, mild conditions, $f^{\dagger} M \to f^! M$ will induce an isomorphism after applying the test module functor $\uptau$; see \autoref{cor.varSigmatestModules}. See \autoref{sec.Transposability} for the details about transposabilty and $f^{\dagger}$.
In particular, we obtain the following result:
\begin{theoremB*}[{\autoref{cor.SpecializationTransposability}, \cf \cite{SchwedeTuckerTestIdealFiniteMaps}}]
Let $f\: \Spec S \to \Spec R$ be as in Theorem A and further assume that $R$ and $S$ satisfy the $\mathbf{G}_1+\mathbf{S}_2$ condition (\eg normal). Suppose that either $R$ is local or of finite type over a field (so that we can ensure $F^!\omega \cong \omega$). Choose an element $T \in \omega_{S/R}$ such that the $S$-linear map $S \to \omega_{S/R}$ defined by $1 \mapsto T$ is generically an isomorphism. Then, for all almost Cartier divisors $D$ on $\Spec R$ we have:
\begin{equation}
\label{eqn.DivisorialCase}
T\big(f_*\uptau (S(f^\ast D),f^*\sC)\big) = \uptau(R(D),\sC)
\end{equation}
if $R(D)$ is a $T$-transposable Cartier $\sC$-module.
\end{theoremB*}

Theorem B should be thought of as a generalization of the main results in \cite[\S6]{SchwedeTuckerTestIdealFiniteMaps} by Schwede--Tucker; see \autoref{cor.SpecializationTransposability} for the details. However, we have substantially weakened the normality hypothesis; the conditions we assume are the bare minimum for which such transformation rules are possible.

%Then, \autoref{eqn.TransRuleTestModules} yields:
%\begin{equation} \label{eqn.DivisorialCase}
%T\big(f_\ast \uptau(S(f^*D), f^*\sC) \big) = \uptau\big(R(D),\sC\big)
%\end{equation}
%if $R(D)$ is a $T$-transposable Cartier $\sC$-module; see \autoref{cor.SpecializationTransposability} for more.
To know when $R(D)$ is a $T$-transposable, we extend the Schwede--Tucker criterion for $T$-transposability; see \autoref{thm.TransposabilityCriterion}. We explain how the formula \autoref{eqn.DivisorialCase} generalizes those in \cite{SchwedeTuckerTestIdealFiniteMaps} when $D=0$ and $\sC$ is the Cartier algebra of triples; see \autoref{sec.testmodulefinitemorphism}. Likewise, we obtain analogous formulas for \emph{test ideals along closed subschemes} (as treated in  \cite[\S 3.1]{Smolkinphdthesis}, \cite[\S4]{SmolkinSubadditivity}); see \autoref{theo.adjointidealtransformation}. Similarly, \autoref{cor.CohenMacualayCanonicalModulesTrans} contains the special cases for canonical modules.  Analogous results for \emph{non-$F$-pure modules} are established in \autoref{sec.NonFPuremodulesResults}.

In the local case, $F$-regularity can be measured by invariants other than test modules. Our main result in this case is the following theorem. %Let $(R,\fram, \kay, K) \subset (S, \fran, \el, L)$ be a finite local extension (the data $(R,\fram, \kay, K)$ means that $R$ is a local domain with maximal ideal $\fram$, residue field $\kay$, and field of fractions $K$) with corresponding morphism $f \: \Spec S \to \Spec R$. Let $\omega_{S/R} = \Hom_R(S,R)$ be the \emph{relative canonical module}, consider a section $\sigma \: S \to \omega_{S/R}$, and set $T\coloneqq \sigma(1)$. Let $\sC$ be a Cartier $R$-algebra acting on $R$ such that $R$ is a $T$-transposable Cartier $\sC$-module (\eg $\sigma$ is an isomorphism). If $\sigma$ is a generic isomorphism, $T$ is surjective, and $T(\fran) \subset \fram$ then: 
%\begin{equation}
%\sp\big(S,f^* \sC\big)\cap R = %\sp\big(R,\sC\big),
%\end{equation}
%\begin{equation}
%\big[\kappa(\fran):\kappa(\fram)\big] \cdot r(S,f^* \sC ) = \big[\kappa\big(\sp(S,f^* \sC)\big):\kappa\big(\sp(R,\sC)\big)\big] \cdot r(R,\sC),
%\end{equation}
%where $\kappa(-)$ denotes residue fields. In particular, $(R,\sC)$ is $F$-pure (resp. $F$-regular) if and only if so is $(S,f^*\sC)$. In the $F$-regular case:
%\begin{equation}
%[\el:\kay]\cdot s(S,f^* \sC) = [L:K] \cdot s(R,\sC).
%\end{equation}
%See \autoref{thm.TransRuleSplittingRatios} and \autoref{thm.MainTheorem}.

\begin{theoremC*}[{\autoref{thm.MainTheorem} and \autoref{thm.TransRuleSplittingRatios}, \cf \cite[Theorems 3.1, 4.4]{CarvajalSchwedeTuckerEtaleFundFsignature}, \cite[Theorem 4.11]{CarvajalFiniteTorsors}}]
Let $\theta\: (R,\fram) \to (S, \fran)$ be a finite local extension between $F$-finite rings defining a cover $f\: \Spec S \to \Spec R$. Suppose that $R$ is an integral domain with field of fractions $K$, set $L \coloneqq S \otimes_R K$, and write $[L:K] \coloneqq \dim_K L$. Suppose that there is a \emph{generic} isomorphism $\sigma_R\: S \to \omega_{S/R}$ of $S$-modules such that $T\coloneqq\sigma_R(1)$ is surjective and $T(\fran) \subset \fram$. Let $\sC$ be a Cartier $R$-algebra acting on $R$ such that $R$ is a $T$-transposable Cartier $\sC$-module (\eg $\sigma$ is an isomorphism). Then: %If $R$ is a $T$-transposable Cartier $\sC$-moduleLet $\phi\: (R,\fram, \kay) \to (S, \fran, \el)$ be a finite local extension inducing a generically flat cover $f\: \Spec S \to \Spec R$. Let $\omega_{S/R} = \Hom_R(S,R)$ be the relative canonical module, consider a section $\sigma \: S \to \omega_{S/R}$, and set $T\coloneqq \sigma(1)$. Let $\sC$ be a Cartier $R$-algebra acting on $R$ such that $R$ is a $T$-transposable Cartier $\sC$-module (\eg $\sigma$ is an isomorphism). If $\sigma$ is a generic isomorphism, $T$ is surjective, and $T(\fran) \subset \fram$ then: 
\begin{equation}
\sp\big(S,f^* \sC\big)\cap R = \sp\big(R,\sC\big)
\end{equation}
and
\begin{equation}
\big[\kappa(\fran):\kappa(\fram)\big] \cdot r(S,f^* \sC ) = \big[\kappa\big(\sp(S,f^* \sC)\big):\kappa\big(\sp(R,\sC)\big)\big] \cdot r(R,\sC),
\end{equation}
where $\kappa(-)$ denotes residue fields. In particular, $(R,\sC)$ is $F$-pure (resp. $F$-regular) if and only if so is $(S,f^*\sC)$. In the $F$-regular case:
\begin{equation} \label{eqn.TransRuleFSign}
\big[\kappa(\fran):\kappa(\fram)\big]\cdot s(S,f^* \sC) = [L:K] \cdot s(R,\sC).
\end{equation}
\end{theoremC*}
As we explain in \autoref{rem.Comparison}, the transformation rule \autoref{eqn.TransRuleFSign} generalizes those in \cite[Theorems 3.1, 4.4]{CarvajalSchwedeTuckerEtaleFundFsignature}, \cite[Theorem 4.11]{CarvajalFiniteTorsors}, which focus on rather specific setups.

These results illustrate the crucial role transposability plays in the behavior of Cartier modules under finite covers. Bearing this in mind, we revisit Schwede--Tucker's transposability criterion in \autoref{sec.SchwedeTuckerFaithfullyFlat}. There, we study how norm functions can be used to translate the criterion from effectiveness of divisor upstairs to effectiveness of divisors downstairs, which tend to be simpler as illustrated by Noether normalizations of Cohen--Macaulay singularities.

\subsection*{Acknowledgements} 
 We would like to thank Manuel Blickle, Manfred Lehn, Linquan Ma, Zsolt Patakfalvi, Karl Schwede, Anurag Singh, Ilya Smirnov, Daniel Smolkin, and Maciej Zdanowicz for very useful discussions and help throughout the preparation of this preprint. We are very thankful to Karl Schwede for suggesting to us many ways to improve a first version of this work such a removing normality from our hypothesis and working out a transformation rule for adjoint ideals under finite covers. Further, we thank Anne Fayolle for reading through our manuscript making many corrections and providing enlightening insights, especially in our treatment of \autoref{theo.adjointidealtransformation}. The first named author commenced working on this project while in his last year of Graduate School at the Department of Mathematics of the University of Utah. He is greatly thankful for their hospitality and support. He is particularly thankful to his advisor Karl Schwede for his guidance and generous support. Finally, we are grateful to the anonymous referees for their valuable comments and corrections.

\section{Preliminaries} \label{sec.Preliminaries}
We recall the notions of Cartier modules, $F$-signature, splitting primes, splitting ratios, and test modules as we employ them here. However, we recall first some generalities about Hartshornes's theory of generalized divisors \cite[\S 2]{HartshonreGeneralizedDivisorsAndBiliaison}. The main takeaway of \autoref{sec.GeneralizedDivisors} is \autoref{prop.MApsAndDivisors} as well as our conventions (\eg \autoref{con.GeneralCOnventionFFinitenessAndNoetherianity}, \autoref{rem.COndition(!)}, \autoref{term.Sheafification}), and it could be skipped by experts. %\todo{we should probably specify  our conventions with numbers now -- I guess 2.6/2.7}

\begin{convention} \label{con.GeneralCOnventionFFinitenessAndNoetherianity}
All schemes and rings are defined over $\mathbb{F}_p$ and assumed to be noetherian and $F$-finite. We denote the $e$-th iterate of the Frobenius endomorphism by $F^e \: X \to X$ if $X$ is a scheme or by $F^e \: R \to F^e_* R$ if $R$ is a ring. We use the shorthand notation $q \coloneqq p^e$. We refer to a finite dominant morphism as a \emph{finite cover} or simply as a \emph{cover}. We shall denote the category of finitely generated $R$-modules by $R\textnormal{-fmod}$.
\end{convention}

\subsection{Generalized divisors, duality, and canonical modules} \label{sec.GeneralizedDivisors} The following is the version of Grothendieck duality that we use throughout; see \cite[III, Ex. 6.10]{Hartshorne}. 

\begin{proposition}[Grothendieck duality for finite covers] \label{pro.GroDuality}
 Let $f\: \Spec S \to \Spec R$ be a finite cover and $f^!$ be the functor $R\textnormal{-fmod} \to S\textnormal{-fmod}$ given by $f^!=\Hom_R(S, -)$. Then, the morphism of $R$-modules
\[
\xi=\xi(M,N)\: f_* \Hom_S(N,f^!M) \to \Hom_R(f_*N, M), \quad \psi \mapsto \Tr_M \circ f_*\psi
\]
is a natural isomorphism on both $M$ and $N$. Here, $\Tr_M \: f_*f^!M \to M$ is the \emph{trace natural transformation} given by the evaluation-at-$1$ map. The inverse $\zeta=\zeta(M,N)$ is given by $\zeta(\vartheta)(n)\coloneqq\vartheta(- \cdot n)$ for all $n\in N$. In other words, $\big(\zeta(\vartheta)(n)\big)(s) \coloneqq \vartheta(s \cdot n)$ for all $n\in N, s \in S$.
By gluing on affine charts, the same duality applies to a general cover $f\: Y \to X$.
\end{proposition}
\begin{remark} \label{re,.GroDualityII}
The maps $\xi$, $\zeta$ in \autoref{pro.GroDuality} preserve the $S$-linear structures. Thus, we may think of $\xi$ as a natural $S$-linear isomorphism $\xi' \: \Hom_S(N,f^!M) \to \Hom_R(f_*N, M)$,  where the $S$-linear structure of the target is given by scalar pre-multiplication. We shall refer to these natural $S$-linear isomorphisms as Grothendieck duality too. Thus, $\Hom_S(N,f^!M) \to \Hom_R(N,M)$, $\psi \mapsto \Tr_M \circ \psi$, is an isomorphism of abelian groups that preserves both $R$-linear and $S$-linear structures.
\end{remark}

\begin{terminology}
Let $X$ be a scheme. Following Hartshorne, we say that a coherent sheaf $\sF$ on $X$ satisfies Serre's condition $\mathbf{S}_k$ if $\depth \sF_x \geq \min\{k,\dim \sO_{X,x}\}$ for all $x \in X$. We say that $X$ satisfies $\mathbf{S}_k$ if so does $\sO_X$. We say that $X$ satisfies the $\mathbf{G}_l$ condition if $\sO_{X,x}$ is a Gorenstein local ring for all points $x\in X$ of codimension $\leq l$. Naturally, the $\mathbf{G}_l + \mathbf{S}_k$ condition means that $X$ satisfies both $\mathbf{G}_l$ and $\mathbf{S}_k$. The same terminology is applied to the ring/module theoretic setting.
\end{terminology}

\begin{definition}[Canonical modules] \label{def.CanonicalModules}
Let $X$ be a noetherian equidimensional scheme. $X$ \emph{admits a canonical module/sheaf} if there are morphisms $i \circ f\: X \to S \to G$ where $f$ is a finite cover, $i$ is a closed embedding, and $G$ is a Gorenstein scheme. If $\omega_G$ is a fixed dualizing invertible sheaf on $G$, the \emph{canonical module} of $X$ is defined by $\omega_X \coloneqq f^! \omega_S$ where $\omega_S \coloneqq  {\sE}xt_G^{\delta}(\sO_S, \omega_G) = \mathbf{H}^{-\dim S}\big(i^!\omega_G[\dim G]\big)$, and $\delta = \dim G - \dim S$. If $g\:Y \to X$ is a finite cover between schemes admitting canonical modules, we always make the choice of canonical modules compatible with $g$, that is, we always set $\omega_Y \cong g^!\omega_X$; see \eg \autoref{rem.COndition(!)}. 
\end{definition}

\begin{caveat}
In \cite{HartshonreGeneralizedDivisorsAndBiliaison}, the scheme $G$ in \autoref{def.CanonicalModules} is assumed to be regular. However, the results in \cite[\S 1, \S 2]{HartshonreGeneralizedDivisorsAndBiliaison} hold in our more general setup by \cite{HerzogKunzCanonicalModules}.
\end{caveat}

\begin{remark} \label{rem.COndition(!)}
We say that an $F$-finite scheme $X$ satisfies property $(!)$ if it admits a canonical module satisfying
\begin{equation}
\label{eqn.!}
F^! \omega_X \cong \omega_X. \tag{$!$}
\end{equation}
This includes the local case and the (essentially) of finite type case. Indeed, $F^! \omega_X$ is a dualizing sheaf and thus by \cite[Chapter V, Theorem 3.1]{HartshorneResidues} one has $\omega_X \cong F^! \omega_X \otimes \mathcal{L}$ for some invertible sheaf $\mathcal{L}$, which implies the local case. In the essentially of finite type case, say $f\: X \to \Spec \kay$, we may take $\omega_X \coloneqq f^! \kay$ since $\kay \cong F^! \kay$. It is however not always possible to satisfy \eqref{eqn.!} in the general $F$-finite case; see \cite[\S 2.5]{SchwedeTuckerTestIdealFiniteMaps}.
%Let $X$ be an $F$-finite scheme admitting a canonical module. In general, $\omega_X$ cannot be defined so that $F^! \omega_X \cong \omega_X$; see \cite[\S 2.5]{SchwedeTuckerTestIdealFiniteMaps}. When possible, we say that $X$ satisfies $(!)$. This includes the local case and the (essentially) of finite type case. Indeed, $F^! \omega_X$ is a dualizing sheaf and thus by \cite[Chapter V, Theorem 3.1]{HartshorneResidues} one has $\omega_X \cong F^! \omega_X \otimes \mathcal{L}$ for some invertible sheaf $\mathcal{L}$, which implies the local case. In the essentially of finite type case, say $f\: X \to \Spec \kay$, we may take $\omega_X \coloneqq f^! \kay$ since $\kay \cong F^! \kay$.
\end{remark}

\begin{terminology}[\cite{HartshonreGeneralizedDivisorsAndBiliaison}] \label{term.Sheafification}
Let $X$ be a scheme admitting a canonical module $\omega_X$. For a quasi-coherent sheaf $\sF$, one defines $\sF^{\omega} \coloneqq \ssHom_X(\sF,\omega_X)$ and refers to it as its \emph{$\omega$-dual}. Consider the natural $\sO_X$-linear maps $\alpha_{\sF} \: \sF \to \sF^{\omega \omega}$. One refers to $\alpha$ as the \emph{${\omega}$-reflexification} or as the \emph{$\mathbf{S}_2$-ification natural transformation} when $\sF$ satisfies $\mathbf{S}_1$; see \cite[Remark 1.8]{HartshonreGeneralizedDivisorsAndBiliaison}. One says that a coherent $\sO_X$-module $\sF$ is \emph{$\omega$-reflexive} if $\alpha_{\sF}$ is an isomorphism and we denote the corresponding full subcategory by $\sO_X\textnormal{-mod}^{\omega}$. We apply the same terminology to the ring/module theoretic setting.
\end{terminology}

%\todo{Is the upshot of this remark that we can restrict $f^!$ to $\omega$-duals or do we need anything else? If so I'd suggest changing this into a proposition {\color{blue} I think that makes sense!}}
\begin{lemma}
Let $X$ be a scheme satisfying $\mathbf{S}_1$ and admitting a canonical module. If $f\: Y \to X$ is a cover, then $f^!$ restricts to a functor $f^! \: \sO_X\textnormal{-mod}^{\omega} \to \sO_Y\textnormal{-mod}^{\omega}$. 
\end{lemma}
\begin{proof}
By \cite[Lemma 1.3]{HartshonreGeneralizedDivisorsAndBiliaison}, $\omega_X$ satisfies $\mathbf{S}_2$. Further, $\omega$-reflexivity and $\mathbf{S}_2$ are equivalent conditions; % Indeed, $\alpha_\sF \: \sF \to \sF^{\omega \omega}$ is in general an isomorphism in codimension $0$. However, if $\sF$ satisfies $\mathbf{S}_1$, then $\alpha_\sF$ is an isomorphism in codimension $1$ and so injective.
see \cite[Proposition 1.5]{HartshonreGeneralizedDivisorsAndBiliaison}. Moreover, the $\omega$-dual of any coherent sheaf is $\omega$-reflexive \cite[Corollary 1.6]{HartshonreGeneralizedDivisorsAndBiliaison}. In particular, $\ssHom_X(\sF,\sG)$ is $\omega$-reflexive if so is $\sG$ by $\tensor$-Hom adjointness. Additionally, we have natural isomorphisms $f_*(\sG^{\omega}) \cong (f_*\sG)^{\omega}$---provided by Grothendieck duality---for all coherent $\sO_Y$-modules $\sG$.
\end{proof}
%\begin{remark}
%Let $X$ be a scheme satisfying $\mathbf{S}_1$ and admitting a canonical module. Then, $\omega_X$ satisfies $\mathbf{S}_2$ \cite[Lemma 1.3]{HartshonreGeneralizedDivisorsAndBiliaison}.\footnote{We use Hartshorne's definition of Serre's conditions: a coherent sheaf $\sF$ satisfies the $\mathbf{S}_k$ condition if $\depth \sF_x \geq \min\{k,\dim \sO_{X,x}\}$ for all $x \in X$.} Further, $\omega$-reflexivity and $\mathbf{S}_2$ are equivalent conditions -- % Indeed, $\alpha_\sF \: \sF \to \sF^{\omega \omega}$ is in general an isomorphism in codimension $0$. However, if $\sF$ satisfies $\mathbf{S}_1$, then $\alpha_\sF$ is an isomorphism in codimension $1$ and so injective.
%see \cite[Proposition 1.5]{HartshonreGeneralizedDivisorsAndBiliaison}. Moreover, the $\omega$-dual of any coherent sheaf is $\omega$-reflexive \cite[Corollary 1.6]{HartshonreGeneralizedDivisorsAndBiliaison}. In particular, $\ssHom_X(\sF,\sG)$ is $\omega$-reflexive if so is $\sG$ by $\tensor$-Hom adjointness. Additionally, if $f\:Y \to X$ is a cover then we have natural isomorphisms $f_*(\sG^{\omega}) \cong (f_*\sG)^{\omega}$---provided by Grothendieck duality---for all coherent $\sO_Y$-modules $\sG$. Thus, $f^!$ can be restricted to a functor $f^! \: \sO_X\textnormal{-mod}^{\omega} \to \sO_Y\textnormal{-mod}^{\omega}$. 
%The same remarks apply to the ring/module theoretic setting. 
%\end{remark}

Let $X$ be a scheme satisfying the $\mathbf{S}_2$ condition and admitting a canonical module $\omega_X$. In particular, $\sO_X$ is $\omega$-reflexive (\cite[Proposition 1.5]{HartshonreGeneralizedDivisorsAndBiliaison}). Let $\sK(X)$ denote its sheaf of total fractions, which has good properties as $X$ satisfies $\mathbf{S}_1$ \cite[Proposition 2.1]{HartshorneGeneralizedDivisorsOnGorensteinSchemes}. For instance, $\sK(X) = \bigoplus_{\eta} j_* \sO_{X,\eta}$ where the sum runs over the generic points $\eta$ of $X$ and $j\:\{\eta\} \to X$ is the inclusion. In particular, a coherent $\sO_X$-module $\sF$ is supported in codimension $\geq 1$ if and only if $\sK(X) \otimes \sF = 0$, and if and only if $\sF$ is locally annihilated by regular elements of $\sO_X$.

%\todo{I guess we should recall linear equivalence here? I added this without messing any thing up hopefully, some things I used tacitly: $D + (g)$ can be defined without dualizing since $g$ is principal, any nonzero $g \in \sK_X$ automatically is a divisor (I think these are all okay but it can't hurt if you think this through again).}
\begin{definition}
In the situation above we have the following notions:
\begin{enumerate}[(a)]
\item A \emph{fractional ideal} is a coherent subsheaf $\sI \subset \sK(X)$. A fractional ideal $\sI$ is said to be \emph{non-degenerate} if $\sI$ is supported in codimension $0$, \ie $\sI_\eta \neq 0$ for each generic point $\eta \in X$. In particular, $\sI_{\eta} \cong \sO_{X,\eta}$ for all generic points $\eta$ and so $\sK(X) \otimes \sI \cong \sK(X)$.
\item{ A \emph{generalized divisor} on $X$ is a non-degenerate fractional ideal $\sI \subset \sK(X)$ satisfying the $\mathbf{S}_2$ condition. We will follow the usual convention of denoting a generalized divisor by $D$ and refer to $\sI_D$ as the fractional ideal associated to $D$. Further, we denote
\[
\sO_X(D) \coloneqq \sI_D^{\vee} = \ssHom_X(\sI_D, \sO_X) \cong \sI_{D}^{-1} = \sI_{-D},
\]
and $
\omega_X(D) \coloneqq \ssHom_X(\sI_D,\omega_X)$, which Hartshorne denotes by $\sL(D)$ and $\sM(D)$; respectively. A pair of generalized divisors $D,E$ are added as follows:
\[
\sI_{D+E} \coloneqq (\sI_D \cdot \sI_E)^{\omega \omega}.
\]
Following Hartshorne, we define another generalized divisor by \[
(D)(-E) \coloneqq \ssHom_X(\sI_E, \sI_D)^{\omega\omega}.
\]
\item A generalized divisor $D$ is said to be \emph{Cartier} (resp. \emph{principal}) if $\sI_{D}$ is an invertible sheaf (resp. $\sI_{D} \cong \sO_X$). If $D$ is principal and $\sI_D = \sO_X \cdot g$ for some invertible global section $g \in K(X)$, we write $D = \Div_X g = \Div g$.
\item A generalized divisor $D$ is said to be \emph{almost Cartier} if there is a big open $U \subset X$ (\ie $U$ contains every codimension $1$ point of $X$) such that $(\sI_D)|_{U}$ is Cartier.} We denote by $\ADiv X$ the group of almost Cartier divisors on $X$ (under the addition law from (b)). 
\item We say that two divisors $D, E$ are \emph{linearly equivalent} if $\sI_D \cdot g = \sI_E$, where $g \in \sK(X)$ is an invertible global section. We denote the group of almost Cartier divisors modulo linear equivalence by $\APic X$,
\item  Let $f \: Y \to X$ be a cover, so that $Y$ satisfies the $\mathbf{S}_2$ condition and admits a canonical module $\omega_Y \cong f^!\omega_X$. We define the \emph{pullback homomorphism of almost Cartier divisors} \[f^\ast\: \ADiv X \to \ADiv Y\] as follows: Let $D$ be an almost Cartier divisor on $X$. Let $i \: U \to X$ be a big open such that $i^* \sI_D$ is Cartier and $j \: V \to Y$ be the pullback of $i$ along $f$, which is a big open too. Then, we define $f^*D$ by $\sI_{f^*D} \coloneqq j_* f_U^* (i^*\sI_D)$ (where $f_U\: V \to U$ is the pullback of $f$ to $U$), which is an almost Cartier divisor. We note that $f^\ast$ is compatible with linear equivalence. Furthermore, we obtain a homomorphism of abelian groups $f^* \: \APic X \to \APic Y$; see \cite[Definition after Corollary 2.3]{HartshonreGeneralizedDivisorsAndBiliaison}.
\item For a given abelian group $A$ (such as $\bQ$ or $\bZ_{(p)}$), we define an \emph{almost Cartier $A$-divisor} to be an element of $A \otimes_{\bZ} \ADiv X \eqqcolon \ADiv_A X$. We define $\APic_A X$ similarly, so that we have a surjective homomorphism $\ADiv_A X \to \APic_A X$ (obtained by tensoring $\ADiv X \to \APic X$ with $A$).

We say that two almost Cartier $A$-divisors $D_1$, $D_1$ are \emph{$A$-linearly equivalent} and write $D_1 \sim_A D_2$ if the difference $D_1 - D_2$ is in $\ker(\ADiv_A X \to \APic_A X)$. If $A = \bQ$ (resp. $A = \bZ_{(p)}$), $D_1 \sim_{\bQ} D_2$ (resp. $D_1 \sim_{\bZ_{(p)}} D_2$) if and only if there is $n\in \bZ$ nonzero (resp. prime-to-$p$) such that $nD_1 \sim nD_2$ in $\ADiv X$.

The pullback of almost Cartier $A$-divisors $f^* \: \ADiv_A X \to \ADiv_A Y $ is defined by twisting the pullback of almost Cartier divisors by $A$.
\end{enumerate}
\end{definition}

%If $D$ is a generalized divisor on $X$, we denote its corresponding fractional ideal by $\sI_D \subset \sK(X)$. Moreover, we denote
%\[
%\sO_X(D) \coloneqq \sI_D^{\vee} = \ssHom_X(\sI_D, \sO_X) \cong \sI_{D}^{-1} = \sI_{-D},
%\]
%and $\omega_X(D) \coloneqq \ssHom_X(\sI_D,\omega_X)$, which Hartshorne denotes by $\sL(D)$ and $\sM(D)$; respectively.

\begin{lemma}
Let $X$ be a scheme satisfying the $\mathbf{G}_0+\mathbf{S}_2$ condition and admitting a canonical module $\omega_X$.
\begin{enumerate}
\item There is an embedding $\omega_X \to \sK(X)$ realizing $\omega_X$ as a non-degenerate fractional ideal, say $\sI_{-K_X}$, which defines an \emph{anticanonical (generalized) divisor} $-K_X$. More generally, if $D$ is a generalized divisor then $\omega_X(D)$ is also a generalized divisor. Further, if either $X$ is $\mathbf{G}_1$ or $D$ is almost Cartier, we have the following (canonical) relation \begin{equation} \label{eqn.RelationOandOmega}
\sO_X(D) \cong \omega_X\big(D+(-K_X)\big).
\end{equation}
\item Assume now that $X$ also satisfies $\mathbf{G}_1$, then one defines a \emph{canonical divisor} $K_X$ on $X$ by $K_X \coloneqq -(-K_X)$, which is an almost Cartier divisor and $\sO_X(K_X) \cong \omega_X$. Furthermore, $\omega_X(D)$ is a \emph{reflexive} generalized divisor for any  generalized divisor $D$.
\end{enumerate}
\end{lemma}
\begin{proof}
For (a), since $X$ is $\mathbf{G}_0$, the sheaf $\omega_X$ is generically free of rank $1$ (\ie free of rank $1$ at every generic point of $X$). Using \cite[2.4]{HartshonreGeneralizedDivisorsAndBiliaison}, we find an embedding $\omega_X \to \sK(X)$ realizing $\omega_X$ as a non-degenerate fractional ideal; see \cite[Definition--Remark 2.7]{HartshonreGeneralizedDivisorsAndBiliaison}. Note that $\omega_X(D)$ is $\omega$-reflexive and generically free of rank $1$. Hence, by \cite[Proposition 2.8]{HartshorneGeneralizedDivisorsOnGorensteinSchemes}, $\omega_X(D)$ is isomorphic to $\sI_E$ for some generalized divisor $E$, namely $E=(-K_X)(-D)$). Since $X$ satisfies $\mathbf{S}_2$, $\sO_X$ is $\omega$-reflexive; see \cite[Proposition 1.5]{HartshonreGeneralizedDivisorsAndBiliaison}. In other words, $\sO_X$ defines a generalized divisor. Further, $\sO_X(D)$ is $\omega$-reflexive for all generalized divisors. Hence, we have a canonical isomorphism $\sO_X(D) \to \ssHom_X\big(\omega_X, \ssHom(\sI_{-D}, \omega_X)\big)$. However, $\ssHom_X(\sI_{-D}, \omega_X) = \sI_{D + (-K_X)}$ if either $D$ or $-K_X$ is almost Cartier by \cite[Proposition 2.8.(c)]{HartshonreGeneralizedDivisorsAndBiliaison}. Putting these together yields $\sO_X(D) \cong \ssHom_X(\omega_X, \sI_{D+(-K_X)}) = \omega\big(D+(-K_X)\big)$.

%Then, we obtain a canonical isomorphism $\alpha_{\sO_X(D)}\:\sO_X(D) \to  \omega_X\big((-K_X)(-D)\big)$ (given by $\omega$-reflexification). However, $ (-K_X)(D) = D+(-K_X)$ if $D$ or $-K_X$ is almost Cartier by \cite[Proposition 2.8.(c)]{HartshonreGeneralizedDivisorsAndBiliaison}.{\color{blue}Putting these isomorphisms together: \begin{align*}\sO_X(D) \cong   \omega_X\big((-K_X)(-D)\big) \cong \omega_X(\big(D +(-K_X) \big).\end{align*}}
%Then, it follows that:
%\begin{align*}
%\sO_X(D) \cong \ssHom_X\big(\sI_D,\omega_X(-K_X)\big) &\cong \ssHom_X\big((\sI_D \cdot \sI_{-K_X})^{\omega \omega}, \omega_X\big)\\ &= \ssHom_X\big(\sI_{D+(-K_X)}, \omega_X\big)\\&= \omega_X\big(D+(-K_X)\big) 
%\end{align*}
%where the second isomorphism is defined via $\tensor$-$\ssHom$ adjointness and the fact that $\alpha_{\sF} \: \sF \to \sF^{\omega \omega}$ is an isomorphism after $\omega$-dualization for all coherent sheaves $\sF$. Indeed, since $\sF^{\omega}$ is $\omega$-reflexive and so $\mathbf{S}_2$ for all coherent sheaves $\sF$, showing that $(\alpha_{\sF})^{\omega}$ is an isomorphism can be done by localizing at codimension $1$ points. However, $\alpha_{\sF}$ is an isomorphism in codimension $1$

For (b), we have $\sO_X(K_X) \cong \omega_X$ by \autoref{eqn.RelationOandOmega}. Using \cite[Proposition 2.2 (d)]{HartshonreGeneralizedDivisorsAndBiliaison}, we have $\omega_X(D) \cong \sO_X(D+K_X)$, which is reflexive.
\end{proof}

Next proposition establishes the connection between generalized divisors and maps.

\begin{proposition} \label{prop.MApsAndDivisors}
Let $f\:Y \to X$ be a finite cover between schemes satisfying $\mathbf{G}_0+\mathbf{S}_2$. Let $E$ be a reflexive divisor on $Y$ and $D$ be an almost Cartier divisor on $X$. Then, there is a canonical isomorphism of $\sO_X$-modules:
\begin{equation} \label{eqn.CanonicalCorrespondanceMapsDivisorsO_X}
\ssHom_X\big(f_*\omega_Y(E),\omega_X(D)\big) \cong f_*\sO_Y(f^*D-E),
\end{equation}
 which can also be thought of as an isomorphism of $\sO_Y$-modules
\begin{equation} \label{eqn.CanonicalCorrespondanceMapsDivisorsO_Y}
\ssHom_X\big(f_*\omega_Y(E),\omega_X(D)\big) \cong \sO_Y(f^*D-E),
\end{equation}
as in \autoref{re,.GroDualityII}. In particular, to any non-degenerate map $\varphi \: f_*\omega_Y(E) \to \omega_X(D)$ there corresponds an effective generalized divisor $ D_{\varphi} \sim f^*D-E$ (on $Y$).\footnote{The non-degeneracy of $\varphi$ means that the corresponding morphism $\sO_Y \to \ssHom_X\big(f_*\omega_Y(E), \omega_X(D)\big)$ is injective. In other words, $\varphi$, as a global section of the $\sO_Y$-module $\ssHom_X\big(f_*\omega_Y(E), \omega_X(D)\big)$, is supported generically (\ie $\varphi$ is locally a regular element). Recall we assume $Y$ satisfies $\mathbf{S}_1$.}

Further, let $g\: Y' \to Y$ be a finite cover of $X$-schemes satisfying $\mathbf{G}_0+ \mathbf{S}_2$, and set $f' = f \circ g \: Y' \to X$. Let $E'$ be a reflexive generalized divisor on $Y'$, and let $E$ and $D$ be almost Cartier divisors on $Y$ and $X$ respectively. Then the following equality holds for all non-degenerate
$\psi \in \Hom_Y\big(g_*\omega_{Y'}(E'),\omega_Y(E)\big)$ and $\varphi \in \Hom_X\big(f_*\omega_Y(E),\omega_X(D)\big)$:
\begin{equation} \label{eqn.NaturalityMapsAndDivisors}
D_{\varphi \circ f_*\psi} = D_{\psi} + g^* D_{\varphi}.
\end{equation}
\end{proposition}
\begin{proof}
Work in the affine setting. Set $X= \Spec R$, $Y=\Spec S$, $Y' = \Spec S'$. By the projection formula, the canonical morphism $f_* \omega_S \otimes_R \sI_D \to f_*\big(\omega_S(E) \otimes_S f^*\sI_D \big)$ is an isomorphism on a big open $U \subset X$ on which $D$ is Cartier. Hence, its $\omega$-dual is an isomorphism and so:
\begin{align*}
\Hom_R\big(f_*\omega_S(E),\omega_R(D)\big) \cong \Hom_R\big(f_*\omega_S(E) \otimes_R \sI_D,\omega_R\big) &\cong \Hom_R\big(f_*(\omega_S(E) \otimes_S f^*\sI_D),\omega_R\big)
\end{align*}
Next, observe that
\begin{align*}
    \big( \omega_S(E) \otimes_S f^*\sI_D \big)^{\omega \omega}  \cong \sI_{(-K_S)(-E)+f^*D} = \sI_{(-K_S)(-(E-f^*D))} \cong \omega_S(E-f^*D), 
\end{align*}
where the middle equality follows from \cite[Proposition 2.8 (b)]{HartshonreGeneralizedDivisorsAndBiliaison} and that $f^*D$ is almost Cartier. Summing up, we have canonical isomorphisms
\begin{equation} \label{eqn.FirstIsomorphism}
\Hom_R\big(f_*\omega_S(E),\omega_R(D)\big) \cong \Hom_R\big(f_*\omega_S(E-f^*D),\omega_R\big).
\end{equation}
By denoting $E^* \coloneqq E-f^*D$, we may assume $D=0$. Note that $E^*$ is reflexive by \cite[Proposition 2.2 (c), (e)]{HartshonreGeneralizedDivisorsAndBiliaison}. Notice that the isomorphism \autoref{eqn.FirstIsomorphism} is trivial at codimension $1$ points or whenever $\sI_D \cong \O_X$. Thus, we may use Grothendieck duality to conclude:
\begin{equation} \label{eqn.SecondIsomorphism}
\Hom_R\big(f_*\omega_S(E^*), \omega_R\big) \cong f_* \Hom_S\big(\omega_S(E^*), \omega_S\big) \cong f_*\sI_{E^*} = f_*\sI_{-(-E^*)} \cong f_*S(-E^*).
\end{equation}

Note that a non-degenerate element of $S(- E^*)$ is an embedding $\sI_{-E^*} \to S$ and so an effective divisor linearly equivalent to $-E^*$. Thus, we obtain a mapping $\varphi \mapsto  D_{\varphi} \in |f^*D-E|$.

We show now \autoref{eqn.NaturalityMapsAndDivisors}, \ie the naturality of $\varphi \mapsto D_{\varphi}$. As explained above, we may assume $D=0$---this will simplify the notation. Let us consider the following commutative triangle.
\[
\xymatrix{
f'_* \omega_{S'}(E') \ar[rr]^-{f_* \psi} \ar[dr]_-{\varphi'} & & f_*\omega_S(E) \ar[dl]^-{\varphi} \\
& \omega_R &
}
\]
By naturality (on the source) of the first isomorphism in \autoref{eqn.SecondIsomorphism}, this diagram becomes:
\[
\xymatrix{
g_* \omega_{S'}(E') \ar[rr]^-{\psi} \ar[dr]_-{\varphi''} & & \omega_S(E) \ar[dl]^-{x \in \sI_E} \\
& \omega_S &
}
\]
By naturality of \autoref{eqn.SecondIsomorphism} (on the target) and \autoref{eqn.FirstIsomorphism}, we obtain the commutative diagram
\[
\xymatrix{
\omega_{S'}(E'-g^*E)  \ar[dr]_-{y \in \sI_{E'-g^*E}} & & \omega_{S'}(E') \ar[dl]^-{y' \in \sI_{E'}} \ar[ll]_-{x} \\
& \omega_{S'} &
}
\]
Writing this down additively yields the claimed divisorial equality.
\end{proof}

\begin{example}
Work in the setup of \autoref{prop.MApsAndDivisors}. Setting $D,E=0$, \autoref{eqn.CanonicalCorrespondanceMapsDivisorsO_Y} means that $\Tr_{\omega_R} \: f_* \omega_Y \to \omega_X$ is a free generator of $\ssHom_X(f_*\omega_Y, \omega_X)$ as an $\sO_Y$-module. Assuming $X$ satisfies $\mathbf{G}_1$ and setting $D=-K_R$ and $E=-K_S$ in \autoref{eqn.RelationOandOmega}, then $\omega_{Y/X} \cong \sO_Y\big(f^*(-K_X)-(-K_Y)\big)$ where $\omega_{Y/X} \coloneqq f^! \sO_X$. Set $K_{Y/X} \coloneqq f^*(-K_X)-(-K_Y)$ to be the \emph{relative canonical divisor of $f$}. Thus, an injective global section $\sigma \:\sO_Y \to \omega_{Y/X}$ (\ie a non-degenerate map $T \: f_* \sO_Y \to \sO_X$) defines an effective divisor $D_T \sim K_{Y/X}$. We often write $\Ram_T = D_T$.
\end{example}

\begin{example}[Schwede's correspondence] \label{ex.SchwedeCorrespondence}
Let $X $ be a scheme satisfying the conditions $\mathbf{G}_0+\mathbf{S}_2$ and \eqref{eqn.!}; see \autoref{rem.COndition(!)}. Applying \autoref{prop.MApsAndDivisors} to $F^e \:X \to X$ yields:
\[
\ssHom_X\big(F^e_*\omega_X(D), \omega_X(D)\big) \cong F^e_*\sO_X\big((q-1)D\big)
\]
for all almost Cartier divisors $D$ on $X$. Equivalently,
\begin{equation} \label{eqn.SchwedeCorrespondence}
\ssHom_X\big(F^e_*\omega_X(D), \omega_X(D)\big) \cong \sO_X\big((q-1)D\big)
\end{equation}
where the $\sO_X$-linear structure on the left-hand side module is given by pre-multiplication. These isomorphisms can be described via \emph{Cartier operators}. Indeed, let $\kappa_X^e \: F_*^e \omega_X \to \omega_X$ be the map defined by \autoref{eqn.SchwedeCorrespondence} when $D=0$. We refer to it as the \emph{$e$-th Cartier operator} of $X$. We will see below that $\kappa_X^e$ is the $e$-th power of $\kappa_X^1 \eqqcolon \kappa_X$. Using the projection formula, we may twist $\kappa_X^e$ by an almost Cartier divisor $D$ to obtain a map $\kappa_{X,D}^e \: F^e_* \omega_X(qD) \to \omega_X(D)$. For instance, $\kappa_{X,-K_X}^e$ is the so-called Frobenius trace. Locally, a section $x$ of $\sO_X\big((q-1)D\big)$ defines a map $\sO_X(D) \to \sO_X(qD)$ and further a map $\omega_X(D) \to \omega_X(qD)$ which when composed with $\kappa^e_{X,D}$ gives $\kappa_{X,D} ^e \cdot x \: F^e_* \omega_X(D) \to F^e_* \omega_X(qD) \to \omega_X(D)$. Thus, to any non-degenerate map $\varphi \: F^e_* \omega_X (D) \to \omega_X(D)$ we associate an effective almost Cartier divisor $D_{\varphi} \sim (q-1)D$. One defines the \emph{normalized divisor of $\varphi$} by $\Delta_{\varphi} \coloneqq \frac{1}{q-1} D_{\varphi}$, which is an effective almost Cartier $\bZ_{(p)}$-divisor such that $\Delta_{\varphi} \sim_{\bZ_{(p)}} D$. Let us look at the case $\sO_X(D)=\omega_R\big(D+(-K_X)\big)$. Then, to any non-degenerate map $\varphi \: F^e_*\sO_X(D) \to \sO_X(D)$ there corresponds an effective almost Cartier $\bZ_{(p)}$-divisor $\Delta_{\varphi} \sim_{\bZ_{(p)}} D+(-K_X)$. If $X$ also satisfies $\mathbf{G}_1$, then $\Delta_{\varphi}+K_X \sim_{\bZ_{(p)}} D$. We then recover Schwede's correspondence \cite{SchwedeFAdjunction}. Concretely, if $\Delta$ is an effective almost Cartier $\bZ_{(p)}$-divisor such that $\Delta + K_X \sim_{\bZ_{(p)}} D$, then there is a non-degenerate map $\varphi \: F^e_* \sO_X(D) \to \sO_X(D)$ for some $e>0$ such that $\Delta_{\varphi}=\Delta$ and $\varphi$ is unique up to pre-multiplication by global units and by powers of $\varphi$ (where $\varphi^n \:F^{en}_* \sO_X(D) \to \sO_X(D)$ is defined inductively by $\varphi^{n+1} = \varphi \circ F^e_* \varphi^n$). In particular, $\Delta_{\varphi^n} = \Delta_{\varphi}$ for all $n$ by  \autoref{eqn.NaturalityMapsAndDivisors}. We denote the minimal such $e>0$ as $e_{\Delta}$. Thus, Schwede's correspondence holds in this generality (using \cite[Proposition 2.9]{HartshorneGeneralizedDivisorsOnGorensteinSchemes}, \cite[Remark 2.9]{HartshonreGeneralizedDivisorsAndBiliaison}, and $\bZ_{(p)} = \colim_e \bZ_{q-1}$) by the arguments in \cite{SchwedeFAdjunction}; see \cite[\S4]{BlickleSchwedeLinearMapsAlgebraAndGeometry}. \end{example}

\subsection{Cartier modules and the functors $f^*$, $f^!$} \label{Sec.CartierModules} Let us define Cartier algebras and modules as well as the functors $f^*$, $f^!$ introduced in \cite{BlickleStablerFunctorialTestModules}, which is the formalism we need to express our transformation rules naturally. For simplicity, we will work in the affine setting yet everything generalizes to the general scheme-theoretic setting by gluing on affine charts.
\begin{definition}[Cartier algebras] \label{def.CartierAlgebras}
Let $R$ be a ring. A \emph{Cartier $R$-algebra} $\sC$ is an $\bN$-graded ring $\sC = \bigoplus_e \sC_e$ with $\sC_0=R$, and equipped with a graded $R$-bimodule structure such that $r \cdot \kappa_e = \kappa_e \cdot r^q$ for all $r\in R$, $\kappa_e \in \sC_e$. Cartier $R$-algebras form a category in the obvious way.
\end{definition}

\begin{example} \label{fullCartieralgebras}
Let $M$ be an $R$-module. Its \emph{full Cartier algebra} is $\sC_M \coloneqq \bigoplus_e \Hom_R(F_*^e M, M)$ (setting  $\sC_{0,M} \coloneqq R$). Write $\sC_{e,M} = \Hom_R(F_*^e M, M)$. The (graded) left $R$-module structure is given by post-multiplication: $(r \cdot \varphi_e)(-) \coloneqq r \cdot \varphi_e(-)$ for all $r \in R, \varphi_e \in \sC_{e,M}$. The (graded) right $R$-module structure is given by pre-multiplication: $(\varphi_e \cdot r)(-) \coloneqq \varphi_e(F^e_*r \cdot -)$ for all $r \in R, \varphi_e \in \sC_{e,M}$ The (graded) ring structure is given by: if $\varphi_e \in \sC_{e,M}$ and $\varphi \in \sC_{d,M}$, $\varphi_e \cdot \varphi_d \in \sC_{e+d, M}$ is the composition $\varphi_e \circ F^e_* \varphi_d \: F^{e+d}_{*} M \to F^e_* M \to M$. 

Given map $\varphi \: F^{e}_* M \to M$, we write $\sC_M^{\varphi} \subset \sC_M$ for the Cartier subalgebra of $\sC_M$ generated by $\varphi$. That is, $\sC_{ne, M}^{\varphi} = \varphi^n \cdot R$ whereas $\sC_{d,M}^{\varphi}=0$ if $e \nmid d$.

Let $X = \Spec R$ be as in \autoref{ex.SchwedeCorrespondence} and $M = R(D)$ for some almost Cartier divisor $D$ on $X$. Then, $
\sC_{R(D)} \cong \bigoplus_e R\big((q-1)(D-K_X)\big)$ and if we treat elements of $R\big((q-1)(D-K_X)\big)$ as invertible elements $x_e \in \sK(R)$ such that $\Div x_e + (q-1)(D-K_X) \geq 0$, then $x_e \cdot x_d \coloneqq x_e^{p^d}  x_d$ by \autoref{eqn.NaturalityMapsAndDivisors}. Further, such an element $x_e$ acts on $x \in R(D)$ as $x_e \cdot x \coloneqq \kappa_{X,D-K_X}^e(F^e_*x_e x) \in R(D)$. 
\end{example}

\begin{example}
\label{exa.Cartieralgebradivisor}
Let $X = \Spec R$ be as in \autoref{ex.SchwedeCorrespondence} and $D$ be an almost Cartier divisor. Fixing an effective almost Cartier $\bQ$-divisor $\Delta$, we define $\sC_{\omega_R(D)}^{\Delta}$ to be the Cartier subalgebra of $\sC_{\omega_R(D)}$ consisting of degree-$e$ homogeneous elements $\varphi$ such that $\Delta_{\varphi}- \Delta \geq 0$, \ie $D_{\varphi} \geq \lceil (q-1) \Delta \rceil$. Thus (assuming $R$ further satisfies $\mathbf{G}_1$),
\[
\sC_{e,R(D)}^{\Delta} = \Hom_R\bigl(F^e_* R\bigl(D+\lceil (q-1)\Delta\rceil\bigr), R(D) \bigr) \subset \Hom_R(F^e_*R(D),R(D)),
\]
as in \cite[\S 4.3.1]{BlickleSchwedeTuckerFSigPairs1}. One also defines $\sC_{\omega_R(D)}^{\Delta, \mathfrak{a}^t}$ as customary.  When $\Delta = \Delta_{\varphi}$ for some $\varphi \: F^{e_{\Delta}}_* R(D) \to R(D)$ (which is unique up to pre-multiplication by units), we have $\sC_{R(D),ee_{\Delta}}^{\Delta}=\sC_{R(D),ee_{\Delta}}^{\varphi}$ for all $e \in \bN$. As in the last paragraph of \autoref{fullCartieralgebras}, \[
\sC_{R(D)}^{\Delta} \cong \bigoplus_{e \in \bN} R\big((q-1)(D-K_X) - \lceil (q-1) \Delta \rceil\big).\] 
\end{example}
\begin{definition}[Cartier modules] \label{def.CartierModules}
Work in the setup of \autoref{def.CartierAlgebras}. A \emph{Cartier $\sC$-module} is a left $\sC$-module that is finite as an $R$-module, \ie it is a module $M$ in $R\textnormal{-fmod}$ together with a homomorphism of Cartier $R$-algebras $\Xi \: \sC \to \sC_M$. The category of Cartier modules is defined in the obvious way.
\end{definition}

Next, we define the functors $f^*$ and $f^!$ associated to a cover $f \: \Spec S \to \Spec R$. Let $\sC$ be a Cartier $R$-algebra. One defines a Cartier $S$-algebra $f^* \sC$ as follows. As a right $S$-module, $f^* \sC$ is equal to $\sC \otimes_ R S = \bigoplus_e {\sC_e \otimes_R S}$, which is a graded ring in the obvious way. The graded left $S$-module structure is defined by the rule: $s \cdot (\kappa_e \otimes s') \coloneqq \kappa_e \otimes s^q s'$ for all $s,s' \in S, \kappa_e \in \sC_e$. This defines a functor $f^*$ from Cartier $R$-algebras to Cartier $S$-algebras.

For the ``upper-shriek'' functor, consider $f^!=\Hom_R(S, -) \: R\textnormal{-fmod} \to S\textnormal{-fmod}$. It is extended to a functor from Cartier $\sC$-modules to Cartier $f^* \sC$-modules as follows. Let $M$ be a Cartier $\sC$-module and $\mu \in f^!M =  \Hom_R(S,M)$, $\kappa_e \otimes s' \in f^* \sC _e$. Define $(\kappa_e \otimes s') \cdot \mu \in f^! M$ as: $\big((\kappa_e \otimes s') \cdot \mu\big)(s) \coloneqq \kappa_e \cdot \mu(s' s^q)$. Alternatively, given a homomorphism $\Xi \: \sC \to \sC_M$, define a natural homomorphism $f^!\Xi \: f^*\sC \to \sC_{f^!M}$ by setting in degree $e$:
\[
f^! \Xi (\kappa_e \otimes s') \coloneqq \big(\Xi(\kappa_e)\big)^{!} \cdot s' \textnormal{ for all } \kappa_e \in \sC_e, s' \in S,
\]
where, for a given $\varphi \in \Hom_R(F^e_*M,M)$, $\varphi^{!}$ is defined by Grothendieck duality as the only element of $\Hom_R(F_*^e f^! M, f^!M)$ making the following diagram commutative
\[
\xymatrix@C=3em{
f_\ast F_*^e f^! M  \ar[r]^-{ f_\ast \varphi^{!} } \ar[d]_-{ F_*^e \Tr_M} & f_\ast f^! M \ar[d]^-{\Tr_M}\\ 
 F_*^e M \ar[r]^-{\varphi} & M
}
\]
That is, $\varphi^{!}(F_*^e \mu)(s)= \varphi\big(F_*^e \mu(s^q) \big) \text{ for all } \mu \in f^!M, s\in S$.

\begin{remark} \label{rem.RemarkRelGor} The following observations are in order: \begin{enumerate}
    \item  If $N \subset f^!M$ is a Cartier $f^*\sC$-submodule, $\Tr_M(f_* N) \subset M$ is a Cartier $\sC$-submodule.
    \item  If $\Xi \: \sC \to  \sC_M$ is an inclusion, the image of $f^! \Xi \: f^*\sC \to \sC_{f^!M}$ is the right $S$-span of $\{\varphi^{!} \mid  \varphi \in \sC\}$ in $\sC_{f^!M}$. In fact, we have an isomorphism $f^* \sC_M^{\varphi} \xrightarrow{\sim} \sC_{f^! M}^{\varphi^{!}}$ for all non-degenerate $\varphi \: F^e_* M \to M$ (\ie $\varphi \cdot r \neq 0$ for all regular elements $r \in R$).
    \item If $\sC \subset \sC_R$ and $S \to \omega_{S/R}\coloneqq f^! R$; $1 \mapsto T$, is an isomorphism, Grothendieck duality yields natural isomorphisms on $N$: $f_* \Hom_S(N,S) \to \Hom_R(f_*N, R)$, $ \psi \mapsto T \circ \psi$. Plugging in $N=F_*^e S$ gives that, for every $\varphi \in \sC_R$, there is a unique $\varphi^{\top} \in \sC_S$ such that
    \[
\xymatrix@C=3em{
F_*^e S  \ar[r]^-{\varphi^{\top}} \ar[d]_-{ F_*^e T} & S \ar[d]^-{T}\\ 
 F_*^e R \ar[r]^-{\varphi} & R
}
\]
is commutative. Hence, \emph{the image of $f^* \sC_R \to \sC_S$ is the right $S$-span of $\{\varphi^{\top} \mid \varphi \in \sC\}$ in $\sC_S$ which we shall often denote by $f^* \sC_R$ by abuse of notation}.
\end{enumerate}

\end{remark}

\subsection{$F$-signatures, splitting primes, and test modules} \label{sec.F-Stuff}
For the reader's convenience, we recall the notions $F$-signature, splitting ratio, splitting prime, and test module.

\subsubsection{Test modules, $F$-purity, and $F$-regularity}

\begin{definition}
Let $(M,\sC)$ be a Cartier module over a ring $R$ and $\sC_+ \coloneqq \bigoplus_{e \geq 1} \sC_e$.
\begin{enumerate}[(a)]
\item A morphism $\Phi\colon N \to M$ of Cartier modules is a \emph{nil-isomorphism} if $\ker \Phi$ and $\coker \Phi$ are annihilated by some power of $\sC_+$.
\item The \emph{test module} $\uptau(M, \sC)$ is defined as the smallest Cartier submodule $N$ of $M$ for which $H^0_\eta(N_\eta) \to H^0_\eta(M_\eta)$ is a nil-isomorphism for all $\eta \in \Ass M$. $M$ is \emph{$F$-regular} if $\uptau(M, \sC) = M$. See \cite[\S 1]{BlickleStablerFunctorialTestModules} for more details.
\item If $R$ is a Cohen--Macaulay ring satisfying \eqref{eqn.!} with Cartier operator $\kappa_R \: F^e_* \omega_R \to \omega_R$, we say that $R$ is $F$-rational (resp. $F$-injective) if $(\omega_R,\kappa_R)$ is $F$-regular (resp. $F$-pure); where $(\omega_R,\kappa_R)$ is the Cartier module $\omega_R$ with respect to its full Cartier algebra (as $\kappa_R \cdot R = \Hom_R(F_\ast \omega_R, \omega_R)$). Equivalently, there are no non-zero proper submodules $M \subset \omega_R$ stable under $\kappa_R$; \cf. \cite[\S8.1, 8.2]{SchwedeTuckerTestIdealSurvey}, \cf \cite{SmithFRatImpliesRat,FedderFPureRat,FedderWatanabe, VelezOpennessOfTheFRationalLocus}.
\end{enumerate}
\end{definition}
We also note that $F$-rationality can be defined without a Cohen--Macaulay assumption but we do not use this here.

%Let $(M,\sC)$ be a Cartier module over a ring $R$ and $\sC_+ \coloneqq \bigoplus_{e \geq 1} \sC_e$. $M$ is called \emph{$F$-pure} if $\sC_+ M = M$. A morphism $\Phi\colon N \to M$ of Cartier modules is a \emph{nil-isomorphism} if $\ker \Phi$ and $\coker \Phi$ are annihilated by some power of $\sC_+$. We define the \emph{test module} $\uptau(M, \sC)$ as the smallest Cartier submodule $N$ of $M$ for which $H^0_\eta(N_\eta) \to H^0_\eta(M_\eta)$ is a nil-isomorphism for all $\eta \in \Ass M$. $M$ is \emph{$F$-regular} if $\uptau(M, \sC) = M$. See \cite[\S 1]{BlickleStablerFunctorialTestModules} for more details.  If $R$ is a Cohen--Macaulay ring satisfying $(!)$ with Cartier operator $\kappa_R \: F^e_* \omega_R \to \omega_R$, we say that $R$ is $F$-rational (resp. $F$-injective) if $(\omega_R,\kappa_R)$ is $F$-regular (resp. $F$-pure); where $(\omega_R,\kappa_R)$ is the Cartier module $\omega_R$ with respect to its full Cartier algebra (as $\kappa_R \cdot R = \Hom_R(F_\ast \omega_R, \omega_R)$). Equivalently, there are no non-zero proper submodules $M \subset \omega_R$ stable under $\kappa_R$; \cf. \cite[\S8.1, 8.2]{SchwedeTuckerTestIdealSurvey}, \cf \cite{SmithFRatImpliesRat,FedderFPureRat,FedderWatanabe, VelezOpennessOfTheFRationalLocus}. We also note that $F$-rationality can be defined without a Cohen--Macaulay assumption but we do not use this here.

\subsubsection{Splitting primes and ratios}

Let $(R,\fram,\kay)$ be a local ring and $\sC$ be a Cartier $R$-algebra acting on $R$. Following \cite{BlickleSchwedeTuckerFSigPairs1}, we define the \emph{$e$-th splitting number} of $(R,\sC)$ to be
\[
a_e(R,\sC) \coloneqq \length_R \big(\sC_{e} \big/\sC_e^{\mathrm{ns}} \big),
\]
where lengths are computed as left $R$-modules, and $\fram \cdot \sC_e \subset \sC_e^{\mathrm{ns}} \subset \sC_e $ is the submodule of $\sC_e$ given by $\sC_e^{\mathrm{ns}} \coloneqq \{\kappa \in \sC_e \mid \kappa \cdot  R \subset \fram \}$. Since $\sC_e^{\mathrm{ns}}$ contains $\ker(\sC_e \to \sC_{e,R})$ and surjectivity/nonsurjectivity is preserved under this map, splitting numbers can be computed by replacing $\sC$ for its image along $\sC \to \sC_R$.  The $F$\emph{-signature} of $(R,\sC)$ is defined as
\[
s(R,\sC) \coloneqq \lim_{e \rightarrow \infty} a_{e\cdot n}/q^{\delta},
\]
where $n\coloneqq \gcd\{e \in \bN \mid a_e(R,\sC) \neq 0 \}$ and $\delta = \dim R + \log_p [\kay^{1/p}:\kay]$. This limit exists (\cite{TuckerFSigExists}) and its positivity characterizes the $F$-regularity of $(R,\sC)$; see \cite[\S 3]{BlickleSchwedeTuckerFSigPairs1}.

Following \cite{AberbachEnescuStructureOfFPure}, we define the \emph{splitting prime} to be the ideal
\[
\sp (R,\sC)\coloneqq\{r\in R \mid \kappa \cdot r \in \fram \text{ for all } e\in \mathbb{N}, \kappa \in \sC_e\}.
\]
Observe that $(R,\sC)$ is $F$-pure if and only if $\upbeta(R,\sC)$ is proper. In that case, $\upbeta(R,\sC)$ is a prime ideal and is zero if and only if $(R,\sC)$ is $F$-regular; see \cite[Theorem 1.1]{AberbachEnescuStructureOfFPure}, \cite[Proposition 2.12, Lemma 2.13]{BlickleSchwedeTuckerFSigPairs1}. Further, $\sp(R,\sC) \subset R$ is the largest proper Cartier $\sC$-submodule  \cite[Remark 4.4]{SchwedeCentersOfFPurity}, \cite[Proposition 2.12]{BlickleSchwedeTuckerFSigPairs1}. Thus, a map $\varphi \in \Image(\sC_e \to \sC_{e,R})$ induces a unique map $\overline{\varphi} \in \sC_{e,R/\sp(\sC)}$ making the following diagram commutative:
\[
\xymatrix{
F_*^e R  \ar[r]^-{ \varphi } \ar@{->>}[d]_-{} & R \ar@{->>}[d]^-{}\\ 
 F_*^e\big( R/\sp(\sC) \big)\ar[r]^-{\overline{\varphi}} & R/\sp(\sC)
}
\]
That is, $\sC$ induces a Cartier $R\big/\sp(\sC)$-algebra, say $\overline{\sC}$; see \cite[Definition 2.10]{BlickleSchwedeTuckerFSigPairs1}. It follows that $\bigl(R/\sp(R,\sC),\overline{\sC} \bigr)$ is $F$-regular; see \cite[Theorem 4.7]{AberbachEnescuStructureOfFPure}, \cite[Corollary 7.8]{SchwedeCentersOfFPurity}, \cite[Lemma 2.13]{BlickleSchwedeTuckerFSigPairs1}, and one defines the \emph{splitting ratio} of $(R,\sC)$ to be
\[
r(R,\sC)\coloneqq s\big(R/\sp(\sC),\overline{\sC}\big).
\]
See \cite[Theorem 4.2]{BlickleSchwedeTuckerFSigPairs1}, \cf \cite[Theorem 4.9]{TuckerFSigExists}, \cite{AberbachEnescuStructureOfFPure}. 

The following is well-known to experts but the key to recover previous results from ours.

\begin{lemma} \label{lem.RecoveringAllStuuf}
Let $\sC \subset \sD$ be Cartier algebras over a ring $R$ and $M$ be a Cartier $\sD$-module. Suppose that there is an $M$-regular element $c \in R$ such that $\sD_e \cdot c \subset \sC_e \subset \sD_e$ for all $e>0$. Then, $\uptau(M,\sC) \supset \uptau(M,\sD)$ is an equality. If $R$ is local and $M=R$, then $s(R,\sC)=s(R,\sD)$.
\end{lemma}
\begin{proof} By assumption, $c$ is not contained in any $\eta \in \Ass M$ and so $\sC_\eta = \sD_\eta$. We have an inclusion $\uptau(M, \sC) \subset \uptau(M, \sD)$ and an inclusion $\sD_+ c \uptau(M, \sC) \subset \uptau(M, \sC)$ where the former is a $\sD$-module. Hence, 
\[
H^0_\eta\big(\sD_+ c\uptau(M, \sC)\big)_\eta =
H^0_\eta\big(\uptau(M, \sC)_\eta\big) = H^0_\eta\big(\uptau(M_\eta, \sC_\eta)\big) = H^0_\eta \big(\uptau(M_\eta, \sD_\eta)\big) = H^0_\eta\big(\uptau(M, \sD)_\eta\big) 
\]
using \cite[Proposition 1.19 (b), Remark 2.4]{BlickleStablerFunctorialTestModules}. By definition, $\uptau(M, \sD) \subset M$ is the smallest $\sD$-submodule such that $H^0_\eta(\uptau(M, \sD)_\eta) \subset H^0_\eta(M_\eta)$ is a nil-isomorphism for all $\eta$. Then, $\uptau(M, \sD) \subset \sD_+ c \uptau(M, \sC)$ which yields the other inclusion. The claim about $F$-signatures is shown as in \cite[Lemma 4.17]{BlickleSchwedeTuckerFSigPairs1}, \cite[\S2.2]{CarvajalSchwedeTuckerEtaleFundFsignature}.
\end{proof}

\section{Transposability along a section of the relative canonical module} Let us consider the following observation. \label{sec.Transposability}

\begin{lemma} \label{lem.TransformationsSections}
Let $f \: Y \to X$ be a finite cover and $\Hom(f^*,f^!)$ be the set of natural transformations $\sigma \: f^* \to f^!$. Then, the mapping $\sigma \mapsto (\sigma_{\sO_X} \: \sO_Y \to \omega_{Y/X})$ defines a bijection $\Hom(f^*,f^!) \to \Gamma(Y,\omega_{Y/X})$. Further, splittings of $\sO_X \to f_* \sO_Y$ in $\omega_{Y/X}$ correspond to natural tranformations $\sigma \: f^* \to f^!$ such that the composition of natural transformations
\[
\id \xrightarrow{\eta} f_* \circ f^* \xrightarrow{f_* \circ \sigma} f_* \circ f^! \xrightarrow{\Tr} \id
\]
is the identity, 
where $\eta$ is the unit functor of the adjointness $(f^*,f_*)$.
\end{lemma}
\begin{proof}
Work in the affine case and set $X = \Spec R$ and $Y=\Spec S$. The mapping $\omega_{S/R} \to \Hom(f^*,f^!)$ sending $S \to \omega_{S/R}$ to
\[
  M \mapsto (\omega_{S/R} \otimes_R M \xrightarrow{\mathrm{can}} f^!M) \circ \big( (S \to \omega_{S/R}) \otimes_R M \big)
\] 
is the required inverse. More succinctly, if $T \in \omega_{S/R}$, then one defines a natural transformation $\sigma : f^* \to f^!$ by declaring $\sigma_M \: f^*M \to f^!M$ to be the $S$-linear map adjoint to the $R$-linear map $M \to f_* f^!M$ given by $m \mapsto (s \mapsto T(s)m)$ for all $s\in S, m \in M$, and all $R$-modules $M$. In other words, $\sigma_M$ sends $s \otimes m$ to the $R$-linear map $(T(s\cdot -)m\: S \to M$ given by $s'\mapsto T(ss')m$.

We readily see that $\omega_{S/R} \to \Hom(f^*,f^!) \to \omega_{S/R}$ is the identity. To demonstrate that $\Hom(f^*,f^!) \to \omega_{S/R} \to \Hom(f^*,f^!)$ is the identity, we use that a natural transformation $\sigma \: f^* \to f^!$ is compatible with the maps in $\Hom_R(R,M)$ to write down commutative squares
\[
\xymatrix@C=3em{
S  \ar[r]^-{ \sigma_R  } \ar[d]_-{ s \mapsto s \otimes m} & \omega_{S/R} \ar[d]^-{\sigma \mapsto \big(s \mapsto \sigma(s)m \big) }\\ 
 f^*M \ar[r]^-{\sigma_M} & f^!M
}
\]
for all $m\in M$. Equivalently, $\sigma_M = (\omega_{S/R}\otimes_R M \to f^!M) \circ (\sigma_R
\otimes_R M)$. The last statement follows directly from the definition of the bijections.
\end{proof}

Following \cite{SchwedeTuckerTestIdealFiniteMaps}, given a finite cover $f\: Y \to X$, we choose a natural transformation $\sigma \: f^* \to f^!$, and set $T\coloneqq \Tr_{X} \circ \sigma_{X}$ to be the corresponding global section of $\omega_{Y/X}$. Before considering the general case, we illustrate with a key example what our goal is. Set $X = \Spec R$ and $Y=\Spec S$. Suppose that $\sigma_R \: S \to \omega_{S/R}$ is injective (\ie $T$ is non-degenerate). We say that $\varphi \in \sC_{e,R}$ is $T$-transposable if $\varphi^{!} \circ F^e_* \sigma_R \: F^e_* S \to \omega_{S/R}$ belongs to the image of $\Hom_S(F^e_*S, S)$ under the embedding $\Hom_S(F^e_* S, \sigma_R) \: \Hom_S(F^e_*S, S) \to \Hom_S(F^e_*S, \omega_{S/R})$. That is, $\varphi$ is \emph{$T$-transposable} if there is a \emph{necessarily unique} map $\varphi^{\top} \in \sC_{e,S}$ such that
\[
\xymatrix@C=3em{
F^e_*S  \ar[r]^-{ \varphi^{\top}  } \ar[d]_-{ F^e_*\sigma_R } & S \ar[d]^-{\sigma_R }\\ 
 F_\ast^e\omega_{S/R} \ar[r]^-{\varphi^{!} } & \omega_{S/R}.
}
\]
is commutative. We refer to $\varphi^{\top}$ as the \emph{$T$-transpose} of $\varphi$ (suppressing $T$ from the notation hoping the context will make it clear). Note that $\varphi^{\top}$ is characterized by the equality
\begin{equation} \label{eqn.TransposabilityEquation}
\varphi\big(F^e_* T(ss'^q)\big) = T\big(\varphi^\top(F^e_*s)s'\big)
\end{equation}
for all $s,s' \in S$. In particular, $\varphi \circ F^e_*T = T \circ \varphi^{\top}$; obtained by setting $s'=1$ in \autoref{eqn.TransposabilityEquation}. However, $T\big(\varphi^\top(F^e_*s)s'\big) = 
T\big(\varphi^\top(s'F^e_*s)\big) =T\big(\varphi^\top(F^e_*ss'^q)\big) $ and so \autoref{eqn.TransposabilityEquation} holds if so does $\varphi \circ F^e_*T = T \circ \varphi^{\top}$. In other words, $\varphi^{\top}$ is characterized by the equality
\begin{equation} \label{eqn.TransposabilitySimpleEquation}
\varphi \circ F^e_*T = T \circ \varphi^{\top}.
\end{equation}
We thank Anne Fayolle for pointing this out to us.
\begin{proposition} \label{prop.SigmaRisoOnTestModules}
With notation as above, suppose that $\sigma_R$ is an isomorphism in codimension $0$ and that $S$ satisfies $\mathbf{S}_1$. Then, the injective morphism of Cartier modules $\sigma_R \: (S,\varphi^{\top}) \to (\omega_{S/R}, \varphi^{!}) $ induces an isomorphism on test modules.
\end{proposition}
\begin{proof}
See \autoref{cor.varSigmatestModules} below.
\end{proof}

If $S$ satisfies $\mathbf{S}_1$ then so does $R$. The converse holds if $\sigma_R$ is injective. Moreover, if $\sigma_R$ is an isomorphism in codimension $0$ and $S$ satisfies $\mathbf{S}_1$, then $\sigma_R$ is injective. Thus, in view of \autoref{prop.SigmaRisoOnTestModules}, the natural setup for transposability is that both $R$ and $S$ satisfy $\mathbf{S}_1$ and that $\sigma_R \: S \to \omega_{S/R}$ is an injective generic isomorphism. If $R$ and $S$ also satisfy $\mathbf{G}_0$, then $\sigma_R$ being injective implies that it is an isomorphism in codimension $0$.

We aim to generalize \autoref{prop.SigmaRisoOnTestModules} and the notions behind it to all Cartier modules. The first hurdle we face is controlling $\ker \sigma_M$ for arbitrary $M$. To bypass this issue, we restrict ourselves to the full subcategory of $\omega$-reflexive modules, which is well-behaved assuming the underlying spaces satisfy $\mathbf{S}_1$; see \cite{HartshonreGeneralizedDivisorsAndBiliaison,HartshorneGeneralizedDivisorsOnGorensteinSchemes}. We introduce the following setup.

\begin{setup} \label{set.Transposability}
Let $f \: Y=\Spec S \to X=\Spec R$ be a finite cover of $\mathbf{S}_1$ schemes admitting canonical modules. Choose $T \in \omega_{S/R}$ and consider the corresponding natural transformation $\sigma \: f^* \to f^!$. We assume that $\sigma_R$ is a generic isomorphism. Let $K \coloneqq \prod_{\height \p = 0} R_{\p}$ be the total ring of fractions of $R$; see \cite[Proposition 2.1]{HartshorneGeneralizedDivisorsOnGorensteinSchemes}. Since $\phi\:R \to S$ is integral, $L\coloneqq K \otimes_R S$ is the total ring of fractions of $S$. In particular, proving something generically (or in codimension $0$) means pulling back along $\Spec K \to X$ and replacing $f$ by $f_K \: \Spec L \to \Spec K$. When $f$ is generically flat (\ie $f_K$ is flat), we write $[L:K]$ for the free rank of $L/K$.
\end{setup}

\begin{definition-proposition} \label{def.DaggerFunctor}
Work in \autoref{set.Transposability}. Then, $\sigma \: f^* \to f^!$ factors through $\omega$-reflexifications. That is, $\sigma_M : f^*M \to f^!M$ factors naturally as
\[
\sigma_M \: f^*M \xrightarrow{\alpha_{f^*M}} (f^*M)^{\omega \omega} \xrightarrow{\varsigma_M} f^!M
\]
for all $M$ in $R\textnormal{-mod}^{\omega}$. Thus, $M \mapsto (f^*M)^{\omega \omega}$ defines a functor $f^{\dagger} \:R\textnormal{-mod}^{\omega} \to S\textnormal{-mod}^{\omega}$ together with a natural transformation $\varsigma \: f^{\dagger} \to f^!$ factoring $\sigma \: f^* \to f^!$ through $\omega$-reflexifications.
\end{definition-proposition}

The advantage of working with $\varsigma$ instead of $\sigma$ is the following. 

\begin{proposition} \label{pro.InjectivitySigma} 
Work in \autoref{set.Transposability}. Then, $\sigma_M^{\omega \omega}$ is an (injective) generic isomorphism for all generically flat finite $R$-modules $M$ satisfying $\mathbf{S}_1$. In particular, $\varsigma_M \: f^{\dagger}M \to f^! M$ is an (injective) generic isomorphism for all generically flat modules $M$ in $R\textnormal{-mod}^{\omega}$.
\end{proposition}
\begin{proof}
For $M$ in $R\textnormal{-fmod}$, we have natural isomorphisms of $R$-modules:
\begin{align*}
f_*f^!(M^{\omega \omega}) = \Hom_R(S,M^{\omega \omega}) \cong \Hom_R(S \otimes_R M^{\omega}, \omega_R) &= \Hom_R(f_*f^*(M^{\omega}), \omega_R) \\
&\cong f_*\Hom_S\big(f^*(M^{\omega}),\omega_S\big)\\
&= f_*\big(f^*(M^{\omega})\big)^{\omega} ,
\end{align*}
given by $\otimes$-Hom adjointness and duality, and giving natural isomorphisms $f^!(M^{\omega \omega}) \cong \big(f^*(M^{\omega})\big)^{\omega}$ of $S$-modules. Thus, applying $f^!$ to $\alpha_M \: M \to M^{\omega \omega}$ yields a transformation $\beta_M \: f^! M \to \big(f^*(M^{\omega})\big)^{\omega}$, which is an isomorphism (resp.\ an isomorphism in codimension $1$) if $M$ is $\omega$-reflexive (resp. satisfies $\mathbf{S}_1$). Note that $\alpha_M$ and $\beta_M$ are injective if $M$ satisfies $\mathbf{S}_1$. 

Consider the diagram:
\begin{equation} \label{eqn.CommDiagramRealizngReflixificationNaturally}
\xymatrix@C=3em{
f^! M \ar[r]^-{\beta_M} \ar[d]_-{\alpha_{f^!M}}  & \big(f^* (M^{\omega}) \big)^{\omega} \ar[d]^-{\alpha} \\
(f^!M)^{\omega \omega} \ar[r]^-{\beta_M^{\omega \omega}} \ar[ru]^-{\delta_M^\omega}  & \Big( \big(f^* (M^{\omega}) \big)^{\omega} \Big)^{\omega \omega}
}
\end{equation}
where $\delta_M^\omega$ is the natural transformation given by
\[
\delta_M^\omega \coloneqq \big(f^*(M^{\omega}) \xrightarrow{\delta_M} (f^!M)^{\omega} \big)^{\omega}
\]
and $\delta_M$ is the $(f^*,f_*)$-adjoint of the $R$-linear map $M^{\omega} \to f_*(f^!M)^{\omega}$ given by $\mu \mapsto f^!\mu$.\footnote{Since $f_*(f^!M)^{\omega} \cong (f_*f^!M)^{\omega}$ by Grothendieck duality, $\delta_M$ is the $(f^*,f_*)$-adjoint of $\Tr_M^{\omega} \: M^{\omega} \to (f_*f^!M)^{\omega}$.} Of course, the outer rectangle in \autoref{eqn.CommDiagramRealizngReflixificationNaturally} is commutative for all $M$. A straightforward computation verifies that the upper triangle in \autoref{eqn.CommDiagramRealizngReflixificationNaturally} is commutative as well (for all $M$). The lower triangle, however, is not necessarily commutative unless $\alpha_{f^!M}$ is surjective. Thus, if $M$ and so $f^!M$ satisfy $\mathbf{S}_1$,\footnote{Use \cite[\href{https://stacks.math.columbia.edu/tag/0AV6}{Lemma 0AV6}]{stacks-project} and the fact that depth is invariant under pushforwards along finite morphisms.} then $\alpha_{f^!M}$ is an isomorphism in codimension $1$ thereby the lower triangle commutes in codimension $1$. Therefore, since all modules in the lower triangle are reflexive, this triangle is commutative if $M$ satisfies $\mathbf{S}_1$; see \cite[Remark 1.8]{HartshonreGeneralizedDivisorsAndBiliaison}. Moreover, in that case, it would be a commutative diagram of natural isomorphisms.

In conclusion, $\delta_M^\omega$ \emph{naturally} realizes $\beta_M$ as the $\mathbf{S}_2$-ification of $f^!M$ in the (full) subcategory of modules $M$ that satisfy $\mathbf{S}_1$. This will be crucial below.

On the other hand, specializing to $M=R$ gives us an injective map:
\begin{equation} \label{eqn.SigmaOmegaR}
S \xrightarrow{ \sigma_R} \omega_{S/R} \xrightarrow{\beta_R} (f^* \omega_R)^{\omega} = \Hom_{S}(f^*\omega_R, \omega_S).  
\end{equation}
Note that $(\beta_R \circ \sigma_R)(1)=\beta_R(T)=\sigma_{\omega_R}$.
Then, we may consider the following composition:
\begin{equation} \label{eqn.DefinitionOfGamma}
   f^*\Hom_R(M,\omega_R) \xrightarrow{c} \Hom_S(f^*M,f^*\omega_R) \xrightarrow{\Hom_S(f^*M, \sigma_{\omega_R})} \Hom_S(f^*M, \omega_S) 
\end{equation}
where $c$ is the $(f^*,f_*)$-adjoint of the $R$-linear map $
\Hom_R(M,\omega_R) \to f_*\Hom_S(f^*M,f^*\omega_R)$ given by $\mu \mapsto f^*\mu$ for all $\mu \in \Hom_R(M,\omega_R)$. This
gives us a natural transformation $\gamma_M \: f^*(M^{\omega}) \to (f^*M)^{\omega}$. Dualizing $\gamma_M$ gives $\gamma_M^{\omega} \: f^{\dagger} M \to \big(f^* (M^{\omega}) \big)^{\omega}$. Further, consider the following diagram:
\[
\xymatrix@C=3em{
f^{\dagger} M \ar[r]^-{\gamma_M^{\omega}}  & \big(f^* (M^{\omega}) \big)^{\omega} \\
f^* M \ar[r]^-{\sigma_M} \ar[u]^-{\alpha_{f^*M}} & f^!M \ar[u]_-{\beta_M}
}
\]
We claim that it is commutative. This is the case for $M = R$ by construction. Noting that all natural transformations are compatible with $\Hom_R(R, M)$, we obtain the commutativity for general $M$. Next, we prove that this diagram (naturally) realizes $\gamma_M^{\omega}$ as $\sigma_M^{\omega \omega}$ via $\delta_M^\omega$. 
\begin{claim}
$\gamma_M^\omega = \delta_M^\omega \circ \sigma_M^{\omega \omega}$ for all $M$. In fact, $\gamma_M = \sigma_M^{\omega} \circ \delta_M$ for all $M$.
\end{claim}
\begin{proof}[Proof of claim]
We verify the second equality which implies the first one. Both sides are natural transformations $f^*(M^{\omega}) \to (f^*M)^{\omega}$. Thus, it suffices to check that both sides agree at an arbitrary element $1\otimes \mu \in f^*(M^{\omega})$ with $\mu \in M^{\omega}$. On one side,  $\gamma_M \: 1 \otimes \mu \mapsto \sigma_{\omega_R} \circ f^*\mu$. On the other, $\sigma_M^{\omega} \circ \delta_M \: 1 \otimes \mu \mapsto f^!\mu \circ \sigma_M$. Thus, we just need to show $\sigma_{\omega_R} \circ f^* \mu = f^!\mu \circ \sigma_M$ for all $\mu$. However, this is clear as both send $1 \otimes m$ to $s \mapsto T(s)\mu(m)$.
\end{proof}

\emph{In particular, $\sigma_M^{\omega \omega} = \delta_M^{\omega -1} \circ \gamma_M^\omega$ if $M$ satisfies $\mathbf{S}_1$, and $\varsigma_M = \beta_M^{-1} \circ \gamma_M^\omega$ if $M$ satisfies $\mathbf{S}_2$.} Consequently, in either case, we just need to prove that $\gamma_M^{\omega}$ is an injective generic isomorphism \emph{for all $M$}. Since the domain of $\gamma_M^{\omega}$ satisfies $\mathbf{S_1}$, it suffices to prove that $\gamma_M^{\omega}$ is a generic isomorphism as then its injectivity is automatic. Then, it is enough to prove that $\gamma_M$ is generically an isomorphism. To this end, recall that $\gamma_M$ is the composition \autoref{eqn.DefinitionOfGamma}. Note that the map $c$ in \autoref{eqn.DefinitionOfGamma} is generically an isomorphism if $M$ is generically flat (or if $f$ is generically flat). Hence, it suffices to prove that $\sigma_{\omega_R}$ is generically an isomorphism. To this end, recall that $\sigma_{\omega_R}$ is the image of $1$ under the composition \autoref{eqn.SigmaOmegaR},
which is generically an isomorphism (by hypothesis). Generically, we then have an isomorphism $L \to \Hom_L(f_K^*\omega_K, \omega_L)$ of $L$-modules. Taking $\omega$-duals, we obtain an isomorphism $\rho \: f^*_K \omega_K \to \omega_L$ using that in dimension $0$ all modules are $\omega$-reflexive \cite[Lemma 1.1]{HartshonreGeneralizedDivisorsAndBiliaison}. Therefore, there is (a unique) $l \in L$ such that $\rho = l \cdot \sigma_{\omega_K}$ and $l$ cannot be a zerodivisor as $\rho$ is an isomorphism, whence $\sigma_{\omega_K} = l^{-1} \cdot \rho$ is an isomorphism; as desired.
\end{proof}

\begin{remark}
In \autoref{pro.InjectivitySigma}, we may add the hypothesis of $f$ being generically flat and remove the one of generic flatness on $M$ (as pointed out in its proof). Such condition is determined by $T$ as $\sigma_R$ is a generic isomorphism by assumption. Indeed, if $(R,\fram, \kay)\subset (S,\fran, \el)$ is a local extension of artinian rings such that $S \to \omega_{S/R}$; $1 \mapsto T$, is an isomorphism, then $S/R$ is free if and only if the inclusion of ideals $\fram S \subset f^! \fram  : T = \{s\in S \mid T(sS) \subset \fram\}$ is an equality. To see this, note that $\dim_{\kay} S/\fram S$ computes the minimal number of generators of $S$ as an $R$-module whereas $\dim_{\kay} S/(f^!\fram : T)$ computes its free rank as $1 \mapsto T$ is an isomorphism; \cf \cite[Lemma 3.6]{BlickleSchwedeTuckerFSigPairs1}. This is all irrelevant if $R$ satisfies $\mathbf{R}_0$.
\end{remark}

\begin{definition}[Transposability] \label{def.transposability}
Working in \autoref{set.Transposability}, let $M$ be a generically flat module in $R\textnormal{-mod}^{\omega}$. Consider the following definitions:
\begin{itemize}
    \item  A map $\varphi \in \sC_{e,M}$ is said to be \emph{$T$-transposable} if $\varphi^{!} \circ F^e_* \varsigma_M \in \Hom_S(F^e_*f^{\dagger} M , f^!M)$ belongs to $\Hom_S(F^e_*f^{\dagger}M, f^{\dagger}M)$, \ie if there is $\psi \in \sC_{e,f^{\dagger}M}$ such that the diagram
\[
\xymatrix@C=3em{
F_*^e f^{\dagger}M  \ar[r]^-{ \psi } \ar[d]_-{ F_*^e \varsigma_M} & f^{\dagger} M \ar[d]^-{\varsigma_M}\\ 
 F_*^e f^!M \ar[r]^-{\varphi^{!}} & f^!M
}
\]
is commutative. Since $\varsigma_M$ is injective, any such $\psi$ is unique and we denote it by $\varphi^{\top}$ (suppressing $T$ from the notation hoping that it is clear from the context) and refer to it as the \emph{$T$-transpose} of $\varphi$.  We denote the set of $T$-transposable maps by $\sC_M^{\top}$, which is a Cartier subalgebra of $\sC_M$.

\item Let $\sC$ be a Cartier $R$-algebra acting on $M$ via $\Xi \: \sC \to \sC_M$. We say that $M$ is a \emph{$T$-transposable} Cartier module if $\Xi$ factors through $\sC^{\top}_M \subset \sC_M$. In particular, $\varsigma_M \: f^{\dagger} M \to f^! M$ is a morphism of Cartier $f^*\sC$-modules where $f^*\sC$ acts on $f^{\dagger} M$ via $\varphi \otimes s \mapsto \varphi^{\top} \cdot s$.  

\item The above definitions do not apply to $M=R$ unless $R$ satisfies $\mathbf{S}_2$, which is only assumed to be $\mathbf{S}_1$. We define the same transposability notions for $M=R$ by replacing $M$ by $R$ and $\varsigma_M$ by $\sigma_R$ in the previous two items---this generalizes the preliminary definitions we gave in the paragraphs that followed \autoref{lem.TransformationsSections}. Of course, if $R$ and $S$ satisfy $\mathbf{S}_2$, these two notions of transposability coincide for $M=R$.
\end{itemize}

\end{definition}

%\begin{remark}(The case $M=R$) \label{rem.TheCaseM=R}According to \autoref{def.transposability}, a $T$-transposable Cartier module is necessarily $\omega$-reflexive. Thus, for this definition to apply to $M=R$ we need $R$ to satisfy $\mathbf{S}_2$, which we find very restrictive at the moment. Thus, we define the same transposability notions for $M=R$ satisfying only $\mathbf{S}_1$ by replacing $M$ by $R$ and $\varsigma_M$ by $\sigma_R$ in \autoref{def.transposability}---this generalizes the preliminary definitions we gave in the paragraphs that followed \autoref{lem.TransformationsSections}. These two notions coincide for $M=R$ when $R$ and $S$ satisfy $\mathbf{S}_2$. Also see \autoref{remTestMosulesoFS1Modules}.\end{remark}
\begin{remark}(Degeneracy) \label{rem.Degeneracy}
Work in the setup of \autoref{def.transposability}. Let $\varphi \: F^e_* M \to M$ be a $T$-transposable map. If any map in the set $\big\{\varphi, \varphi^{!}, \varphi^{\top}\big\}$ is nondegenerate then so are the other two. Indeed, $\varphi^{!}$ and $\varphi^{\top}$ are generically the same map up to isomorphism (of both source and target) as $\varsigma_M$ is generically an isomorphism by \autoref{pro.InjectivitySigma}. To see why $\varphi$ and $\varphi^{!}$ share degeneracy, observe that if $\varphi=0$ then $\varphi^{!}=0$ in general and so generically. Next, suppose $\varphi^{!}$ is generically zero. Since $M$ is generically flat (so free), we may assume that $M=R$. Thus, to show $\varphi$ is generically zero, we show that $\Tr_R \:f_* \omega_{S/R} \to R$ is generically surjective. This follows because $R \to S$ is a finite cover which implies $\omega_{S/R}$ is generically nonzero and so there is a map in $\omega_{S/R}$ whose image contains a regular element. However, it is not clear to the authors whether degenerate maps are $T$-transposable, unless $R$ and $S$ are integral domains. 
\end{remark}

\begin{remark} \label{remTestMosulesoFS1Modules}
By definition, a $T$-transposable Cartier module $M \neq R$ is $\omega$-reflexive. However, we may also work with $\mathbf{S}_1$-modules (\eg $M=R$) when studying their test modules as $\alpha_M \: M \to M^{\omega \omega}$ then realizes $M$ as a Cartier submodule of $M^{\omega \omega}$. Indeed, if $\sC$ is acts on $M$ via $\Xi \: \sC \to \sC_{M}$, then $\sC$ acts on $M^{\omega \omega}$ via $\Xi^{\omega \omega} \: \kappa_e \mapsto \Xi(\kappa_e)^{\omega \omega}$ using the canonical isomorphism $(F^e_* M)^{\omega \omega} \cong F^e_* (M^{\omega \omega})$. Since $\alpha_M$ is a generic isomorphism, $\uptau(\alpha_M)$ is an isomorphism. Thus, in studying test modules, we may replace $M$ with its $\mathbf{S}_2$-ification and consider transposability on $M^{\omega \omega}$.
\end{remark}

\begin{corollary} \label{cor.varSigmatestModules}
Work in \autoref{set.Transposability}. Let $\sC$ be a Cartier $R$-algebra and $M$ be an $\omega$-reflexive $T$-transposable Cartier $\sC$-module. Then $\varsigma_M \: f^{\dagger} M \to f^! M$ induces an isomorphism on test modules, \ie $\uptau(\varsigma_M) \: \uptau(f^{\dagger}M, f^*\sC) \to \uptau (f^!M, f^*\sC)$ is an isomorphism. Further, the same holds with $R$ in place of $M$ and $\sigma_M$ in place of $\varsigma_M$.
\end{corollary}
\begin{proof}
Suppose first that $M$ is $\omega$-reflexive. Observe that $\varsigma_M$ is an injective generic isomorphism between modules having no embedded primes. Since $\varsigma_M$ is an injective $f^\ast\sC$-linear map, $\varsigma_M(\uptau(f^\dagger M, f^\ast\sC))$ is a $f^\ast\sC$-submodule of $f^!M$ that generically agrees with $f^!M$. Thus, $\uptau(f^!M, f^\ast\sC) \subset \varsigma_M\big(\uptau(f^\dagger M, f^\ast\sC)\big)$ by minimality of $\uptau(f^!M, f^\ast\sC)$. Conversely, since $\varsigma_M(f^\dagger M)$ is a $f^\ast\sC$-submodule of $f^! M$ that generically agrees with $f^!M$, it must contain $\uptau(f^!M, f^\ast\sC)$. Then, $\varsigma_M\big(\uptau(f^\dagger M, f^\ast\sC)\big) \subset \uptau(f^!M, f^\ast\sC)$ by minimality of $\uptau(f^\dagger M, f^\ast\sC)$. 

If $M=R$ is not necessarily $\omega$-reflexive, the same proof applies with $\sigma_M$ instead of $\varsigma_M$. Alternatively, one may apply \autoref{remTestMosulesoFS1Modules}.
\end{proof}

\begin{remark} \label{rem.Interpreation}
Working in \autoref{set.Transposability}, assume that $R$ and $S$ satisfy $\mathbf{G}_0$. Let $\sC$ be a Cartier $R$-algebra acting on $\omega_R(D)$ with $D$ an almost Cartier divisor. Notice that the isomorphism $\beta_{\omega_R(D)} \: f^! \omega_R(D) \to f^*(\omega_R(D)^{\omega})^{\omega}$ (as in the proof of \autoref{pro.InjectivitySigma}) yields a natural isomorphism $f^!\omega_R(D) \to \omega_S(f^*D)$. Further, assuming that $R$ and $S$ satisfy $\mathbf{G}_1$, we have:
\[
f^{\dagger} \omega_R(D) =(f^* \sI_{-K_R - D})^{\omega \omega} = \sI_{f^*(-K_R) -f^*D} = \sI_{K_{S/R}-f^*D -K_S} = \omega_S(f^*D - K_{S/R}). 
\]
Moreover, the map $\varsigma_{\omega_R(D)} \: f^{\dagger} \omega_R(D) \to f^{!} \omega_R(D)$ is the embedding $\omega_S(f^*D-K_{S/R}) \to \omega_S(f^*D)$ given by $\Ram_T \sim K_{S/R}$---this is how the map $\chi_{\omega_R(D)}=\gamma_{\omega_R(D)}^{\omega}$ was constructed in the proof of \autoref{pro.InjectivitySigma}. In particular, $f^{\dagger}R(D) \cong S(f^*D)$ whereas $f^!R(D) \cong S(f^*D + K_{S/R})$ and $\varsigma_{R(D)}$ corresponds to the embedding $S(f^*D) \to S(f^*D+K_{S/R})$ defined by $\Ram_T \sim K_{S/R}$.
\end{remark}

\begin{theorem}[Schwede--Tucker's transposability criterion] \label{thm.TransposabilityCriterion}
Working in \autoref{set.Transposability}, assume that $R$ and $S$ satisfy $\mathbf{G}_1+\mathbf{S}_2$ and \eqref{eqn.!}.\footnote{It is enough to assume $R$ satisfies \eqref{eqn.!} as then so does $S$, for $F^!\omega_S = F^! f^! \omega_R = f^! F^! \omega_R \cong f^! \omega_R = \omega_S$.}
Let $D$ be an almost Cartier divisor on $X$. Then, a nondegenerate map $\varphi \: F^e_* \omega_R(D) \to \omega_R(D)$ is $T$-transposable if and only if $f^*\Delta_{\varphi}-\Ram_T$ is effective, in which case $\Delta_{\varphi^{\top}} = f^*\Delta_{\varphi}-\Ram_T$.
\end{theorem}
\begin{proof}
By \autoref{rem.Degeneracy}, if $\varphi$ is nondegenerate then so are $\varphi^{!}$ and $\varphi^{\top}$. By \autoref{eqn.NaturalityMapsAndDivisors}, $D_{\varphi^{!}} = f^*D_{\varphi}$ and so $\Delta_{\varphi^{!}} = f^*\Delta_{\varphi}$. By \autoref{eqn.NaturalityMapsAndDivisors} again, the divisor associated to
\[
\vartheta \: F^e_* \omega_S(f^*D - K_{S/R}) \to F^e_* \omega_S(f^*D) \xrightarrow{\varphi^{!}} \omega_S(f^*D)
\]
is $D_{\varphi^{!}} + \Ram_T$. However, $D_{\varphi^{\top}} + q \Ram_T$ is the divisor associated to a composition
\[
F^e_* \omega_S(f^*D - K_{S/R}) \xrightarrow{\varphi^{\top}} \omega_S(f^*D - K_{S/R}) \to \omega_S(f^*D)
\]
Therefore, $\vartheta$ restricts to $\omega_S(f^*D-K_{S/R})$ if and only if $D_{\varphi^{!}} -(q-1)\Ram_T$ is effective, \ie if $\Delta_{\varphi^{!}}-\Ram_T$ is effective. In that case, $\Delta_{\varphi^{\top}} = \Delta_{\varphi^{!}}-\Ram_T$; as required.
\end{proof}

\begin{proposition} \label{pro.RecoveringAllStuff}
Work in the setup of \autoref{thm.TransposabilityCriterion} and suppose that $f$ is generically flat. Let $\Delta$ be an effective almost Cartier $\bQ$-divisor on $X$ and choose $\mathfrak{a} \subset R$, $t \in \bR_{\geq 0}$. Suppose that $\Delta^* \coloneqq f^* \Delta - \Ram_T \geq 0$ and let $\Xi \: f^*\sC_{R(D)}^{\Delta, \mathfrak{a}^t} \to \sC_{S(f^*D)}^{\Delta^*, (\mathfrak{a} \cdot S)^t}$ be the homomorphism defined by \autoref{thm.TransposabilityCriterion}. Then, there exists a regular element $c \in S$ such that
\[
\sC_{S(f^*D)}^{\Delta^*, (\mathfrak{a} \cdot S)^t} \cdot c \subset  \Xi \Big(f^*\sC_{R(D)}^{\Delta, \mathfrak{a}^t} \Big) \subset \sC_{S(f^*D)}^{\Delta^*, (\mathfrak{a} \cdot S)^t}.
\]
In particular, \[\uptau\Big(S(f^*D), f^*\sC_{R(D)}^{\Delta, \mathfrak{a}^t} \Big) = \uptau\Big(S(f^*D),\sC_{S(f^*D)}^{\Delta^*, (\mathfrak{a} \cdot S)^t}\Big) \coloneqq \uptau\big(S(f^*D), \Delta^*, (\mathfrak{a}\cdot S)^t\big)\]
Likewise, if $f$ is local and $D=0$, then
\[
s\Big(S,f^*\sC_{R}^{\Delta, \mathfrak{a}^t}\Big) = s\big(S, \Delta^*, (\mathfrak{a}\cdot S)^t\big).
\]
\end{proposition}
\begin{proof}
We may assume $\mathfrak{a}=0$ as the general case readily follows from this case. The hypothesis $\Delta^*\geq 0$ implies that every map $\varphi \: F^e_* R(D)  \to R(D)$ with $\Delta_{\varphi} \geq \Delta$ satisfies $f^* \Delta_{\varphi} - \Ram_T \geq 0$, which means that $R(D)$ is a $T$-transposable Cartier $\sC_{R(D)}^{\Delta}$-module by \autoref{thm.TransposabilityCriterion}. Moreover, the structural homomorphism $f^*\sC_{R(D)}^{\Delta} \to \sC_{S(f^*D)}$; given by $\varphi \otimes s \mapsto \varphi^{\top} \cdot s$, factors through the inclusion $\sC_{S(f^*D)}^{\Delta^*} \subset \sC_{S(f^*D)}$. This defines the homomorphism $\Xi \: f^*\sC_{R(D)}^{\Delta} \to \sC_{S(f^*D)}^{\Delta^*}$.\footnote{$\Xi$ is an isomorphism if $\Delta = \Delta_{\varphi}$ for some $\varphi \in \sC_{e,R(D)}$ as then both $f^*\sC_{R(D)}^{\Delta}$ and $\sC_{S(f^*D)}^{\Delta^*}$ are $\sC^{\varphi^{\top}}_{S(f^*D)}$}. Writing $\sC_{R(D)}^{\Delta} \cong \bigoplus_e R\big((q-1)(D-K_X)-\lceil (q-1) \Delta \rceil \big)$ via Cartier operators (see \autoref{exa.Cartieralgebradivisor} and \autoref{fullCartieralgebras}), the $\mathbf{S}_2$-ification of $f^*\sC_{e,R(D)}^{\Delta}$ corresponds to
\begin{align*}
  & {}S\big((q-1)(f^*D-f^*K_X)-f^* \lceil(q-1) \Delta \rceil\big) \\
  = & {}S\Big((q-1)(f^*D-K_Y)-\big(f^* \lceil(q-1) \Delta \rceil-(q-1)K_{Y/X}\big)\Big) 
\end{align*}
Further, by the naturality of Cartier operators, $\Xi_e$ corresponds to
\begin{align*} 
 f^*\sC_{e,R(D)}^{\Delta} \xrightarrow{\alpha_e} {}& S\Big((q-1)(f^*D-K_Y)-\big(f^* \lceil(q-1) \Delta \rceil-(q-1)K_{Y/X}\big)\Big) \\
  \xrightarrow{\cong \: \Ram_T \sim K_{Y/X}}{} & S\Big((q-1)(f^*D-K_Y)-\big(f^* \lceil(q-1) \Delta \rceil-(q-1)\Ram_T\big)\Big) \\
  \subset{} & S\Big((q-1)(f^*D-K_Y)-\big( \lceil(q-1) f^*\Delta \rceil-(q-1)\Ram_T\big)\Big) \\
  ={} & S\big((q-1)(f^*D - K_Y) - \lceil (q-1)\Delta^* \rceil\big) \cong \sC_{e,S(f^*D)}^{\Delta^*}.
\end{align*}
where the inclusion is defined by the inequality $0 \leq f^* \lceil (q-1)\Delta \rceil - \lceil (q-1) f^* \Delta \rceil $.

\begin{claim} \label{cla.ExistenceOfa}
There exists a regular element $a \in R$ that belongs to the annihilator of $\coker(S \otimes_R R(D') \to S(f^* D'))$ for all almost Cartier divisors $D'$ on $X$. 
\end{claim}
\begin{proof}[Proof of claim]
Since $f$ is generically flat, we may find a free $R$-submodule $G \subset S$ of rank $[L:K]$ and such that $a \cdot S \subset G \subset S$ for some regular element $a \in R$. In particular, there is an induced commutative diagram for all almost Cartier divisor $D'$ on $X$
\[
\xymatrix{
(a \cdot S) \otimes R(D') \ar[r] \ar[d] & G \otimes R(D') \ar[r] \ar[d]^-{\cong} & S \otimes R(D') \ar[d] \\
a \cdot S(f^* D') \ar[r] & G' \ar[r]  & S(f^*D')
}
\]
where the vertical arrows are $\mathbf{S}_2$-ifications either as $R$-modules or $S$-modules when appropriate (which does not matter as both $R$ and $S$ satisfy $\mathbf{S}_2$) and the lower horizontal arrows are inclusions. In particular, $G' \cong R(D')^{\oplus [L:K]}$ is an $R$-submodule of $S(f^* D')$ that contains $a \cdot S(f^* D')$. This proves the claim.
\end{proof}

Let $a \in R$ be as in \autoref{cla.ExistenceOfa}, so $a \in \bigcap_e \Ann \coker \alpha_e$. Let $b \in S$ a regular element such that  $0 \leq f^* \lceil (q-1)\Delta \rceil - \lceil (q-1) f^* \Delta \rceil \leq \Div_S b $ (see \cite[\S 2.2]{SchwedeTuckerTestIdealFiniteMaps} and \cite[Lemma 3.10]{Carvajalphdthesis}). Then,
\[
\sC_{e,S(f^*D)}^{\Delta^*} \cdot b \subset S\Big((q-1)(f^*D-K_Y)-\big( f^*\lceil(q-1) \Delta \rceil-(q-1)\Ram_T\big)\Big) \subset \sC_{e,S(f^*D)}^{\Delta^*}
\]
for all $e$. Thus, we may take $c \coloneqq a \cdot b$. The remaining statements follow from \autoref{lem.RecoveringAllStuuf}.
\end{proof}

\section{Transformation rule for the $F$-signature and splitting ratio}

In this section, we prove transformation rules for $F$-signature and splitting ratio under finite covers that generalize those in \cite[Theorems 3.1, 4.4]{CarvajalSchwedeTuckerEtaleFundFsignature}, \cite[Theorem 4.11]{CarvajalFiniteTorsors}.

\subsection{Generalizing a Theorem of Tucker} Let $(R,\fram, \kay)$ be a local ring. Given $M$ in $R\textnormal{-fmod}$, we define $M^{\flat} \coloneqq \bigoplus_e M^{\flat}_e \coloneqq \bigoplus_e \Hom_R(F_*^e M, R)$, which is endowed with a graded left $\sC_R$-module structure: $\varphi \cdot \vartheta  = \varphi \circ F_*^d \vartheta$ for all $\varphi \in \sC_{d,R}$, $\vartheta \in M^{\flat}_{e}$. Denote the category of graded left modules over a Cartier algebra $\sC$ by $\sC$\textnormal{-glmod}.  Note that $M \mapsto M^{\flat}$ is a contravariant functor  $R\textnormal{-fmod} \to \sC\textnormal{-glmod}$.\footnote{However, $M^{\flat}$ is not necessarily finitely generated as an $R$-module, therefore not a Cartier module.} Given a Cartier algebra $\sC \subset \sC_R$ and $\sM \subset M^{\flat}$ in $\sC$\textnormal{-glmod}, we define the \emph{splitting numbers}
\[
a_e(\sM) \coloneqq \lambda_R\big(\sM_e/\sM_e^{\textnormal{ns}}\big),
\]
where lengths are computed as left $R$-modules and $\sM_e^{\textnormal{ns}}\coloneqq \sM_e \cap \Hom_R(F_*^eM, \fram)$. Note that $\sM_e^{\textnormal{ns}} \supset \fram \cdot \sM_e$ and so $a_e(\sM)$ is finite. Set $n_{\sM} \coloneqq \gcd \{ e \mid a_e(\sM) \neq 0\}$ and $\delta=\dim R + \log_p [\kay^{1/p}:\kay]$ as in \cite{BlickleSchwedeTuckerFSigPairs1}. Following \cite{YaoModulesWithFFRT}, we define the \emph{splitting ratio of $\sM$}
\[
\asr(\sM)\coloneqq \lim_{e\rightarrow \infty}{a_{e \cdot n_{\sM}}(\sM) /q^\delta }.
\]
\begin{theorem}[{\cf \cite[Theorem 4.11]{TuckerFSigExists}}]
If $(R, \fram, \kay, K)$ is a local domain and $M$ is a finitely generated $R$-module, then $\asr(M^{\flat})=\rank M \cdot s(R, \sC_R)$.
\end{theorem}
Let $\sC \subset \sC_R$ be a Cartier algebra and $M$ in $R\textnormal{-fmod}$. Define $M^{\natural,\sC}\subset M^{\flat}$ in degree $e$ as:
\[
M^{\natural}_e = M^{\natural, \sC}_e \coloneqq  \big\langle \varphi \circ F_*^e \rho \in M^{\flat}_e \bigm| \varphi \in \sC_e \text{ and } \rho \in M^{\vee}=\Hom_R(M,R)\big\rangle_{\mathbb{Z}}.
\]
Observe that $M \mapsto M^{\natural}$ defines a contravariant functor $R\textnormal{-fmod} \to \sC\textnormal{-glmod}$. Consider the functor $R\textnormal{-fmod} \to \sC\textnormal{-glmod}$, $M \mapsto \sC \otimes_R M^\vee$. Then, $\eta_M\: \sC \otimes_R M^\vee \to M^{\natural}$, $\varphi \otimes \rho \mapsto \varphi \circ F_\ast^e \rho$ is a surjective natural transformation (although not injective in general).

\begin{theorem} \label{thm.TuckerGeneralized} Let $(R,\fram, \kay, K)$ be a local domain and $\sC \subset \sC_R$ a Cartier $R$-algebra. Then, $\asr\big(M^{\natural} \big)=\rank M \cdot s(R,\sC)$ for all finite $R$-modules $M$.
\end{theorem}
\begin{proof}
Let $g \coloneqq \rank M$ and consider a short exact sequence of $R$-modules
\[
0 \to L \xrightarrow{\subset} M \to M/L \to 0,
\]
where $L$ is free of rank $g$ and $T\coloneqq M/L$ is torsion; \ie $\Ann_R T \neq 0$. Applying the exact functor $F_*^e$ followed by the left exact functor $\Hom_R(-,R)$ yields an exact sequence
\[
0 \to M^{\flat}_e \overset{\iota}{\longrightarrow} L^{\flat}_e,
\]
as $\Hom_R(F_*^eT, R) = 0$ (for $R$ is a domain and $F_*^eT$ is torsion). The map $\iota$ is none other than restriction. Thus, a map $\vartheta \: F^e_* M \to R$ is determined by its values at $F^e_* L \subset F^e_* M$. Let us think of $\iota$ as an inclusion $M^{\flat}_e \subset L^{\flat}_e$ by realizing $M^{\flat}_e$ inside $L^{\flat}_e$ as the maps $\vartheta \: F^e_* L \to R$ admitting a (necessarily unique) extension to a map $F^e_* M \to R$. Note that $\iota$ respects nonsurjectivity, \ie $\big( M^{\flat}_e \big)^{\textnormal{ns}} \subset \big( L^{\flat}_e \big)^{\textnormal{ns}}$.
For a nonzero $ c \in \Ann_ R T$, we have $L^{\flat}_e \cdot c \subset M^{\flat}_e \subset L^{\flat}_e$ and we readily see that it restricts to
\begin{equation} \label{eqn.GenericContainmentAsterisk}
L^{\natural}_e \cdot c \subset M^{\natural}_e \subset L^{\natural}_e.
\end{equation}
Consequently, one may conclude as usual (using the well-known argument to experts):
\begin{align*}
a_e(M^{\natural}) &= \lambda_R \big(F^e_* L\bigm/\{F^e_* l \mid \vartheta(F^e_* l) \in \fram \text{ for all } \vartheta \in M^{\natural}_e \}\big) \\
&= \sum_{i=1}^g \lambda_R \big(F^e_* R \bigm/\{F^e_* r \mid (\vartheta \circ \sigma_i)(F^e_* r) \in \fram \text{ for all } \vartheta \in M^{\natural}_e \}\big)
\end{align*}
where $\sigma_i \: F^e_* R \to F^e_* L$ are the direct sum structural maps. Consider the ideals\footnote{It is worth noting that $J_{i,e}$ is an ideal because $\vartheta(\sigma(F_\ast^e rs)) = \vartheta(F_\ast^e r \cdot \sigma_i(F_\ast^es)) = ((\vartheta \cdot r) \circ \sigma_i)(F^e_*s)$.} $J_{i,e} \coloneqq \{ r \mid (\vartheta \circ \sigma_i)(F^e_* r) \in \fram \text{ for all } \vartheta \in M^{\natural}_e \} \subset R$ and $I_e \coloneqq \{r \mid \varphi(F^e_* r) \in \fram \text{ for all } \varphi \in \sC_e \} \subset R$. Then \autoref{eqn.GenericContainmentAsterisk} implies $I_e \subset J_{i,e} \subset I_e : c$ for all $i=1,\ldots ,g$. One then argues as in \eg \cite[Lemma 4.17]{BlickleSchwedeTuckerFSigPairs1}, \cite[Lemma 2.7]{CarvajalSchwedeTuckerEtaleFundFsignature}, \cite[Theorem 4.11]{TuckerFSigExists}.
\end{proof}

\subsection{Transformation rule for $F$-signature} The desired formula is the following.

\begin{theorem} \label{thm.MainTheorem}
Let $\theta\: (R,\fram, \kay) \to (S, \fran, \el)$ be a finite local extension defining a cover $f\: \Spec S \to \Spec R$ and $\sC$ be a Cartier $R$-algebra. Suppose that $R$ is an integral domain with field of fractions $K$, set $L \coloneqq S \otimes_R K$, and write $[L:K] \coloneqq \dim_K L$. Suppose that there is a \emph{generic} isomorphism $\sigma_R\: S \to \omega_{S/R}$ of $S$-modules such that $T\coloneqq\sigma_R(1)$ is surjective and $T(\fran) \subset \fram$. If $R$ is a $T$-transposable Cartier $\sC$-module, then
\[
[\el:\kay]\cdot s(S,f^* \sC ) = [L:K] \cdot s(R,\sC).
\]
In particular, $(R,\sC)$ is $F$-regular if and only if so is $(S,f^* \sC)$.
\end{theorem}
\begin{proof}
Note that $\theta$ is an extension of $\mathbf{S}_1$ rings. Let $d= \dim R =\dim S$ and $\delta = d+\log_p[\kay^{1/p}:\kay]= d +\log_p [\el^{1/p}: \el]$. Note that:
\[
[\el:\kay] \cdot a_e(f^*\sC)=[\el:\kay] \cdot \lambda_S\big(f^*\sC_e \big/ (f^*\sC_e)^{\textnormal{ns}}\big) = \lambda_R \Big( f_*\big(f^*\sC_e\big/(f^*\sC_e)^{\textnormal{ns}} \big) \Big).
\]

\begin{claim} \label{cla.KeyClaim} If $\sigma_R$ is an isomorphism, the isomorphism given by \autoref{pro.GroDuality} (\cf \autoref{rem.RemarkRelGor}) $ \xi = \xi(R,F^e_* S) \: f_*\Hom_S(F_*^e S, S) \to \Hom_R(F_*^e f_*S, R)$, $\psi \mapsto T \circ f_*\psi$, induces an isomorphism
\[
f_*\big(f^*\sC_e \big/ (f^*\sC_e)^{\textnormal{ns}} \big)  \to  (f_* S)_e^{\natural} \big/\bigl( (f_* S)_e^{\natural} \bigr)^{\textnormal{ns}}.
\]
\end{claim}
\begin{proof}[Proof of claim] 
We must prove the equality $\xi\big(f_*f^* \sC_e\big) = (f_*S)^{\natural}_e $ and that $\psi$ is surjective if and only if so is $\xi(\psi)$. Let us recall that an element of $f^* \sC_e$ is a finite sum $\sum_i {\varphi^{\top}_i \cdot s_i}$ where all $\varphi_i$ are in $\sC_e$; see \autoref{rem.RemarkRelGor}. Thus, for such element, we have:
\[
\xi \bigg( \sum_i {\varphi_i^{\top} \cdot s_i} \bigg) = \sum_i T \circ \big(\varphi^{\top}_i \cdot s_i\big) = \sum_i \varphi_i \circ F_*^e (s_i \cdot T) \in (f_* S)^{\natural}_e.
\]
In other words, $\xi\big(f_*f^* \sC_e\big) \subset (f_* S)^{\natural}_e$. Conversely, on the right-hand side of the above equality, we hit all the elements of $(f_* S)^{\natural}_e$ as any $\rho \in \Hom_R(S,R)$ is of the form $s\cdot T$ for some $s \in S$, for $\sigma_R$ is an isomorphism. The equivalence between the surjectivity of $\psi$ and $\xi(\psi)=T \circ \psi$ follows from the remaining two hypothesis on $T$. Indeed, since $T$ is surjective, then $\xi(\psi)$ is surjective if so is $\psi$. Conversely, suppose $\psi$ is not surjective, meaning that it maps $F_*^e S$ into $\fran$. Then, since $T(\fran) \subset \fram$, we have that $\xi(\psi)$ maps $F_*^e S$ into $\fram$. This proves the claim.
\end{proof}

If $\sigma_R$ were an isomorphism, we could use \autoref{cla.KeyClaim} directly to write
\[
[\el:\kay] \cdot a_e(f^*\sC) = \lambda_R \Big( f_*\big(f^*\sC_e\big/(f^*\sC_e)^{\textnormal{ns}} \big) \Big) = \lambda_R \Big( (f_* S)^{\natural}_e \big/ \big( (f_* S)^{\natural}_e\big)^{\textnormal{ns}} \Big) = a_e\big((f_* S)^{\natural}\big).
\]
Dividing by $q^{\delta}$, letting $e\rightarrow \infty$, and using \autoref{thm.TuckerGeneralized} yields desired result. We are, however, assuming $\sigma_R$ to be an isomorphism only generically and so \autoref{cla.KeyClaim} and its proof still hold generically. Let $c \in S$ be a regular element in $\Ann_S \coker \sigma_R$. Thus, $\xi$ still sends $f_*f^* \sC_e$ injectively into $(f_* S)^{\natural}_e$ preserving surjectivity/nonsurjectivity in the process and
\[
(f_* S)^{\natural}_e\cdot c\subset \xi\big(f_*f^* \sC_e\big) \subset (f_* S)_e^{\natural}.
\]
As in the proof of \autoref{thm.TuckerGeneralized}, this suffices to conclude that 
\[\lambda_R\Big(\xi\big(f_*f^* \sC_e\big) \big/ \xi\big(f_*f^* \sC_e\big)^{\mathrm{ns}} \Big) \Big/q^{\delta} \text{ and } a_e\big((f_* S)^{\natural}\big)/q^{\delta}\] 
have the same limit as $e$ goes to $\infty$, which is all we need to conclude as before. 
\end{proof}

\begin{remark} \label{rem.Comparison}
Using \autoref{pro.RecoveringAllStuff}, we see that \autoref{thm.MainTheorem} recovers the transformation rules in \cite{CarvajalSchwedeTuckerEtaleFundFsignature,CarvajalFiniteTorsors}.
\end{remark}

\begin{remark}
If $\sigma_R$ is an isomorphism in \autoref{thm.MainTheorem}, $[\el: \kay]$ is the free rank of $S$ as an $R$-module: $\textnormal{frk}_R f_*S =\lambda_R\bigl( \Hom_R(S,R)\big/ \Hom_R(S, \fram)\bigr) = \lambda_R(S \cdot T/\fran \cdot T) = \lambda_R(\el) = [\el : \kay]$.
\end{remark}

\begin{example} \label{ex.NecessityHypothesisTransRule} The following examples show why the hypothesis in \autoref{thm.MainTheorem} about $T$ are necessary; see \cite[Example 3.15]{Carvajalphdthesis} for further details.
\begin{enumerate}[(a)]
\item{For the surjectivity of $T$, consider \cite[Example 7.12]{SchwedeTuckerTestIdealFiniteMaps} of a
a quasi-\'etale degree-$2$ extension $R\subset S=\mathbb{F}_2 \llbracket u, v\rrbracket$ of $2$-dimensional $\mathbb{F}_2$-algebras such that $\Tr_{S/R}$ is not surjective. In fact, $R$ is a log terminal singularity that is $F$-pure but not strongly $F$-regular. $R$ is the ring of invariants of $S$ under certain non-linear action of $\mathbb{Z}/2\mathbb{Z}$ due to M.~Artin \cite{ArtinWildlyRamifiedZ2Actions}.}
\item{For $T(\fran) \subset \fram$, consider any Noether normalization $R \subset S$ of a singular Gorenstein local ring $S$ such that $\Char \kay \nmid [\sK(S):\sK(R)]$. For instance, take $R = \kay \bigl\llbracket x^2, y^2 \bigr\rrbracket \subset \kay \bigl\llbracket x^2, xy, y^2 \bigr\rrbracket = S $ with $\Char \kay \neq 2$. Then, $\omega_{S/R} \cong S$, say with free generator $T$. Then $T$ is surjective as $R \subset S$ splits. However, \autoref{thm.MainTheorem} fails as otherwise it would imply $s(S) \geq s(R)=1$ (but $S$ is singular). A free basis for $\kay\bigl \llbracket x^2, y^2\bigr \rrbracket \subset \kay\bigl \llbracket x^2,xy,y^2\bigr\rrbracket$ is $1, xy$ and $T$ may be taken to be $(xy)^{\vee}$. Hence, $T(\fran) \not\subset \fram$ as $T(xy)=1$.}
\end{enumerate}
\end{example}

\subsection{Splitting primes and splitting ratios under finite covers}
Since splitting ratios are $F$-signatures of Cartier algebras \cite{BlickleSchwedeTuckerFSigPairs1}, it is natural to expect them to satisfy transformation rules. In this regard, we have the following.
\begin{theorem} \label{thm.TransRuleSplittingRatios}
With hypothesis as in \autoref{thm.MainTheorem}, the contraction of $\sp(S,f^*\sC)$ along $\theta$, \ie $\sp(S,f^* \sC )\cap R$, is $\sp(R,\sC)$. % By abuse of notation, we write $\sp(S,f^* \sC )\cap R = \sp\big(R,\sC\big)$. 
Furthermore:
\[
\big[\kappa(\fran):\kappa(\fram)\big] \cdot r\big(S,f^* \sC\big) = \big[\kappa\big(\sp(S,f^* \sC)\big):\kappa\big(\sp(R,\sC)\big)\big] \cdot r\big(R,\sC\big),
\]
where $\kappa(-)$ denotes the residue field at the respective prime ideal. In particular, $(R,\sC)$ is $F$-pure (resp. strongly $F$-regular) if and only if $(S,f^*\sC)$ is so.
\end{theorem}
\begin{proof}
We show first $\sp\big(S,f^* \sC\big) \cap R = \sp(R,\sC)$.
%and treat $\phi$ as an inclusion.
We must prove that for $r\in R$, the existence of $\varphi \in \sC_e$ such that $\varphi(F^e_* r) =1$ is equivalent to the existence of $\psi \in f^*\sC_e$ such that $\psi(F^e_* r) = 1$.

Let $v \in S$ be a unit such that $T(v)=1$, which exists because $T$ is surjective and $T(\fran) \subset \fram$. Let $\varphi \in \sC_e$ such that $\varphi(F^e_* r) = 1$, we claim that $\varphi^{\top}\cdot v$ maps $F^e_*r$ to a unit in $S$. Indeed, if $\varphi^{\top}(F_*^e v \cdot r) \in \fran$, then $T\big((\varphi^{\top}(F_*^e v \cdot r ) \big) \in T(\fran) \subset \fram$. However, 
\[
T\big(\varphi^{\top}(F_*^e v \cdot r) \big) = \varphi\big( F_*^e T(v\cdot r) \big) =\varphi\big( F_*^e r \big) =1.
\]
Conversely, say there is $\psi \in f^*\sC_e$ such that $\psi(F^e_* r)=1$, so $\psi(F^e_* v^q\cdot r)=v$. Since $\psi$ is a sum of elements of the form $\varphi^{\top}\cdot s$, we may assume $\psi = \varphi^{\top} \cdot s$ with $\varphi \in \sC_e$, $s\in S$. Thus,
\[
\varphi^{\top}(F^e_* sv^q \cdot r)=v.
\]
Hitting this equality with $T$ and using $T\circ \varphi^{\top} = \varphi \circ F_*^e T$ gives:
\[
\varphi\big(F^e_* T(sv^q\cdot r)\big)=\varphi\big(F^e_* T(sv^q)\cdot r\big)=1.
\]
In other words, $\varphi \cdot T(sv^q)$ sends $F_*^e r$ to $1$. This shows $\sp\big(S,f^* \sC\big) \cap R = \sp(R,\sC)$.

In this manner, to prove the transformation rule, we may assume $(R,\sC)$ and therefore $(S,f^*\sC)$ are $F$-pure as otherwise the transformation rule is trivially true ($0=0$). Observe that $R \subset S$ restricts to a local inclusion of domains
\begin{equation} \label{eqn.LocalQUotient}
R\big/\sp\big(R,\sC\big) \subset S\big/\sp\big(S,f^* \sC\big),
\end{equation}
with corresponding morphism of schemes denoted by $\overline{f}$. Further, let us set $\p \coloneqq \sp (R,\sC )$, $\q \coloneqq \sp (S,f^* \sC )$, $\overline{R} \coloneqq R/\p$, and $\overline{S} \coloneqq S/\q$. In order to apply the transformation rule in \autoref{thm.MainTheorem} to \autoref{eqn.LocalQUotient}, we show that $\overline{f}$ inherits the properties of $f$. To this end, we note that
\begin{equation} \label{eqn.TmapsSplittingPrimesToSplittingPrimes}
T(\q) \subset \p,
\end{equation}
as $\varphi\big(F_*^e T(s)\big) = T\big(\varphi^{\top}(F_*^e s)\big) \in T(\fran) \subset \fram$ for all $\varphi \in \sC_e$ if $s\in \q = \sp (S, f^*\sC)$.
In other words, $T$ restricts to a unique map $\overline{T} \in \Hom_{\overline{R}}\big(\overline{S}, \overline{R}\big)$
such that the square
\[
\xymatrix@C=3em{
S  \ar[r]^-{ T } \ar@{->>}[d]_-{} & R \ar@{->>}[d]^-{}\\ 
  \overline{S} \ar[r]^-{\overline{T}} & \overline{R}
}
\]
is commutative. The same holds for any $S$-multiple of $T$, so that we have an $\overline{S}$-linear map
\[
\overline{\sigma}\:\overline{S} \to \Hom_{\overline{R}}\big(\overline{S}, \overline{R}\big) = \omega_{\overline{S}\big/\overline{R}}, \quad \overline{s} \mapsto \overline{T}\cdot \overline{s} = \overline{T\cdot s}.
\]
\begin{claim} \label{cla.Inheritingroperties}
$\overline{\sigma}$ is injective and generically an isomorphism. $\overline{T}$ is surjective and $\overline{T}(\overline{\fran}) \subset \overline{\fram}$
\end{claim}
\begin{proof}[Proof of claim]
The last two statements are clear. For the injectivity of $\overline{\sigma}$, consider the following. If $\overline{T \cdot s} = 0$, then $T(ss')\in \p$ for all $s'\in S$. Equivalently,
\begin{equation}
\label{eqn.TransposeSplittingPrime}
\varphi\big( F_*^e T(ss') \big) = T\big( \varphi^{\top}(s's) \big)  \in \fram,
\end{equation}
for all $\varphi \in \sC_e$. We must conclude that $s \in \q$. Assume to the contrary: there is $\varphi \in \sC_e$ and $s' \in S$ such that $\varphi^\top(s's) \in S \smallsetminus \fran = S^\times$. Rescaling $s'$, we may assume that $\varphi^\top(s's) = v$ which contradicts \autoref{eqn.TransposeSplittingPrime} as $T(v)=1$. Therefore, $\varphi^{\top}(s's) \in \fran$ for all $\varphi \in \sC_e$ and $s' \in S$. Hence $s \in \q$, as required. To show that $\overline{\sigma}$ is generically surjective, it suffices to prove that $\overline{T} \neq 0$, \ie $T(S) \not \subset \p$. This follows from the surjectivity of $T$ and the $F$-purity of $(R,\sC)$.
\end{proof}

\begin{claim} \label{cla.TranspositionRestriction}
$\big(\overline{R}, \overline{\sC}\big)$ is $\overline{T}$-transposable. In fact, $\overline{\varphi}^{\top} = \overline{\varphi^{\top}}$ and in particular $\overline{f}^* \overline{\sC} = \overline{f^* \sC}$.
\end{claim}
\begin{proof}[Proof of claim]
Recall that $\varphi^{\top}$ is the unique map fitting in \autoref{eqn.TransposabilitySimpleEquation} and similarly for $\overline{\varphi}^{\top}$. Reducing \autoref{eqn.TransposabilitySimpleEquation} modulo $\q$ yields that $\overline{\varphi^{\top}}$ satisfies the condition characterizing $\overline{\varphi}^{\top}$.
\end{proof}
Combining \autoref{thm.MainTheorem} with the above gives the desired transformation rule.
\end{proof}

\section{Test modules under finite covers} \label{sec.testmodulefinitemorphism}
We come now to our generalizations of the results in \cite{SchwedeTuckerTestIdealFiniteMaps} (\cf \cite[Theorem D]{MaSchwedeSingularitiesNixedCharacteriticPerfectoidBCM}) of the form $T\big(\uptau(B, \Delta_B)\big) = \uptau(A, \Delta_A)$ for a cover $\Spec B \to  \Spec A$ between normal varieties. We prove a more general version for Cartier modules which says that $\Tr_M \bigl(f_\ast \uptau(f^!M, f^\ast \sC)\bigr) = \uptau(M, \sC)$ for any cover $f$ (this was proved for flat morphisms in \cite[Lemma 4.17]{staeblerunitftestmodules} under an additional technical assumption) and then utilize transposability to obtain more special results yet recovering the classical ones. In subsequent sections, we shall remark the form these results take in the case of canonical modules as well as providing analogous results for non-$F$-pure ideals and test ideals along closed subschemes.

Let us recall some concepts from \cite{BlickleStablerFunctorialTestModules}. Fix a Cartier algebra $\sC$ over a ring $R$ and consider Cartier modules for this fixed algebra.
For a Cartier module $M$, one denotes by $\underline{M} \coloneqq \sC_+^e M$ the stable image of $M$ under $\sC_+$. This exists by \cite[Proposition 2.13]{BlickleTestIdealsViaAlgebras} and one can show that $\underline{M}$ does not admit any nilpotent quotients. In particular, if $N \subset \underline{M}$ is a nil-isomorphism it is an equality. Moreover, the inclusion $\underline{M} \subset M$ is always a nil-isomorphism. Hence, to check that an inclusion $N \subset M$ is a nil-isomorphism, one may equivalently pass to stable images and show that one has an equality $\underline{N} = \underline{M}$. We also note that the operation taking the stable image commutes with localization; see \cite[Lemma 2.11]{BlickleTestIdealsViaAlgebras}. Another useful observation is that $\underline{H^0_\eta(M)} = \underline{i_\ast i^! M}$, where $i: \Spec R/\eta \to \Spec R$ is the natural closed immersion; see \cite[Lemma 3.2]{BlickleStablerFunctorialTestModules}. Moreover, the annihilator of $\underline{M}$ is always a radical ideal. We use freely that $\uptau$ commutes with localization, that is $\uptau(M, \sC)_\eta = \uptau(M_\eta, \sC_\eta)$, where $\sC_\eta = \sC \otimes_R R_\eta$ and $\eta$ is any prime in $R$ \cite[Proposition 1.19 (b)]{BlickleStablerFunctorialTestModules}. We also use that the (underived) local cohomology functor $H^0_\eta$ is a functor of $\sC$-modules (\cf paragraph after \cite[Example 1.4]{BlickleStablerFunctorialTestModules}).

\begin{theorem}
\label{theo.TraceTauSurjective}
Let $f\: \Spec S \to \Spec R$ be a finite cover. Let $\sC$ be a Cartier $R$-algebra and $M$ a Cartier $\sC$-module. Then, 
\begin{equation*}
\Tr_{M}\big(f_\ast \uptau (f^! M, f^*\sC ) \big) = \uptau(M,\sC).
\end{equation*}
In particular, $\Tr_M$ is surjective if $(M,\sC)$ is $F$-regular. Conversely, if $\Tr_M$ is surjective and $(f^! M, f^*\sC)$ is $F$-regular, then $(M,\sC)$ is $F$-regular.
\end{theorem}
\begin{proof}
By \cite[Proposition 6.13]{BlickleStablerFunctorialTestModules}, there is a natural inclusion $\uptau \circ f^{!} \hookrightarrow f^{!} \circ \uptau$ and so
\[
\uptau\big(f^! M, f^*\sC\big) \subset \Hom_R\big(S, \uptau(M, \sC)\big) \subset \Hom_R(S,M).
\]
Therefore, $\Tr_{M}\big(f_\ast \uptau\big(f^! M, f^*\sC\big) \big) \subset \uptau(M,\sC)$. Set $\sT \coloneqq \Tr_{M}\big(f_\ast \uptau\big(f^! M, f^*\sC\big) \big)$. For the converse inclusion, recall that for any $f^*\sC$-submodule $N \subset f^! M$ we have that $\Tr_M(f_* N)$ is a $\sC$-submodule of $M$; see \autoref{rem.RemarkRelGor}. Then, by definition of $\uptau(M,\sC)$, it suffices to show that $H^0_\eta(\sT)_\eta \subset H^0_\eta(M)_\eta$ is a nil-isomorphism for all $\eta \in \Ass M$. By applying $\underline{\phantom{M}}$ and using that for any Cartier module $N$ the inclusion $\underline{N} \subset N$ is a nil-isomorphism, it suffices to show that
\begin{equation}\label{eqn.testmodulelocalcohom}
\underline{H^0_\eta(\sT)}_\eta \subset \underline{H^0_\eta(M)}_\eta \end{equation} is a nil-isomorphism (equivalently, an equality). Since localization is flat, $(\Tr_M(N))_{\eta} = \Tr_{M_\eta}(N_\eta)$ for all submodules $N \subset f^!M$, where $\Tr_{M_\eta}$ is the trace map of the base change of $f$ along $\Spec R_\eta \to \Spec R$.
Note that $\Tr$ is compatible with Cartier structures and so preserves nil-isomorphisms. We will show the following equalities and inclusions thereby showing that \autoref{eqn.testmodulelocalcohom} is an equality.
\begin{equation} \label{eq.whatweshow} \underline{H^0_\eta(M)}_\eta \overset{(1)}{=} \Tr_{M_\eta}\Bigl(\underline{H^0_{\eta S}(f^!M)}_\eta \Bigr) \overset{(2)}{=} \Tr_{M_\eta}\Bigl(\underline{H^0_{\eta S}\big(\uptau(f^! M, f^\ast \sC)\big)}_\eta \Bigr) \overset{(3)}{\subset}  \underline{H^0_\eta(\sT)}_\eta\end{equation}

Most readers will be more familiar with the case where $R$ and $S$ are domains, $\eta$ is the generic point of $R$, and $M$ has rank $1$. Up to technicalities, the general case proceeds very much like that one. Roughly speaking, the stable image of $H^0_\eta(f^!M)_\eta$ may be identified with $\Hom_{R/\eta}(S/\eta, M)_\eta$ and the stable image of $H^0_\eta(M)_\eta$ may be identified with the $\eta$-torsion of $M_\eta$. Denoting the $\eta$-torsion of $M$ by $N$, we therefore have to consider the Cartier module $N$ on $\Spec R/\eta$ where $\eta$ is now a generic point. Since $S$ is not assumed to be a domain, we need to work with the points in the fiber of $\eta$ separately. We now show \autoref{eq.whatweshow}.

\begin{claim} (1) holds in \autoref{eq.whatweshow} above.
\end{claim}
\begin{proof}[Proof of claim]
The restriction of the trace map to $\underline{H^0_{\eta S}(f^! M)}_\eta$ may be identified with the trace map $\Hom_{R_\eta/\eta}\bigl((S/\eta S)_\eta, M_\eta \bigr) \to M_\eta$. In particular, $R_\eta/\eta$ is a field and thus ${(S/\eta S)}_\eta$ is free with $1$ being part of some basis $B$. The containment from right to left is clear. Let $m$ belong to the left hand side. Fix $e \gg 0$ such that $\sC_+^e H^0_\eta(M)_\eta = \underline{H^0_\eta(M)}_\eta$ and such that it also computes $\underline{H^0_{\eta S}(f^!M)_\eta}$. Let $\varphi \in \mathcal{C}_+^e$ be homogeneous of degree $a$ and $n \in H^0_\eta(M)_\eta$ be such that $\varphi(n) = m$. Our goal is to construct an element $\alpha \in (f^!M)_\eta$ such that $\varphi(\alpha) \coloneqq \varphi \circ \alpha \circ F^a$ is $(\eta S)^k$-torsion and $\varphi(\alpha(1)) = m$. To this end, we define $\alpha\: (S/\eta S)_\eta \to M_\eta$ by setting $\alpha(1) = n$ and $\alpha(b) = 0$ for all other $b \in B$.
Write $\varphi = \sum_{i} \varphi_i$, where $\varphi_i$ is homogeneous of degree $i$. Then, for any $b \in B$, we have 
\begin{equation}
\label{eqn.testmodulestorsioncomp}
\varphi(\alpha)(b) = \sum_i \varphi_i \alpha\bigl(b^{p^{i}}\bigr) = \sum_i \varphi_i \alpha(r_{1,i}) = \sum_i \varphi_i(r_{1,i} n),
\end{equation} where we write $b^{p^i} = \sum_{b \in B} r_{b, i} b$. If $b =1$ in \autoref{eqn.testmodulestorsioncomp}, then $r_{1,i} = \delta_{i1}$ and the expression yields $m$. For general $b$, we still have $(\eta S)^k$-torsion since $n$ is $\eta^k$-torsion.
\end{proof}

\begin{claim}
(2) holds in \autoref{eq.whatweshow} above.
\end{claim}
\begin{proof}[Proof of claim]
We show that $\underline{H^0_{\eta S}(f^! M)}_\eta$ and $\underline{H^0_{\eta S}(\uptau(f^!(M, f^\ast \sC))}_\eta$ agree. To this end, recall that $\Ass f^!M = f^{-1}(\Ass M)$ by \cite[Lemma 6.12]{BlickleStablerFunctorialTestModules}. Note that the support of $ \underline{H^0_{\eta S}(f^! M)}_\eta$ is contained in $f^{-1}(\eta)$, which consists of a finite number of points. Hence,
\begin{equation}
\label{eqn.stableimage1}
\underline{H^0_{\eta S}(f^!M)}_\eta = \bigoplus_{\nu \in f^{-1}(\eta)} \underline{H^0_{\nu}(f^! M)}_\nu,
\end{equation}
where we note that a direct summand on the right hand side is zero whenever $\nu \notin \Supp \underline{H^0_{\eta S}(f^! M)}_\eta$. A similar statement holds for $\uptau(f^!M, f^\ast \sC)$ instead of $f^!M$. By definition of $\uptau$, we have for each $\nu \in f^{-1}(\eta)$ a nil-isomorphism $H^0_{\nu}\bigl(\uptau(f^!M, f^\ast \sC)_\nu\bigr) \subset H^0_{\nu}(f^!M_\nu)$, where again we may pass to the stable image on both sides and take direct sums to obtain:
\begin{equation}
\label{eqn.stableimage2}
\bigoplus_{\nu \in f^{-1}(\eta)} \underline{H^0_{\nu}\bigl(\uptau(f^!M, f^\ast \sC)\bigr)}_\nu = \bigoplus_{\nu \in f^{-1}(\eta)} \underline{H^0_{\nu}(f^!M)}_\nu.
\end{equation}Putting \autoref{eqn.stableimage1} and \autoref{eqn.stableimage2} together, we obtain
\begin{equation} \label{eqn.testmodulelocalcohom2} \underline{H^0_{\eta S}(\uptau(f^!M, f^\ast \sC))}_{\eta} = \underline{H^0_{\eta S} (f^!M)}_\eta. \end{equation} \end{proof}

\begin{claim}
(3) holds in \autoref{eq.whatweshow} above.
\end{claim}
\begin{proof}[Proof of claim]
We may drop the $\underline{\phantom{M}}$ on the right hand side. Indeed, $\Tr_{M_\eta}$ is a morphism of Cartier modules and quite generally, if $g$ is a Cartier morphism and $g(\underline{N}) \subset M$, then $\underline{M} = \sC_+^e M \supset \sC^e_+ g(\underline{N}) = g(\sC_+^e \underline{N}) = g(\underline{N})$. By construction, the left hand side is $\eta$-torsion. Hence, it suffices to show that it is contained in $\sT_\eta$ . Since $\Tr_M$ commutes with localization:
\[ 
\sT_{\eta} = \Tr_{M_\eta}\Bigl(\bigl(f_\ast \uptau(f^!M, f^\ast \sC)\bigr)_\eta\Bigr).
\]
By definition of $\uptau(f^!M, f^\ast \sC)$, for any $\nu \in \Spec S$, the inclusion $H^0_\nu\bigl(\uptau(f^!M, f^\ast \sC)\bigr)_\nu \subset H^0_\nu(f^!M)_\nu$ is a nil-isomorphism. This entails 
\[
\underline{H^0_\nu(f^!M)}_\nu = \underline{H^0_\nu\bigl(\uptau(f^!M, f^\ast \sC)\bigr)}_\nu \subset H^0_\nu\bigl(\uptau(f^!M, f^\ast \sC)\bigr)_\nu.
\] 
Taking the (direct) sum over all $\nu \in f^{-1}(\eta)$ in the inclusion above, we get 
\begin{align*}
 \bigoplus_{\nu \in f^{-1}(\eta)} \underline{H^0_\nu(\uptau(f^!M, f^\ast \sC))}_\nu =\underline{H^0_{\eta S}(f^! M)}_\eta & \subset  \bigoplus_{\nu \in f^{-1}(\eta)} H^0_\nu\bigl(\uptau(f^!M, f^\ast \sC)\bigr)_\nu\\ &= H^0_{\eta S}\bigl(\uptau(f^!M, f^\ast \sC)\bigr)_\eta \subset \uptau(f^!M_\eta, f^\ast \sC_\eta).
\end{align*}
Applying $\Tr$ shows the desired inclusion.
\end{proof}
This demonstrates the result.
\end{proof}

\begin{corollary}[{\cf \cite{SchwedeTuckerTestIdealFiniteMaps}}] \label{cor.SpecializationTransposability}
Work in \autoref{set.Transposability}. Let $\sC$ be a Cartier $R$-algebra. Then,
\begin{equation}
\label{eq.tracetaugeneral}
T_{M}\big(f_* \uptau(f^\dagger M, f^*\sC) \big) = \uptau(M,\sC)
\end{equation}
for all $T$-transposable Cartier $\sC$-modules $M$, where $T_M = \Tr_M \circ \varsigma_M$. In particular, if $R$ is a $T$-transposable Cartier $\sC$-module, we have
\begin{equation} \label{eqn.TransformationRuleTestIdelasRingsTrans}
T\big(f_* \uptau (S, f^*\sC ) \big) = \uptau(R,\sC).
\end{equation}
Further, working in the setup of \autoref{thm.TransposabilityCriterion}, if $R(D)$ is $T$-transposable then
\begin{equation}
\label{eq.tracetauspecial}
T_K\Big(f_* \uptau\big(S(f^*D), f^*\sC\big) \Big) = \uptau\big(R(D),\sC\big)
\end{equation}
where $T_K \: L \to K$ coincides with $K \otimes_R T$.
Moreover, if $f$ is generically flat and $\Delta^* \coloneqq f^*\Delta - \Ram_T \geq 0$ for some effective almost Cartier $\bQ$-divisor on $X$, then
\begin{equation}
\label{eq.tracetauschwedetuckerrecovered}
T_K\Big(f_* \uptau\big(S(f^*D), \Delta^*, (\fraa S)^t\big) \Big) =
\uptau\big(R(D),\Delta, \fraa^t \big).
\end{equation}
\end{corollary}
\begin{proof}
By \autoref{cor.varSigmatestModules}, $\varsigma_M(f_\ast \uptau(f^\dagger M, f^\ast \sC)) = \uptau(f^! M, f^\ast \sC)$. Applying $\Tr_M$ and using \autoref{theo.TraceTauSurjective} shows \autoref{eq.tracetaugeneral}. For \autoref{eq.tracetauspecial}, \autoref{rem.Interpreation} gives that $f^\dagger R(D) \cong S(f^\ast D)$ and, under this isomorphism, $T_{R(D)}$ corresponds to the composition
\[
S(f^*D) \xrightarrow{\Ram_T \sim K_{S/R}} S(f^*D+K_{S/R}) \xrightarrow{} R(D)
\]
where, by naturality, the latter map is the restriction of $T_K \: L \to K$ under the inclusions $R(D) \subset K$ and $S(f^*D) \subset L$. Lastly, \autoref{eq.tracetauschwedetuckerrecovered} follows from \autoref{eq.tracetauspecial} and \autoref{pro.RecoveringAllStuff}.
\end{proof}

We illustrate next the simpler form \autoref{theo.TraceTauSurjective} takes in case $M=\omega_R$. Recall that if $\kappa_R \: F^e_* \omega_R \to \omega_R$ is a Cartier operator of $R$ then $\kappa_S \coloneqq \kappa_R^{!} \: F^e_* \omega_S \to \omega_S$ is a Cartier operator of $S$ as $ \Delta_{\kappa_S}=f^* \Delta_{\kappa_R} = f^* 0 = 0$.

\begin{corollary}[\cf \cite{BlickleSchwedeTuckerTestAlterations}]
\label{cor.CanonicalModulesTrans}  \label{cor.CohenMacualayCanonicalModulesTrans}
Work in the setup of \autoref{thm.TransposabilityCriterion} assuming $R$ and $S$ are Cohen--Macaulay. Then, $\Tr_{\omega_R} \big(f_* \uptau(\omega_S, \kappa_S)\big) = \uptau(\omega_R, \kappa_R)$. In particular, $\Tr_{\omega_R}\: f_* \omega_S \to \omega_R$ is surjective if $R$ is $F$-rational. Conversely, if $\Tr_{\omega_R} \: f_* \omega_S \to \omega_R$ is surjective and $S$ is $F$-rational, then $R$ is $F$-rational.
\end{corollary}
A.~Singh constructed a $\bQ$-Gorenstein $F$-rational singularity whose canonical cover is not $F$-rational \cite{SinghCyclicCoversOfRational}. Therefore, we cannot expect in \autoref{cor.CanonicalModulesTrans} (resp. \autoref{theo.TraceTauSurjective}) that the $F$-injectivity of $R$ (resp.\ the $F$-regularity of $M$) implies the $F$-injectivity of $S$ (resp.\ the $F$-regularity of $f^! M$). 

\subsubsection{Relationship with splinters} \label{subsubsec.RelationshipSplinters}
By plugging in $\sC = \sC_R$ and $M=R$ in \autoref{theo.TraceTauSurjective}, we see why strongly $F$-regular rings are splinters; see \cite{MASplinters, HochsterContractedIdealsFromIntegralExtensions}. Moreover, we may take a splitting of the finite cover $R \subset S$ to be an element of $\uptau(\omega_{S/R}, f^*\sC_R)$ when $R$ is strongly $F$-regular. In other words, while a splinter guarantees a splitting for any finite extension $R \subset S$, a strongly $F$-regular ring guarantees that splitting to be an element of $\uptau(\omega_{S/R}, f^*\sC_R)$. 
\begin{corollary}
Let $R$ be a domain. Then, $R$ is strongly $F$-regular if and only if every ( equivalently some) finite extension $R \subset S$ splits by an element in $\uptau(\omega_{S/R}, f^*\sC_R)$.
\end{corollary}
The $F$-regularity of splinters is to many the main open problem in positive characteristic singularity theory. The answer is known to be affirmative for $\bQ$-Gorenstein rings by \cite{HochsterHunekeTCParameterIdealsAndSplitting, SinghQGorensteinSplinters} yet far open beyond this case. Let us fix a domain $R$ and say that a finite $R$-module $M$ is a \emph{splinter} if $\Tr_M \: f_*f^! M \to M$ is surjective for all covers $f\: \Spec S \to \Spec R$. Thus, \autoref{theo.TraceTauSurjective} says that $F$-regular Cartier modules are splinters. We may wonder about the converse: Is a splinter module $F$-regular with respect to its full Cartier algebra? Of course, the case of interest is $M=R$. We even got the following inclusion of subodules of $M$
\begin{equation} \label{eqn.FundamentalSplinterRegularInclusion}
\uptau(M, \sC_M) \subset \bigcap_{f} \Tr_M\big(f_*f^! M\big),
\end{equation}
where the intersection traverses over all covers $f\: \Spec(S) \to \Spec(R)$. For what (families of) modules is \autoref{eqn.FundamentalSplinterRegularInclusion} an equality? For example, this is known to hold for the class of canonical modules over Cohen--Macaulay rings by \cite{BlickleSchwedeTuckerTestAlterations}, where $F$-rationality of Cohen--Macaulay rings is characterized by the canonical module being a splinter module. Does equality hold in \autoref{eqn.FundamentalSplinterRegularInclusion} for $M= \omega_R(\varepsilon D)$ if $0 < \varepsilon \ll 1$?

\subsection{Non-$F$-pure modules under finite covers} \label{sec.NonFPuremodulesResults} Given a Cartier module $(M,\sC)$, let us write $\upsigma(M,\sC)\coloneqq \underline{M}$ and refer to it as the \emph{non-$F$-pure module} (this terminology is borrowed from the case of non-$F$-pure ideals \cite{FujinoSchwedeTakagiSupplements}). It is natural to ask whether the formula $\Tr_M \big(f_*\upsigma(f^!M,f^*\sC)\big) = \upsigma(M,\sC)$ holds. We readily see that the inclusion ``$\subset$'' holds, whereas the converse requires the surjectivity of $\Tr_M$ to hold.
\begin{proposition} \label{pro.NonFPureModulesTrans}
Work in the setup of \autoref{theo.TraceTauSurjective}. Then, the following inclusion holds
\[
\Tr_M \big(f_*\upsigma(f^!M,f^*\sC)\big) \subset \upsigma(M,\sC)
\]
and equality holds if $\Tr_M \: f_* f^! M \to M$ is surjective.
\end{proposition}
\begin{proof}
We already observed in \autoref{theo.TraceTauSurjective} that $\Tr_M$ is a $\sC$-morphism, which implies
\[
\Tr_M \big(f_\ast f^*\sC_+^{e}f^!M\big) = \sC_+^e \Tr_M\big(f_\ast f^!M\big) \subset \sC_+^e M.
\]
The result follows by taking $e$ sufficiently large and noting that if $\Tr_M$ is surjective the displayed inclusion is an equality.
\end{proof}
\begin{corollary}
Work in the setup of \autoref{cor.CanonicalModulesTrans}. If $\Tr \: f_* \omega_S \to \omega_R$ is surjective, then $\Tr\bigl(f_* \upsigma(\omega_S, \kappa_S)\bigr) = \upsigma(\omega_R,\kappa_R)$. In particular, $R$ is $F$-injective if so is $S$.
\end{corollary}
\begin{proof}
Use \autoref{pro.NonFPureModulesTrans}.
\end{proof}

The following example shows that surjectivity of the trace is necessary for the equality in \autoref{pro.NonFPureModulesTrans} to hold. Compare this to the discussion in \cite[\S 8]{SchwedeTuckerTestIdealFiniteMaps}. 

\begin{example}
Consider the example (a) in \autoref{ex.NecessityHypothesisTransRule}. Note that $R$ and $S$ are $F$-pure Gorenstein local rings; use Fedder's criterion for $R$ \cite{FedderFPureRat}. 
Then, $\upsigma(R,\kappa_R)=R$ and $\upsigma(S,\kappa_S)=S$, yet $\Tr_R(S)=(x,y,z)$. Nevertheless, $\uptau(R,\kappa_R)=(x,y,z)$ and $\uptau(S,\kappa_S)=S$, which verifies \autoref{theo.TraceTauSurjective}.
\end{example}

\subsection{Test ideals along closed subschemes under finite covers} \label{sec.AdjointIdeals}
In this section, we explain how test ideals along closed subschemes---as treated in \cite[\S 3.1]{Smolkinphdthesis}, \cite[\S4]{SmolkinSubadditivity}---behave under finite covers. These were introduced as positive characteristic analogs of adjoint ideals in characteristic zero by Takagi \cite{TakagiPLTAdjoint,TakagiHigherDimensionalAdjoint}.

Fix a Cartier subalgebra $\sC \subset \sC_R$ and a radical $\sC$-compatible ideal $\mathfrak{a}$ with irredundant primary decomposition  $\mathfrak{a} = \p_1 \cap \cdots \cap \p_r$. All $\p_i$ are prime since $\mathfrak{a}$ is radical and further $\sC$-compatible by \cite[Corollary 4.8]{SchwedeCentersOfFPurity}. Set $P = \bigcup_{i=1}^r \p_i$. The \emph{adjoint ideal} $\uptau_{\mathfrak{a}}(R,\sC)$ is defined as the smallest Cartier $\sC$-submodule of $R$ not contained in $P$. The Cartier algebra $\sC$ is called \emph{non-degenerate with respect to $\mathfrak{a}$} if there is $e > 0$ such that $\sC_e(R)$ is not contained in $P$. By prime avoidance, this is equivalent to $\sC_e(R) \nsubset \p_i$ for all $i$. In other words, $(R,\sC)$ is non-degenerate with respect to $\mathfrak{a}$ if and only if $(R,\sC)$ is non-degenerate with respect to $\mathfrak{p_i}$ for all $i$,  \ie $\p_i$ is a center of $F$-purity of $(R,\sC)$ for all $i$. Smolkin proved that, for $\mathfrak{a}$ prime, $\uptau_\mathfrak{a}(R, \sC)$ exists if $\sC$ is non-degenerate, assuming $R$ is a domain (as well as noetherian and $F$-finite; as we do)---see \cite[\S4.1]{SmolkinSubadditivity} and \cite[\S3.1]{Smolkinphdthesis}. Smolkin explained to us that a similar theory can be developed assuming $\mathfrak{a}$ radical and $R$ noetherian and $F$-finite (\cite{SmolkinGeneralAdjointIdeals}). Nevertheless, we will be able to reduce to the case $\mathfrak{a}$ is prime using \autoref{lem.ReductionOneComponent} below but we will assume that $R$ is a domain to cite \cite[\S3.1]{Smolkinphdthesis}.

One says that a pair $(R,\sC)$ is \emph{purely $F$-regular along a radical ideal $\mathfrak{a} \subset R$} if $\uptau_{\mathfrak{a}}(R,\sC)=R$ (implicitly assuming that $\uptau_{\mathfrak{a}}(R,\sC)$ exists). For these ideals, we have the following.

\begin{lemma} \label{lem.ReductionOneComponent}
Let $R$ be a domain and $\mathfrak{a} \subset R$ be a radical ideal that is a Cartier $\sC$-submodule. Let $\p_1, \ldots , \p_k$ be the minimal primes of $\mathfrak{a}$. Assume that $(R,\sC)$ is non-degenerate with respect to $\fra$. Then, $\uptau_\mathfrak{a}(R, \sC) = \sum_{i=1}^k \uptau_{\p_i}(R, \sC)$.
\end{lemma}
\begin{proof}
First, note that $\uptau_{\p_i}(R, \sC)$ exists by \cite[Section 3.1.1]{Smolkinphdthesis}. Write $P = \bigcup_i \p_i$. By definition, $\uptau_\mathfrak{a}(R, \sC)$ is the smallest $\sC$-compatible ideal not contained in $P$. If $I$ is any $\sC$-compatible ideal not contained in $P$, then it is not contained in any $\p_i$ whence $\uptau_{\p_i} \subset I$ for all $i$.  To finish the proof, assume that $\sum_i \uptau_{\p_i}(R, \sC) \subset P$. By prime avoidance, the sum is then contained in, say, $\p_1$. Then, $\uptau_{\p_1}(R, \sC) \subset \sum_i \uptau_{\p_i}(R, \sC) \subset \p_1$, which is a contradiction. Thus, $\sum_{i=1}^k \uptau_{\p_i}(R, \sC)$ is the smallest $\sC$-compatible ideal not contained in $P$ and so coincides with $\uptau_\mathfrak{a}(R, \sC)$.
\end{proof}

\begin{theorem}
\label{theo.adjointidealtransformation}
Let $f \: \Spec S \to \Spec R$ be a cover between integral schemes. Let $\mathfrak{a} \subsetneq R$ be a radical ideal with minimal primes $\mathfrak{p}_1,\ldots, \mathfrak{p}_k$, and set $\mathfrak{b}\coloneqq \sqrt{\mathfrak{a}S}$. Let $\sigma_R\:  S \to \omega_{S/R} $ be a generic isomorphism such that $T = \sigma_R(1)$ satisfies
\begin{equation} 
\label{eqn.TraceAndidealsPre}
\sqrt{\p_i S} = f^! \mathfrak{p}_i :_S T, \quad \text{for all } i=1,\ldots,k.
\end{equation}
Let $\sC$ be a Cartier $R$-algebra acting on $R$ such that $R$ is $T$-transposable and $\mathfrak{a}\subset R$ is a Cartier submodule. The following statements hold.
\begin{enumerate}
    \item The ideal $\mathfrak{b} \subset S$ is a Cartier $f^*\sC$-submodule.
    \item The pair $(R,\sC)$ is non-degenerate with respect to $\fra$ if and only if $(S,f^*\sC)$ is non-degenerate with respect to $\mathfrak{b}$. In that case, the following equality holds %for any such prime $\q$\begin{equation} \label{eqn.AdjointIdealsFIniteMorphism}T\big(f_\ast \uptau_{\mathfrak{q}}(S, f^* \sC) \big) = \uptau_{\mathfrak{p_i}}(R, \sC)\end{equation}
%\item Assume that $(R,\sC)$ is non-degenerate with respect to $\p_i$ for some fixed $i$. Then, $(S, f^*\sC)$ is non-degenerate with respect to $\sqrt{\p_i S}$ if either there is only one prime lying over $\p_i$ or $T$ satisfies $T \circ F = F \circ T$. \item If $(S, f^*\sC)$ is non-degenerate with respect to $\sqrt{\p_i S}$ for all $i=1,\ldots, k$, then 
\begin{equation} \label{eqn.AdjointIdealsFIniteMorphismBIG}
T\big(f_\ast \uptau_{\mathfrak{b}}(S, f^* \sC) \big) = \uptau_{\mathfrak{a}}(R, \sC).
\end{equation}
\item If $\fra=\p$ is a prime ideal, then $(R,\sC)$ is non-degenerate with respect to $\fra$ if and only if $(S,f^*\sC)$ is non-degenerate with respect to \emph{all} prime ideals $\mathfrak{q} \subset S$ lying over $\p$. Further, for \emph{any} such prime $\q$ the following equality holds
\begin{equation} \label{eqn.AdjointIdealsFIniteMorphism}
    T\big(f_\ast \uptau_{\mathfrak{q}}(S, f^* \sC) \big) = \uptau_{\mathfrak{p}}(R, \sC).
    \end{equation}
\item If $(R,\sC)$ is purely $F$-regular along $\mathfrak{a}$, then $T$ is surjective. The converse holds provided that $(S,f^*\sC)$ is purely $F$-regular along $\mathfrak{b}$.
\end{enumerate}
\end{theorem}
\begin{proof}
For any ideal $I \subset R$, the ideal $f^!I :_S T$ equals $\{s\in S \mid T(sS) \subset I\}$ and is the largest ideal $J\subset S$ such that $T(J) \subset I$. Further,
\[(f^!I :_S T) \cap (f^!J :_S T) = (f^!I \cap f^! J ):_S T = f^!(I \cap J) :_S T\]for any two ideals $I,J \subset R$. In particular, \autoref{eqn.TraceAndidealsPre} implies
\begin{equation} \label{eqn.TraceAndideals}    
\mathfrak{b} =  f^! \mathfrak{a} :_S T.
\end{equation}

For notation ease, we assume $\sC \subset \sC_R$ and $f^*\sC \subset \sC_S$; see \autoref{rem.RemarkRelGor}. 

\emph{Proof of (a):} Next, we explain why $\mathfrak{b} \subset S$ is a $f^*\sC$-submodule. It suffices to show $\varphi^{\top}(F^e_* \mathfrak{b}) \subset \mathfrak{b}$ for all $\varphi \in \sC_e$. Note that
\[
T \big( \varphi^{\top}(F^e_* \mathfrak{b}) \big) = \varphi\big(F^e_* T(\mathfrak{b})\big) \subset \varphi(F^e_* \mathfrak{a}) \subset \mathfrak{a},
\]
where the inclusion $T(\mathfrak{b}) \subset \mathfrak{a}$ is a consequence of \autoref{eqn.TraceAndideals}, and the last one follows from $\mathfrak{a} \subset R$ being a $\sC$-submodule. Thus, $\varphi^{\top}(F^e_* \mathfrak{b}) \subset T^{-1}(\mathfrak{a})$. Since $\varphi^{\top}(F^e_* \mathfrak{b})$ is an ideal of $S$, this implies $\varphi^{\top}(F^e_* \mathfrak{b}) \subset f^! \mathfrak{a} :_S T$, and so $\varphi^{\top}(F^e_* \mathfrak{b}) \subset \mathfrak{b}$ by \autoref{eqn.TraceAndideals}.

\emph{Proof of (b) and (c):} We prove next the statements regarding non-degeneracy. We may assume that $\mathfrak{a} = \mathfrak{p}$ is prime; see \autoref{lem.ReductionOneComponent}. Let $\q_1, \ldots , \q_n$ be the primes of $S$ lying over $\p$, which are the minimal primes of $\mathfrak{b} = \sqrt{\p S}=\bigcap_i \q_i$ and are all compatible prime ideals of $(S,f^*\sC)$ by part (a). We show that $(R,\sC)$ is degenerate with respect to $\p$ if and only if $(S,f^*\sC)$ is degenerate with respect to $\sqrt{\p S}$ if and only if $(S,f^*\sC)$ is degenerate with respect to $\q_i$ \emph{for all} $i\in \{1,\ldots,n\}$. 

Suppose that $(R,\sC)$ is degenerate with respect to $\p$, that is: $\varphi(F^e_*R) \subset \p$ for all $\varphi \in \sC_e$. In particular, $T\big(\varphi^{\top}(F^e_* S) \big)= \varphi\big(F^e_*T(S)\big) \subset \p$ for all $\varphi \in \sC_e$. In other words, $\varphi^{\top}(F^e_* S) \subset T^{-1}(\p)$ for all $\varphi \in \sC_e$. Since $\varphi^{\top}(F^e_* S) \subset S$ is an ideal, this implies that $\varphi^{\top}(F^e_* S)\subset f^! \p : T = \sqrt{\p S}=\bigcap_i\q_i$ for all $\varphi \in \sC_e$. Therefore, since $f^*\sC$ is generated as a right $S$-module by maps $\varphi^{\top}$, we conclude that $(S,f^*\sC)$ is degenerate with respect to \emph{all} $\q_i$. 

Suppose that $(S,f^*\sC)$ is degenerate with respect to $\sqrt{\p S} = f^!\p:T$. That is, $(S,f^*\sC)$ is degenerate with respect to \emph{some} prime lying over $\p$, say $\q_{i_0}$. In particular, $\varphi^{\top}(F^e_*S) \subset \q_{i_0}$ for all $\varphi \in \sC_e$. Therefore, 
\[
\varphi^{\top}\left(F^e_* \bigcap_{i \neq i_0}\q_i\right)\subset \q_{i_0} \cap \bigcap_{i \neq i_0}\q_i = \sqrt{\p S} = f^!\p:T
\]
for all $\varphi \in \sC_e$. For notation ease, let us set $S_0 \coloneqq \bigcap_{i \neq i_0}\q_i$.
%We may assume, without lost of generality, that $i=1$. By the Chinese reminder theorem, for every $s \in S$, there is $s'\in S$ such that $s-s' \in \q_1$ and $s' \in \q_i$ for all $i=2,\ldots,k$. Further, if $s''$ is any other such element, then $s'-s'' \in \sqrt{\qS}$.  
In this way, for all $\varphi \in \sC_e$, we have $T\big(\varphi^{\top}(F^e_*S_0)\big) \subset \p$ and so $\varphi\big(F^e_*T(S_0)\big) \subset \p$ as $\varphi \circ F^e_* T= T \circ \varphi^{\top}$.  However, $T(S_0) \not\subset \p$ as $f^!\p:T = \sqrt{\p S} \subsetneq S_0$. Thus, we can take an element $r \in T(S_0)\smallsetminus \p$ so that $r^qR \subset T(S_0)$. Then, for all $\varphi \in \sC_e$, we have: 
\[
r\varphi(F^e_*R)=\varphi(F^e_*r^qR)\subset \varphi\big(F^e_*T(S_0)\big) \subset \p,\]
and so $\varphi(F^e_*R) \subset \p$ as $r\notin \p$. Therefore, $(R,\sC)$ is degenerate with respect to $\p$.

With the above in place, we explain why \autoref{eqn.AdjointIdealsFIniteMorphism} and so \autoref{eqn.AdjointIdealsFIniteMorphismBIG} hold. In fact, we show that 
\[
T\big(f_\ast \uptau_{\mathfrak{q}}(S, f^* \sC) \big) = \uptau_{\mathfrak{p}}(R, \sC)
\] 
holds assuming both adjoint test ideals exist, where $\mathfrak{p}$ is prime and $\q \subset S$ is any prime ideal lying over $\p$. We start with the containment ``$\supset$.'' It suffices to show that $T\big(f_\ast \uptau_{\mathfrak{q}}(S, f^* \sC) \big)$ is $\varphi$-compatible for all $\varphi \in \sC_e$ and that it is not contained in $\p$. The compatibility follows once again using that $\uptau_{\mathfrak{q}}(S, f^* \sC) $ is $\varphi^{\top}$-compatible for all $\varphi \in \sC_e$ and employing $T \circ \varphi^{\top} = \varphi \circ T$. On the other hand, $T\big(f_\ast \uptau_{\mathfrak{q}}(S, f^* \sC) \big)$ cannot be contained in $\p$ because else $\uptau_{\mathfrak{q}}(S, f^*\sC) \subset \sqrt{\p S} \subset \q$ by \autoref{eqn.TraceAndidealsPre}, which contradicts its non-degeneracy with respect to $\mathfrak{q}$.

Conversely, for the inclusion ``$\subset$,'' we use the description of adjoint ideals in terms of test elements; see \cite[3.1.11, 3.1.12, 3.1.16]{Smolkinphdthesis}.
Let $c \in \uptau_\q(S, f^\ast \sC) \smallsetminus \q$ so that $c$ is an $f^\ast \sC$-test element along $\q$ using the terminology of \cite[Definition 3.1.15]{Smolkinphdthesis}. Hence, we can write
\[
\uptau_{\mathfrak{q}}(S, f^* \sC) = \sum_e \sum_{\varphi \in \sC_e} \varphi^{\top}(F^e_*cS).
\]Hitting this equality by $T$ and using $T \circ \varphi^\top = \varphi \circ F^e_* T$, we see that it suffices to show $T(cS) \subset \uptau_{\p}(R, \sC)$. Since $c \notin \q$, we have $T(c \cdot S) \not\subset \p$  by \autoref{eqn.TraceAndidealsPre}. In particular, we find $s \in S$ such that $T(cs) \notin \p$. Take $s'\in S$, $r \in R \smallsetminus \p$ arbitrary. First of all, $r \notin \q$ as $\q$ contracts to $\p$ along $R \to S$. Then, there are $s_1,\ldots, s_n \in S$ and $\varphi_1, \ldots, \varphi_n \in \sC_{e}$; for some $e>0$, such that $\sum_i \varphi_i^{\top}\big(F^e_*(s_i\cdot r)\big) = s' c$, for $c$ is an $f^*\sC$-test element along $\q$. In particular, $\sum_i \varphi_i\big(F^{e}_*T(s_i)r\big)=T(s'c)$. In other words, for all $s' \in S$, $r\in R \smallsetminus \p$, there exists $\varphi \in \sC_e$ for some $e>0$ such that $\varphi(F^e_*r)=T(s'c)$. Therefore, specializing to $s'=s$, we have that $T(sc)$ is a $\sC$-test element along $\p$. Hence,
\begin{equation} \label{eqn.AdjointIdealDonstairs} \uptau_{\p}(R,\sC) = \sum_e \sum_{\varphi \in \sC_e} \varphi\big(F^e_* T(s c)R\big).\end{equation}
Now, let $s' \in S$ be arbitrary, and let $r=T(sc)$. By our previous observation, we may find $\varphi \in \sC_e$ such that $\varphi(F^e_*r)=T(s'c)$, and so $T(s'c) \in \uptau_{\p}(R,\sC)$ as a consequence of \autoref{eqn.AdjointIdealDonstairs}. Thus, we have established $T(cS) \subset \uptau_\p(R, \sC)$, as desired.

\emph{Proof of (d):} This follows directly from \autoref{eqn.AdjointIdealsFIniteMorphismBIG}.
\end{proof}

%\begin{remark}The conditions in \autoref{theo.adjointidealtransformation} guaranteeing that $(S,f^*\sC)$ is non-degenerate with respect to every prime lying over one of the $\p_i$ are natural. Indeed, if the cover is purely inseparable then there is at most one prime lying over any other prime. If the cover is separable, then the generic trace map commutes with Frobenius; see \cite[Lemma 6]{SpeyerFrobeniusSplit}.  \end{remark}

As an application of \autoref{theo.adjointidealtransformation} and the restriction theorem for adjoint ideals, we have: 

\begin{corollary}
Work in the setup of \autoref{theo.adjointidealtransformation}. If $\mathfrak{a}=\p$ and $\mathfrak{b}=\q$ are prime, then
\[
\overline{T}\Big(\overline{f}_*\uptau\big(S/\q, \overline{f}^*\overline{\sC}\big)\Big) = \uptau\big(R/\p, \overline{\sC}\big) 
\]
where $\overline{T} \: S/\q \to R/\p$ is the restriction of $T \: S \to R$ given by the inclusion $T(\q) \subset \p$, and $\overline{f} : \Spec (S/\q) \to \Spec (R / \p) $ is the spectrum of the induced homomorphism $R/\p \to S/\q$.
\end{corollary}
\begin{proof}
Reduce \autoref{eqn.AdjointIdealsFIniteMorphism} modulo $\p$ and use \cite[Proposition 3.1.14]{Smolkinphdthesis} as well as $\overline{f}^*\overline{\sC} = \overline{f^*\sC}$, which follows as in \autoref{cla.TranspositionRestriction}.
\end{proof}

\begin{remark}
Let us discuss the meaning of \autoref{eqn.TraceAndidealsPre} for small heights of $\p$ when $T = \Tr_{S/R}$ and $R \subset S$ is separable extension of normal domains. We always have $\sqrt{\mathfrak{p} S} \subset f^! \mathfrak{p} :_S \Tr_{S/R}$ since $\Tr_{S/R}\big(\sqrt{\mathfrak{p}S}\big) \subset \mathfrak{p}$; see \cite[Lemma 9]{SpeyerFrobeniusSplit}. If $\height \p =0$, then $\p=0$ and $\sqrt{\p S}=0$ as we assume $R$ and $S$ to be integral. Then, \autoref{eqn.TraceAndidealsPre} means that $\sigma_R \: S \to \omega_{S/R}$ is injective. This, however, is a consequence of $\sigma_R$ being a generic isomorphism and $S$ satisfying $\mathbf{S}_1$. Hence, in height $0$, \autoref{eqn.AdjointIdealsFIniteMorphismBIG} recovers \autoref{eqn.TransformationRuleTestIdelasRingsTrans} in the integral case. Let us suppose now that $\height \p = 1 $. We explain what $\sqrt{\mathfrak{p} S} \supset f^! \mathfrak{p} :_S \Tr_{S/R}$ means in terms of divisors. Note that $\height \mathfrak{p} = 1$ amounts to saying that $\mathfrak{p} = R(-P)$ with $P$ is an effective reduced divisor on $\Spec R$. Similarly, $\sqrt{\mathfrak{p} S}=S(-Q)$ for some effective reduced divisor $Q$ on $\Spec S$. Thus, $s \in f^! \mathfrak{p} :_S \Tr_{S/R}$ if and only if
\[
\Tr_{S/R} \cdot s \in f^! \mathfrak{p} = \Hom_R\big(S, R(-P)\big) = \Hom_R\big(S \otimes_R R(P), R\big) = \Hom_R\big(S(f^*P),R\big),
\]
which means that $S(f^*P) \subset S(\Ram + \Div s)$, equivalently $s \in S(\Ram - f^* P)=S(-P^*)$ where $P^* \coloneqq f^*P - \Ram$. Hence, $\mathfrak{q} \supset f^! \mathfrak{p} :_S \Tr_{S/R}$ is equivalent to $P^* \geq Q$. The same applies for general effective reduced divisors. Concretely, let $\mathfrak{a}\coloneqq R(-D)$ and $\mathfrak{b}\coloneqq S(-E)$ for reduced effective divisors $D$ and $E$ on $\Spec R$ and $\Spec S$; respectively, such that $E=(f^{-1}D)_{\mathrm{red}}$. Then, $\mathfrak{b} = f^! \mathfrak{a} :_S \Tr_{S/R}$ is equivalent to $f^*D - \Ram = E$. Using the terminology \cite[Definition 3.9]{SchwedeTuckerTestIdealFiniteMaps}, such divisorial equality is equivalent to having that for all height-$1$ prime ideals $\q' \subset S$ the DVRs extension $R_{\mathfrak{q'} \cap R}\subset S_{\mathfrak{q'}}$ is: \'etale if $\mathfrak{q'}$ does not support $E$ and tamely ramified otherwise \cite[IV, Proposition 2.2]{Hartshorne}; see \cite[Remark 2.9]{Carvajalphdthesis}, \cite[Remark 4.6]{SchwedeTuckerTestIdealFiniteMaps}.
\end{remark}

\begin{remark}
In the forthcoming preprint \cite{CarvajalRojasFayolle}, Anne Fayolle and the first named author expand upon the ideas behind \autoref{theo.adjointidealtransformation} and describe how centers of $F$-purity behave under transpositions and finite covers. Moreover, they clarify further the prominent role played by the condition \autoref{eqn.TraceAndidealsPre} in their theory. In particular, they provide better conceptual interpretations for this condition as well as numerical rank-conditions it is equivalent to. For instance, if $\sigma$ is an isomorphism over $\p$, then \autoref{eqn.TraceAndidealsPre} is equivalent to the residual degree $\sum_{i=1}^k \big[\kappa(\q_i):\kappa(\p)\big]$ being equal to the free rank of $S_{\p}$ as an $R_{\p}$-module, which will let us see that \autoref{theo.adjointidealtransformation} holds for $T=\Tr$ if $f$ is 
\'etale over $\p$.
\end{remark}

\section{Schwede--Tucker's transposability criterion revisited} \label{sec.SchwedeTuckerFaithfullyFlat}
Let $f \: Y \to X$ be a finite cover of normal integral schemes and $T$ be a nonzero global section of $\omega_{Y/X}$. Schwede--Tucker's transposability criterion establishes that $\sC^{\top}_{e,X} = \bigl\{\varphi \in \sC_{e,X} \mid f^* \Delta_{\varphi} \geq \Ram_T \bigr\}$. We aim to describe $\sC^{\top}_{e,X}$ in terms of divisors on $X$ rather than divisors on $Y$. To this end, we need the following facts regarding norm functions.
\subsection{Norm functions} For a detailed exposition on norms, see \cite[\href{https://stacks.math.columbia.edu/tag/0BCX}{Tag 0BCX}]{stacks-project}. A multiplicative function $\Norm_f \: f_* \sO_Y \to \sO_X$ is a \emph{norm of degree $d$} if  $\Norm_f \circ f^{\#}\:\sO_X \to f_*\sO_Y \to \sO_X$ is given by raising local sections to the $d$-th power and whenever $v \in \sO_Y\bigl(f^{-1} (U)\bigr)$ vanishes at $y \in f^{-1} (U)$ so does $N_f(v)$ at $f(y) \in U$. Let $K$ and $L$ be the fields of functions of $X$ and $Y$ respectively. Then, there exists a norm $\Norm_f \: f_* \sO_Y \to \sO_X$ of degree $n \coloneqq [L:K]$. Indeed, we may take $\Norm_f$ to be the integral restriction of the field extension norm $\Norm_{L/K} \: L \to K$ \cite[\href{https://stacks.math.columbia.edu/tag/0BD3}{Tag 0BD3}]{stacks-project}. We often use the notation $\Norm_{f}=\Norm_{Y/X}$ in that case.

Applying $H^1(X,-)$ to $\Norm_f \: f_*\sO_Y^{\times} \to \sO_X^{\times}$ gives a homomorphism $\Norm_f\:\Pic Y \to \Pic X$ such that the composition $\Norm_f \circ f^* \:\Pic X \to \Pic Y \to \Pic X$ is multiplication-by-$n$ (in additive notation); see \cite[\href{https://stacks.math.columbia.edu/tag/0BCY}{Tag 0BCY}]{stacks-project}. This shows that if $\sL$ is nontorsion, then so is $f^* \sL$. Moreover, if $\sL \cong \sO_Y(D) \subset L$ for some Cartier divisor $D$, then $\Norm_f (\sL) $ is realized as the rank $1$ subsheaf of $K$ given by $ \Norm_{L/K}\bigl(\sO_Y(D) \bigr)$. We define $\Norm_{f}(D)$ to be the Cartier divisor on $X$ determined by the equality $ \Norm_{L/K}\bigl(\sO_Y(D) \bigr) = \sO_X\bigl(\Norm_f(D)\bigr)$. This extends to a homomorphism $\Norm_f \: \Cl Y \to \Cl X$ as follows: $\Norm_{L/K}\bigl(\sO_Y(D) \bigr) = \sO_X\bigl(\Norm_f(D)\bigr)$ for all Weil divisor $D$ on $Y$. It is worth noticing that $\Norm_f \: \DIV Y \to \DIV X$ is given by $\Norm_f \: \Pic f^{-1} (X_{\mathrm{reg}}) \to \Pic X_{\mathrm{reg}}$. Finally, we define $\Norm_f$ on $\bQ$-divisors by $\Norm_f \otimes_{\bZ} \bQ$. 

\begin{lemma} \label{lem.EffectivenessAndNorms}
Let $f \: Y \to X$ be a finite cover of normal integral schemes, let $\Delta$ be a $\bQ$-divisor on $Y$. Then, $\Norm_f (\Delta) \geq 0$ if $\Delta \geq 0$.
\end{lemma}
\begin{proof}
We may assume that $\Delta$ is integral. Effectiveness of a divisor $D$ on an integral normal scheme $S$ is equivalent to the inclusion $\sO_S \subset \sO_S(D)$ in $K(S)$.  Thus, if $\O_Y \subset \sO_Y(\Delta)$ then $\sO_X \subset \sO_X\bigl(\Norm_f(\Delta)\bigr)$ as $\Norm_f(\sO_Y) = \sO_X$ (since $\sO_Y = f^*\sO_X$, the norm of $\sO_Y$ is $\sO_X^n = \sO_X$).
\end{proof}

\begin{definition} \label{def.relativelytorsion}With notation as in \autoref{lem.EffectivenessAndNorms}, a divisor $D$ on $Y$ is $f$-\emph{torsion} if $mD = f^* D'$ for some $m \neq 0$ and some divisor $D'$ on $X$. \end{definition}

\begin{proposition} \label{pro.RelativeTorsion}With notation as in \autoref{lem.EffectivenessAndNorms}, a divisor $D$ on $Y$ is $f$-torsion if and only if $f^\ast \Norm_f(D) = n \cdot D$.\end{proposition}\begin{proof}Clearly, if $n \cdot D = f^\ast \Norm_f(D)$ then $D$ is $f$-torsion. Conversely, if $m D  = f^\ast D'$ for some $m \neq 0$, then applying $\Norm_f$ gives $m N_f(D) = n D'$. Pulling back yields $m f^\ast N_f(D) = n f^\ast D' = mn D$, and dividing by $m$ gives $f^\ast \Norm_f(D) = n \cdot D$.\end{proof}

\begin{lemma} \label{lem.EffectivenessAndNormsII}
With notation as in \autoref{lem.EffectivenessAndNorms}, let $D$ be an effective divisor on $Y$. Then, $f^*\Norm_f (D) \geq k \cdot D $ for some integer $1\leq k\leq n$. Further, if $f$ is generically Galois and $\sigma\big(\sO_{Y}(D)\big) \subset \sO_Y(D)$ for all $\sigma \in \Gal(L/K)$, we may take $k=n$.
\end{lemma}
\begin{proof}
Since the effectiveness and triviality of a divisor can be checked at finitely many codimension $1$ points, we may assume that $X = \Spec R$ is the spectrum of a DVR and $Y=\Spec S$ is the spectrum of a semi-local Dedekind domain and so a PID. Indeed, once we have  $1 \leq k_P \leq n$ that works for a prime component $P$ of $D$, we may take $k = \min_{P} \{k_P\}$. In particular, we may write $D =\Div_S s $ for some $s \in S^{\times}$. Further, $\Norm_f (D) = \Div_R \Norm_{L/K}(s)$. By setting $r\coloneqq \Norm_{L/K}(s) \in R^{\times}$, we are required to prove that $f^*\Div_R r = \Div_S r \geq k \cdot \Div_S s = \Div_S s^k$ for some integer $1\leq k \leq n$. In other words, we must prove that $r/s^k \in L$ belongs to $S$ for some integer $k \in [1,n]$ (provided that $s \in S$). Note that if $t^m+r_{m-1}t^{m-1}+\cdots +r_1t +r_0 \in R[t]$ is the minimal polynomial of $s$ then $r=(-1)^n r_0^{n/m}$, where $n = [L:K]$. Indeed,
\[r = \Norm_{L/K}(s) = \Norm_{K(s)/K}\bigl(\Norm_{L/K(s)}(s)\bigr) = \Norm_{K(s)/K}\bigl(s^{n/m}\bigr) = \Norm_{K(s)/K}(s)^{n/m}
,\]and $\Norm_{K(s)/K}(s)=(-1)^m r_0$, $n/m=\bigl[L:K(s)\bigr]$. Next, use the equation
\[
s^m+r_{m-1}s^{m-1}+\cdots +r_1 s  +r_0 = 0
\]
to find $s' \in R[s]$ such that $s^l s' = r_0$ with $l$ as large as possible. Also, notice that $1 \leq l \leq m \leq n$. Hence, we conclude that if $k=l(n/m)$ then $s^k$ divides $r$; as needed.
\end{proof}

\subsection{Transposability along finite covers and branching divisors}
As in the previous section $f \: Y \to X$ is a finite cover of normal integral schemes and $T$ a nonzero global section of $\omega_{Y/X}$. We think of $\Ram_T \sim K_{Y/X}$ as an effective divisor measuring the failure of $T$ in generating $\omega_{Y/X}$ in codimension $1$ and call it the \emph{ramification divisor of $T$}. In fact, $\sO_Y(\Ram_T) \to \omega_{Y/X}$, $v \mapsto T(v \cdot -)$, is an isomorphism of $\sO_Y$-modules. By using norms, we may define a divisor on $X$ measuring such failure. 

\begin{definition} \label{def.BranchingDivisor}
The \emph{branching divisor of $T$} is $\Branch_T = \Norm_{Y/X}(\Ram_T)$. If $f\:Y \to X$ is separable and $T= \Tr_{Y/X}$, we write $\mathfrak{R}=\mathfrak{R}_{Y/X}=\Ram_T$, $\mathfrak{B}=\mathfrak{B}_{Y/X}=\Branch_T$ and refer to them as the \emph{ramification divisor} and \emph{branching divisor} of $f$; respectively.
\end{definition}

\begin{lemma} \label{lem.SupportOfBranch}
With notation as in \autoref{def.BranchingDivisor}, $\Branch_T$ is supported at prime divisors whose generic points are images of the generic point of some prime divisor supporting $\Ram_T$.
\end{lemma}
\begin{proof}
We may assume that $X=\Spec R$ is the spectrum of a DVR and $Y=\Spec S$ is the spectrum of a semi-local Dedekind domain, so a PID. Since $R$ and $S$ are both Gorenstein, $\omega_{S/R}$ is freely generated as an $S$-module by some $R$-linear map $\Phi \: S \to R$. Then, there exists a unique $s \in S$ such that $T = \Phi \cdot s$ and so $\Ram_T = \Div_S s$ and $\Branch_T = \Div_R \Norm_{S/R} (s)$. The result then follows by observing that $s$ is a unit if and only if so is $\Norm_{S/R} (s)$. The direction ``$\Rightarrow$'' is clear. The converse follows from $\Norm_{S/R} (s)$ being the determinant of the multiplication-by-$s$ $R$-linear map and its surjectivity implying that $s$ is a unit.    
\end{proof}

\begin{proposition}
With notation as in \autoref{def.BranchingDivisor}, $\Branch_T$ is Cartier if $f$ is flat.
\end{proposition}
\begin{proof}
Let $\{U_i\}$ be an open covering of $X$ trivializing $f_*\sO_Y$ on $X$. Setting $V_i \coloneqq f^{-1}(U_i)$, then $f_* \sO_{V_i}$ is a free $\sO_{U_i}$-module. Let $\delta_i \in \sO_{U_i}$ be the determinant of the $\sO_{U_i}$-bilinear form $T(-\cdot -)$ on $f_* \sO_{V_i}$. Then, $(U_i,\delta_i)$ defines a locally principal ideal sheaf $\mathfrak{D}_T \subset \sO_X$. We claim that $\sO_X(\Branch_T)=\mathfrak{D}_T^{-1}$. This can be checked at codimension $1$ points, and so we may work in the setup of the proof of \autoref{lem.SupportOfBranch}. Let $s_i$ be a free basis for $S$ as an $R$-module, and let $\delta \coloneqq \det \bigl(T(s_i\cdot s_j)\bigr) \in R$. Then, the result amounts to the equality of ideals $(\delta)=\bigl(\Norm_{S/R}(s)\bigr)$. This follows from the equality $\delta = \Norm_{S/R}(s) \cdot \det \bigl(\Phi(s_i\cdot s_j)\bigr) $ and by noticing that $\det \bigl(\Phi(s_i\cdot s_j)\bigr)$ is a unit because $\Phi$ is a free generator of $\omega_{S/R}$.
\end{proof}

\begin{proposition} \label{prop.GaloisGivesRelTorsioness}
With notation as in \autoref{def.BranchingDivisor}, $\mathfrak{R}_{Y/X}$ is $f$-torsion if $f$ is generically Galois.
\end{proposition}
\begin{proof}
Work in the setup of the proof of \autoref{lem.SupportOfBranch}. The statement then reduces to the equality $(s^n)=\big(\Norm_{S/R}(s)\big)$ of ideals in $S$ if $L/K$ is Galois. By \autoref{lem.EffectivenessAndNormsII} and its proof, $\Norm_{S/R}(s) \in (s^n)$ unconditionally. We suppose that $L/K$ is Galois and show that $s^n \in \big(\Norm_{S/R}(s)\big)$. By \autoref{pro.RelativeTorsion}, it suffices to prove that $s \in \sqrt{(\Norm_{S/R}(s))}$, for which it is enough to show that $s$ belongs to all the ramification primes (if any). This is granted by the Galois symmetry and $\Tr = \sum_{\sigma \in \Gal(L/K)}\sigma$.
\end{proof}

\begin{remark}
We cannot expect the converse of \autoref{prop.GaloisGivesRelTorsioness} to hold. Indeed, if $f$ is quasi-\'etale then $\mathfrak{R}_{Y/X}$ is trivially $f$-torsion yet this has no bearing on $f$ being Galois.
\end{remark}

As a direct application of \autoref{lem.EffectivenessAndNorms} and \autoref{lem.EffectivenessAndNormsII}, \autoref{thm.TransposabilityCriterion} translates to:

\begin{theorem} \label{thm.TransposabilityFaithfullyFlat}
With notation as in \autoref{def.BranchingDivisor}, there exists a rational number $1\leq c \leq n$ such that for all divisors $D$ the following inclusions of Cartier algebras hold:
\[
\sC_{\sO_X(D)}^{c \cdot \Delta} \subset \sC_{\sO_X(D)}^{\top} \subset \sC_{\sO_X(D)}^{\Delta},
\]
where $\Delta \coloneqq \frac{1}{n} \cdot \Branch_T$. Furthermore, we may take $c=1$ if $\Ram_T$ is $f$-torsion.
\end{theorem}
\begin{proof}
Work in the affine case $X=\Spec R$. For $\sC_{R(D)}^{\top} \subset \sC_{R(D)}^{\Delta}$, we must show that, for nonzero $\varphi \: F^e_* R(D) \to R(D)$, $f^*\Delta_{\varphi} - \Ram_T \geq 0$ implies $\Delta_{\varphi} - \Delta \geq 0$. This follows from \autoref{lem.EffectivenessAndNorms} by noticing that
\[
0 \leq \Norm_f(f^* \Delta_{\varphi} - \Ram_T) = \Norm_f(f^* \Delta_{\varphi}) - \Norm_f(\Ram_T) = n \cdot \Delta_{\varphi} - n \cdot \Delta,
\]
and dividing by $n$. Apply \autoref{lem.EffectivenessAndNormsII} to $\Ram_T$ to get $1\leq k \leq n$ and set $c \coloneqq n/k \in \bQ \cap [1,n]$ (further, $k$ can be taken to be $n$ if $\Ram_T$ is $f$-torsion and so $c=1$). To have $\sC_{R(D)}^{c \cdot \Delta} \subset \sC_{R(D)}^{\top}$, we must show that, for nonzero $\varphi \: F^e_* R(D) \to R(D)$, if $\Delta_{\varphi} \geq c \cdot \Delta$ then $f^*\Delta_{\varphi} \geq \Ram_T$. This follows from:
\[
f^* \Delta_{\varphi} \geq  f^* (c \cdot \Delta) = \frac{c}{n} \cdot f^* \Norm_f(\Ram_T) \geq \frac{ck}{n} \cdot \Ram_T = \Ram_T.
\]
\end{proof}

Let $(R,\fram,\kay,K)$ be a Cohen--Macaulay normal complete domain of dimension $d$ and write $X=\Spec R$. By the Cohen--Gabber theorem (see \cite{KuranoShimomotoElementaryProof}), $R$ admits a generically \'etale Noether normalization $f\: X \to \hat{\bA}_{\kay}^{d}$ which is flat (\cite[Corollary 18.17]{EisenbudCommutativeAlgebraWithAView}), say of rank $n$. By \cite{SkalitKoszulFactorization}, we may assume that $p \nmid n$ if $\kay=\kay^{\mathrm{alg}}$. Further, $\omega_R = \omega_{R/A}\cong R(\mathfrak{R})$ and so $\mathfrak{R}$ is a canonical divisor on $X$. By \cite[Theorem 6.8]{AltmanKleimanIntroToGrothendieckDuality}, the branch locus and ramification locus of $f$ are divisors. Following \cite[I Exercise 3.9]{MilneEtaleCohomology}, denote the different ideal of $f$ by $\mathfrak{d}_{R/A} \subset R$ and its discriminant ideal by $\mathfrak{D}_{R/A} \subset A$.  These ideals are related by the equality $\Norm_{R/A}(\mathfrak{d}_{R/A}) = \mathfrak{D}_{R/A}$. In particular, $\mathfrak{d}_{R/A} = R(-\mathfrak{R})$, $\mathfrak{D}_{R/A}=A(-\mathfrak{B})$. Note that $\mathfrak{B}$ is a Cartier divisor since $f$ is flat and further principal since $\hat{\bA}_{\kay}^d$ has trivial Picard group. In fact, $\mathfrak{B} = \Div \delta$ where $\delta$ is the discriminant of $T$, \eg $\mathfrak{D}_{R/A}=(\delta)$. If $R$ is Gorenstein, $\omega_R$ is a free rank $1$ $R$-module, and so there is a free generator $T\: R \to A$ for the $R$-module $\Hom_A(R,A)$. In particular, there is a unique $\rho \in R$ such that $\Tr_{R/A} \: R \to A$ equals $T \cdot \rho $, and so $\mathfrak{d}_{R/A}= (\rho)$, $\mathfrak{R} = \Div_R \rho$, $\mathfrak{D}_{R/A} = (\delta)$, and $\mathfrak{B}=\Div_A \delta$, where $\delta \coloneqq \Norm_{R/A}(\rho)$.

\begin{corollary} \label{thm.RadicalNN}
With notation as above, let $(\delta) = \mathfrak{D}_{R/A}$ and $\Delta \coloneqq \frac{1}{n} \cdot \Div_A \delta$. Then, there exists a rational number $1 \leq c \leq n$ such that $\sC_A^{c \cdot \Delta} \subset \sC_A^{\top} \subset \sC_A^{\Delta}$, where transposition is defined with respect to $\Tr_{R/A} \: R \to A$. We may take $c=1$ if $\mathfrak{R}$ is $f$-torsion. Assume that $\kay$ algebraically closed and $p \nmid n$. If the pair $(A,c\cdot \Delta)$ is $F$-regular (resp. $F$-pure) then so is $R$ and $s(R) \geq n \cdot s(A,c \cdot \Delta)$. The converse and equality hold if $\mathfrak{R}$ is $f$-torsion. 
\end{corollary}
\begin{proof}
The first part is a particular case of \autoref{thm.TransposabilityFaithfullyFlat} with $T = \Tr_{R/A}$. For the second part, observe that $\Tr_{R/A}$ is surjective, and $\Tr_{R/A}(\fram) \subset (x_1,\ldots,x_d)$ by \cite[Lemma 2.10]{CarvajalSchwedeTuckerEtaleFundFsignature}. Then, apply \autoref{thm.TransRuleSplittingRatios} and \autoref{thm.MainTheorem}. For the final statement, note that if $\mathfrak{R}$ is $f$-torsion then $\sC_A^{\top} = \sC_A^{\Delta}$ and $f^* \Delta = \mathfrak{R}$.
\end{proof}

We end with some examples illustrating \autoref{thm.RadicalNN}; beginning with two Gorenstein singularities known for their rather mysterious $F$-signature.

\begin{example}
Consider $R = \kay\llbracket x_0, \ldots, x_n\rrbracket/(x_0^2 + \cdots + x_n^2)$, where $\Char \kay \neq 2$. We can directly apply \autoref{thm.RadicalNN} by taking $A = \kay\llbracket x_1^2, \ldots, x_n^2\rrbracket$. The ramification divisor is given by $\mathfrak{R} = \Div_R \rho$ where $\rho = x_0 \cdot x_1 \cdots x_n$, and $\Norm_{R/A}(\rho) = \rho^2 = \bigl(x_1^2 + \cdots + x_n^2\bigr)\cdot x_1^2 \cdots x_n^2$. In particular, $\mathfrak{R}$ is $f$-torsion. Thus, $s(R) = 2^{n+1} \cdot s(A, \Delta), \text{ where } \Delta = \frac{1}{2} \cdot \Div_A (x_1^2 + \cdots+ x_n^2) \cdot x_1^2 \cdots x_n^2$. Let $\Delta'=\frac{1}{2} \cdot \Div_A x_1^2 \cdots x_n^2$, then $s(A, \Delta') = 1/2^n$ by \cite[Example 4.19]{BlickleSchwedeTuckerFSigPairs1}. However, this gives no nontrivial upper bound for $s(R)$. The $F$-signature of $R$ (for large $n$) is quite involved and depends on $p$. It can be computed by applying \cite[Example 2.3]{WatanabeYoshidaMinimalRelative} to the results in \cite{GesselMonskySumOfSquares}. It turns out that $\lim_{p \rightarrow \infty} s(R)$ is the coefficient of $z^n$ in the Taylor series of $\sec z + \tan z$.
\end{example}

\begin{example}
Set $R = \mathbb{F}_2\llbracket x,y,z,u,v \rrbracket/(x^3 + y^3 + xyz + uv)$. A Noether normalization is $A = \bF_2\llbracket y,z, u,v\rrbracket$ and $1, x, x^2$ is an $A$-basis of $R$. A generator of $\Hom_A(R,A)$ is $(x^2)^\vee $, for $(x^2)^\vee \cdot x = x^\vee$ and $(x^2)^\vee \cdot (x^2 + yz) = 1^\vee$. A short computation shows that $\Tr_{R/A} = 1^\vee$. Hence, $\mathfrak{R}=\Div_R \rho$ with $\rho = x^2+yz$. On the other hand, $x \rho = y^3+uv \eqqcolon \epsilon $. Then, from $x^3+(yz)x + \epsilon = 0$, we obtain the minimal equation $\rho^3+(yz)\rho^2+\epsilon^2=0
$. Therefore, $\delta \coloneqq \Norm_{R/A}(\rho) = \epsilon^2$ by a direct computation. Moreover, $
\rho^2(\rho+yz)=\rho^2 x^2 = \delta$. Hence, we may take $c=2/3$ in \autoref{thm.RadicalNN} to say that
\[
\sC_A^{\frac{1}{2} \cdot \Div_A \delta} \subset \sC_A^{\top} \subset \sC_A^{\frac{1}{3} \cdot \Div_A \delta }.
\]
We claim that the left inclusion is an equality, \ie $\sC_A^{\top} = \sC_A^{\Div_A \epsilon}$. Indeed, let $\varphi \in \sC_{e,A}^{\top}$ and write $\varphi = \Phi^e \cdot a$, where $\Phi$ is the Frobenius trace of $A$. We show that $\Delta_{\varphi} \geq \Div_A \epsilon$. We know that $\Delta_{\varphi} = \frac{1}{2^e-1} \cdot \Div_A a$ and $\frac{1}{2^e-1} \cdot \Div_R a = f^* \Delta_{\varphi} \geq \Div_R \rho = \Div_R \epsilon - \Div_R x$. That is, $\Div_R a x^{2^e-1} \geq \Div_R \epsilon^{2^e-1}$, so $ax^{2^e-1}/\epsilon^{2^e-1} \in R$. To grasp this better, we need:

\begin{claim} \label{claim.powersofx}
Let $a_e,b_e,c_e \in A$ be defined by $x^{2^e-1}=a_e x^2+b_e x+c_e$ in $R$. Then $a_e=0$ for all $e$, $b_1=1$, $c_1=0$, and the recursive formulas hold: $b_{e+1}=(yz)b_e^2+c_e^2$ and $c_{e+1}=b_e^2 \epsilon$.
\end{claim}
\begin{proof}[Proof of claim]
The values for $a_e,b_e,c_e$ for $e=1$ are clear. Next, we observe that
\[
x^{2^{e+1}-1} = \bigl(x^{2^e-1}\bigr)^2 \cdot x = \bigl( a_e x^2+b_e x+c_e \bigr)^2 \cdot x =a_e^2 x^5 + b_e^2 x^3 + c_e^2 x.
\]
Using $x^3=(yz)x+\epsilon$, we verify that $x^5=\epsilon x^2 + (yz)^2x + (yz)\epsilon$. Consequently,
\begin{align*}
x^{2^{e+1}-1} &= a_e^2\bigl(\epsilon x^2 + (yz)^2x + (yz)\epsilon \bigr) + b_e^2 \bigl((yz)x+\epsilon\bigr)+c_e^2 x\\
&=\epsilon a_e^2 x^2 + \bigl((yz)^2a_e^2+(yz)b_e^2+c_e^2\bigr)x+(yz)\epsilon a_e^2+b_e^2 \epsilon.
\end{align*}
Hence, $a_{e+1}=\epsilon a_e^2$ and so $a_e=0$ for all $e$ as $a_1=0$. This also gives the desired formulas.
\end{proof}

With notation as in \autoref{claim.powersofx}, we have
\[
\frac{ab_e}{\epsilon^{2^e-1}} \cdot x + \frac{ac_e}{\epsilon^{2^e-1}} = \frac{a x^{2^e-1}}{\epsilon^{2^e-1}} \in R.
\]
Therefore, $ab_e/\epsilon^{2^e-1}$ and $ac_e/\epsilon^{2^e-1}$ belong to $A$. Since $ab_e/\epsilon^{2^e-1} \in A$, we have $\Div_A a + \Div_A b_e \geq (2^e-1) \Div_A \epsilon$. Thus, to prove $\Delta_{\varphi} = \frac{1}{2^e-1} \Div_A a \geq \Div_A \epsilon$, it suffices to show that $\Div_A b_e$ has no support along the prime divisor $ \Div_A \epsilon$, \ie $\val_{(\epsilon)} b_e \leq 0$. Note that \autoref{claim.powersofx} implies $b_{e} \equiv (yz)^{2^{e-1}-1} \bmod \epsilon$. Hence, $\val_{(\epsilon)} b_e \leq (2^{e-1}-1)\val_{(\epsilon)} yz$ and we are left with showing $\val_{(\epsilon)} yz \leq 0$ which is clear as $y,z \notin (\epsilon)$.

In summary, $\sC_A^{\top} = \sC_A^{\Div_A \epsilon}$ which is principally generated by $\Phi \cdot \epsilon^{2-1} = \Phi \cdot \epsilon$ and so not $F$-regular as its splitting prime contains $(\epsilon)$. Consequently, since $f^* \Div_A \epsilon - \Div_R \rho = \Div_R x$, we have that $s(R,\Div_R x) = 3 \cdot s(A,\Div_A \epsilon)=0$, which gives no information about the $F$-signature of $R$. However, we obtain $r(R,\Div_R x) = r(A,\Div_A \epsilon)>0$ as $(A,\Div_A \epsilon)$ is $F$-pure. Finally, we remark that $s(R)$ is conjectured to be $\frac{2}{3}-\frac{5 \sqrt{7}}{98}$; see \cite[Proposition 4.22]{TuckerFSigExists}.
\end{example}

\begin{example} Our final example exhibits how much simpler than $\mathfrak{R}$ the divisor $\mathfrak{B}$ could be. Let $(V_{n,d},\fram)$ be the degree-$d$ Veronese subring of $\kay\llbracket x_1, \ldots, x_n\rrbracket$ (with $\kay$ $F$-finite). 

\begin{claim}
\label{le.VeroneseNNbasis}
$A \coloneqq \kay \bigl\llbracket x_1^d, \ldots, x_n^d\bigr\rrbracket \subset V_{n,d}$ is a Noether normalization with basis $B = \big\{ x_1^{\nu_1} \cdots x_n^{\nu_n} \bigm| d \mid \sum_i \nu_i \text{ and } 0 \leq \nu_i \leq d -1 \big\}$. Let $T_i$ be the dual of $x_1^{d-1} \cdots x_i^r \cdots x_n^{d-1}$ with respect to $B$, where $n - 1= \mu d + r$ and $0 \leq r < d-1$. Then, $A_{x_i^d} \subset (V_{n,d})_{x_i^d}$ is a successive radical extension, $T_i |_{D(x_i^d)}$ is a free generator of $\Hom_{A_{x_i^d}}\big((V_{n,d})_{x_i^d}, A_{x_i^d}\big)$, and on $D(x_i^d)$ we have
\[
T \coloneqq 1^\vee = T_i \cdot \prod_{j \neq i} x_i^{-((d-1)(n-1)+r)} \cdot \prod_{j \neq i} \left(x_j/x_i\right)^{d-1}.
\] 
In particular, $\Ram_T\big\vert_{D(x_i^d)} = \Div \prod_{ j \neq i} x_j^{d-1}$. Further, $T(1) =1$ and $T(\fram) \subset (x_1^d, \ldots, x_n^d)$.
\end{claim} 
\begin{proof}[Proof of claim]
By construction, $B$ is a generating set. Indeed, any monomial in $V_{n,d}$ of degree $< d$ in all $x_i$ is contained in this set. If some monomial has degree $\geq d$ in some $x_i$, then we can extract a suitable monomial $x_1^{a_1 d} \cdots x_n^{a_n d}$ and reduce to the previous case. To show that $B$ is linearly independent, we may invert $x_1^{d}$ (by symmetry). Then,
\begin{align*} 
A\bigl[x_1^{-d}\bigr] \cong \kay\bigl\llbracket x_1^d, (x_2/x_1)^d, \ldots,(x_n/x_1)^d \bigr\rrbracket \bigl[x_1^{-d}\bigr] \subset \kay\bigl\llbracket x_1^d, x_2/x_1, \ldots,  x_n/x_1 \bigr\rrbracket \bigl[x_1^{-d}\bigr] \cong V_{n,d}\bigl[x_1^{-d}\bigr] 
\end{align*} 
which is a successive radical extension with basis $C = \bigl\{ (x_2/x_1)^{i_2} \cdots (x_n/x_1)^{i_n} \bigm| 0 \leq i_j \leq d-1 \bigr\}$. We can transform $B$ into this basis by multiplying a monomial $x_1^{\nu_1} \cdots x_n^{\nu_n}$; with $\sum_i \nu_i = \lambda d$, by the unit $x_1^{-\lambda d}$. $\Hom_A(V_{n,d}, A)$ restricted to $D(x_1^d)$ is generated by $G \coloneqq \bigl( (x_2/x_1)^{d-1} \cdots (x_n/x_1)^{d-1}\bigr)^\vee$, where duals are taken with respect to $C$. In order to transfer this back to our original basis, note that $(x_1^{-n +1} \cdot x_2\cdots x_n)^{d-1}$ is mapped to the monomial $x_1^a x_2^{d-1} \cdots x_n^{d-1}$, where 
\begin{equation}\label{eq.congruence}
(d-1)(n-1) + a = \lambda d 
\end{equation} 
for some $0 \leq a \leq d-1$, and some $\lambda \geq 1$. Writing $n-1 = \mu d +r$ as above and taking \autoref{eq.congruence} modulo $d$, we find that $a \equiv r \bmod d$. Since $0 \leq a \leq d -1$, then $a = r$. Thus $T_1$ is a generator.

For the claim about $T$, note that the duals of $1$ with respect to $C$ and $B$ coincide. With respect to $C$: $1^\vee = G \cdot (x_2/x_1)^{d-1} \cdots (x_n/x_1)^{d-1}$. Clearly, $T(1) = 1$, and $T(\mathfrak{m}) = 0$.
\end{proof}

Note that if $\Char \kay \nmid d$ then $T = 1/d^n \Tr$. However, $T\neq 0$ even if $A \subset V_{n,d}$ is purely inseparable. By \autoref{le.VeroneseNNbasis}, $A \subset V_{n,d}$ is locally a radical extension. Indeed, the pullback of $\Ram_T$ to $\Spec V_{n,d} \smallsetminus \{\fram\}$ is given by the Cartier divisor $\big(D(x_i^d), \Div \prod_{ j \neq i} x_j^{d-1} \big )$. Denoting $\rho_i \coloneqq  \prod_{ j \neq i} x_j^{d-1}$, we see that $\rho_i^d = \big(\prod_{j \neq i} x_j^d \big)^{d-1}$. In particular,
\[
\Branch_T=\Norm_{R/A} (\Ram_T) =\left(D\bigl(x_i^d\bigr), (d-1) \Div \prod_{i\neq j} x_j^d\right) = (d-1) \Div x_1^d \cdots x_n^d,
\]
Thus, $\sC_A^{\top} = \sC_A^{\Delta}$ where $\Delta = \frac{d-1}{d} \Div x_1^d \cdots x_n^d$, and moreover
\[
s(V_{n,d}) = d^{n-1} \cdot s(A, \Delta) = d^{n-1} \cdot (1-(d-1)/d)^n = 1/d 
\]
where the middle equality is \cite[Example 4.19]{BlickleSchwedeTuckerFSigPairs1}. Of course, the $F$-signature of the $V_{n,d}$ is already known by results of Singh \cite{SinghFSignatureOfAffineSemigroup} and of Von Korff \cite{VonKorffFSigOfAffineToric}. 
\end{example}

\bibliographystyle{amsalpha}
\bibliography{MainBib}

\def\cfudot#1{\ifmmode\setbox7\hbox{$\accent"5E#1$}\else
  \setbox7\hbox{\accent"5E#1}\penalty 10000\relax\fi\raise 1\ht7
  \hbox{\raise.1ex\hbox to 1\wd7{\hss.\hss}}\penalty 10000 \hskip-1\wd7\penalty
  10000\box7}
\providecommand{\bysame}{\leavevmode\hbox to3em{\hrulefill}\thinspace}
\providecommand{\MR}{\relax\ifhmode\unskip\space\fi MR }
\providecommand{\MRhref}[2]{%
  \href{http://www.ams.org/mathscinet-getitem?mr=#1}{#2}
}
\providecommand{\href}[2]{#2}
\begin{thebibliography}{CRST18}

\bibitem[AE05]{AberbachEnescuStructureOfFPure}
Ian~M. Aberbach and Florian Enescu, \emph{The structure of {$F$}-pure rings},
  Math. Z. \textbf{250} (2005), no.~4, 791--806. \MR{2180375}

\bibitem[AK70]{AltmanKleimanIntroToGrothendieckDuality}
Allen Altman and Steven Kleiman, \emph{Introduction to {G}rothendieck duality
  theory}, Lecture Notes in Mathematics, Vol. 146, Springer-Verlag, Berlin-New
  York, 1970. \MR{0274461}

\bibitem[Art75]{ArtinWildlyRamifiedZ2Actions}
M.~Artin, \emph{Wildly ramified {$Z/2$} actions in dimension two}, Proc. Amer.
  Math. Soc. \textbf{52} (1975), 60--64. \MR{0374136 (51 \#10336)}

\bibitem[Bli13]{BlickleTestIdealsViaAlgebras}
Manuel Blickle, \emph{Test ideals via algebras of {$p^{-e}$}-linear maps}, J.
  Algebraic Geom. \textbf{22} (2013), no.~1, 49--83. \MR{2993047}

\bibitem[BS13]{BlickleSchwedeLinearMapsAlgebraAndGeometry}
Manuel Blickle and Karl Schwede, \emph{{$p^{-1}$}-linear maps in algebra and
  geometry}, Commutative algebra, Springer, New York, 2013, pp.~123--205.
  \MR{3051373}

\bibitem[BS19]{BlickleStablerFunctorialTestModules}
Manuel Blickle and Axel St\"{a}bler, \emph{Functorial test modules}, J. Pure
  Appl. Algebra \textbf{223} (2019), no.~4, 1766--1800. \MR{3906525}

\bibitem[BST12]{BlickleSchwedeTuckerFSigPairs1}
Manuel Blickle, Karl Schwede, and Kevin Tucker, \emph{{$F$}-signature of pairs
  and the asymptotic behavior of {F}robenius splittings}, Adv. Math.
  \textbf{231} (2012), no.~6, 3232--3258. \MR{2980498}

\bibitem[BST15]{BlickleSchwedeTuckerTestAlterations}
\bysame, \emph{{$F$}-singularities via alterations}, Amer. J. Math.
  \textbf{137} (2015), no.~1, 61--109. \MR{3318087}

\bibitem[CR18]{Carvajalphdthesis}
Javier Carvajal-Rojas, \emph{Arithmetic aspects of strong ${F}$-regularity},
  Ph.D. thesis, University of Utah, ProQuest LLC, Ann Arbor, MI, 2018, Thesis
  (Ph.D.).

\bibitem[CR22]{CarvajalFiniteTorsors}
\bysame, \emph{{Finite torsors over strongly $F$-regular singularities}},
  {\'Epijournal de G\'eom\'etrie Alg\'ebrique} \textbf{{Volume 6}} (2022).

\bibitem[CRF22]{CarvajalRojasFayolle}
Javier Carvajal-Rojas and Anne Fayolle, \emph{Centers of {$F$}-purity and their
  behavior under finite covers}, manuscript in preparation.

\bibitem[CRST18]{CarvajalSchwedeTuckerEtaleFundFsignature}
Javier Carvajal-Rojas, Karl Schwede, and Kevin Tucker, \emph{Fundamental groups
  of {$F$}-regular singularities via {$F$}-signature}, Ann. Sci. \'{E}c. Norm.
  Sup\'{e}r. (4) \textbf{51} (2018), no.~4, 993--1016. \MR{3861567}

\bibitem[EH08]{EnescuHochsterTheFrobeniusStructureOfLocalCohomology}
Florian Enescu and Melvin Hochster, \emph{The {F}robenius structure of local
  cohomology}, Algebra Number Theory \textbf{2} (2008), no.~7, 721--754.
  \MR{MR2460693 (2009i:13009)}

\bibitem[Eis95]{EisenbudCommutativeAlgebraWithAView}
David Eisenbud, \emph{Commutative algebra}, Graduate Texts in Mathematics, vol.
  150, Springer-Verlag, New York, 1995, With a view toward algebraic geometry.
  \MR{1322960 (97a:13001)}

\bibitem[Fed83]{FedderFPureRat}
Richard Fedder, \emph{{$F$}-purity and rational singularity}, Trans. Amer.
  Math. Soc. \textbf{278} (1983), no.~2, 461--480. \MR{701505 (84h:13031)}

\bibitem[FST11]{FujinoSchwedeTakagiSupplements}
Osamu Fujino, Karl Schwede, and Shunsuke Takagi, \emph{Supplements to non-lc
  ideal sheaves}, Higher Dimensional Algebraic Geometry, RIMS K\^oky\^uroku
  Bessatsu, B24, Res. Inst. Math. Sci. (RIMS), Kyoto, 2011, pp.~1--47.

\bibitem[FW89]{FedderWatanabe}
Richard Fedder and Keiichi Watanabe, \emph{A characterization of
  {$F$}-regularity in terms of {$F$}-purity}, Commutative algebra (Berkeley,
  CA, 1987), Math. Sci. Res. Inst. Publ., vol.~15, Springer, New York, 1989,
  pp.~227--245. \MR{1015520 (91k:13009)}

\bibitem[GM10]{GesselMonskySumOfSquares}
Ira~M. {Gessel} and Paul {Monsky}, \emph{The limit as $p \rightarrow \infty$ of
  the {H}ilbert--{K}unz multiplicity of $\sum x_i^{d_i}$}, arXiv e-prints
  (2010), arXiv:1007.2004.

\bibitem[Har66]{HartshorneResidues}
Robin Hartshorne, \emph{Residues and duality}, Lecture notes of a seminar on
  the work of A. Grothendieck, given at Harvard 1963/64. With an appendix by P.
  Deligne. Lecture Notes in Mathematics, No. 20, Springer-Verlag, Berlin, 1966.
  \MR{MR0222093 (36 \#5145)}

\bibitem[Har77]{Hartshorne}
\bysame, \emph{Algebraic geometry}, Springer-Verlag, New York, 1977, Graduate
  Texts in Mathematics, No. 52. \MR{0463157 (57 \#3116)}

\bibitem[Har94]{HartshorneGeneralizedDivisorsOnGorensteinSchemes}
\bysame, \emph{Generalized divisors on {G}orenstein schemes}, Proceedings of
  Conference on Algebraic Geometry and Ring Theory in honor of Michael Artin,
  Part III (Antwerp, 1992), vol.~8, 1994, pp.~287--339. \MR{1291023
  (95k:14008)}

\bibitem[Har07]{HartshonreGeneralizedDivisorsAndBiliaison}
\bysame, \emph{Generalized divisors and biliaison}, Illinois J. Math.
  \textbf{51} (2007), no.~1, 83--98 (electronic). \MR{2346188}

\bibitem[HH94]{HochsterHunekeTCParameterIdealsAndSplitting}
Melvin Hochster and Craig Huneke, \emph{Tight closure of parameter ideals and
  splitting in module-finite extensions}, J. Algebraic Geom. \textbf{3} (1994),
  no.~4, 599--670. \MR{1297848 (95k:13002)}

\bibitem[HK71]{HerzogKunzCanonicalModules}
J{\"u}rgen Herzog and Ernst Kunz (eds.), \emph{Der kanonische {M}odul eines
  {C}ohen-{M}acaulay-{R}ings}, Lecture Notes in Mathematics, Vol. 238,
  Springer-Verlag, Berlin-New York, 1971. \MR{0412177}

\bibitem[Hoc73]{HochsterContractedIdealsFromIntegralExtensions}
Melvin Hochster, \emph{Contracted ideals from integral extensions of regular
  rings}, Nagoya Math. J. \textbf{51} (1973), 25--43. \MR{0349656 (50 \#2149)}

\bibitem[HW15]{HunekeWatanabeUpperBoundMultiplicityFpireRIngs}
Craig Huneke and Kei-ichi Watanabe, \emph{Upper bound of multiplicity of
  {F}-pure rings}, Proc. Amer. Math. Soc. \textbf{143} (2015), no.~12,
  5021--5026. \MR{3411123}

\bibitem[Jac85]{JacobsonBasicAlgebraI}
Nathan Jacobson, \emph{Basic algebra. {I}}, second ed., W. H. Freeman and
  Company, New York, 1985. \MR{780184}

\bibitem[JS22]{JeffriesSmirnovTransformationRule}
Jack Jeffries and Ilya Smirnov, \emph{A transformation rule for natural
  multiplicities}, Int. Math. Res. Not. IMRN (2022), no.~2, 999--1015.
  \MR{4368877}

\bibitem[KM09]{KumarMehtaFiniteness}
Shrawan Kumar and Vikram~B. Mehta, \emph{Finiteness of the number of compatibly
  split subvarieties}, Int. Math. Res. Not. IMRN (2009), no.~19, 3595--3597.
  \MR{2539185 (2010j:13012)}

\bibitem[KS18]{KuranoShimomotoElementaryProof}
Kazuhiko Kurano and Kazuma Shimomoto, \emph{An elementary proof of
  {C}ohen--{G}abber theorem in the equal characteristic {$p>0$} case}, Tohoku
  Math. J. (2) \textbf{70} (2018), no.~3, 377--389. \MR{3856772}

\bibitem[Ma88]{MASplinters}
Frank Ma, \emph{Splitting in integral extensions, {C}ohen-{M}acaulay modules
  and algebras}, J. Algebra \textbf{116} (1988), no.~1, 176--195. \MR{944154}

\bibitem[Mil80]{MilneEtaleCohomology}
James~S. Milne, \emph{\'{E}tale cohomology}, Princeton Mathematical Series,
  vol.~33, Princeton University Press, Princeton, N.J., 1980. \MR{559531}

\bibitem[MS21]{MaSchwedeSingularitiesNixedCharacteriticPerfectoidBCM}
Linquan Ma and Karl Schwede, \emph{Singularities in mixed characteristic via
  perfectoid big {C}ohen-{M}acaulay algebras}, Duke Math. J. \textbf{170}
  (2021), no.~13, 2815--2890. \MR{4312190}

\bibitem[Sch09]{SchwedeFAdjunction}
Karl Schwede, \emph{{$F$}-adjunction}, Algebra Number Theory \textbf{3} (2009),
  no.~8, 907--950. \MR{2587408 (2011b:14006)}

\bibitem[Sch10]{SchwedeCentersOfFPurity}
\bysame, \emph{Centers of {$F$}-purity}, Math. Z. \textbf{265} (2010), no.~3,
  687--714. \MR{2644316 (2011e:13011)}

\bibitem[Sha07]{SharpGradedAnnihilatorsOfModulesOverTheFrobeniusSkewPolynomialRing}
Rodney~Y. Sharp, \emph{Graded annihilators of modules over the {F}robenius skew
  polynomial ring, and tight closure}, Trans. Amer. Math. Soc. \textbf{359}
  (2007), no.~9, 4237--4258 (electronic). \MR{MR2309183 (2008b:13006)}

\bibitem[Sin99]{SinghQGorensteinSplinters}
Anurag~K. Singh, \emph{{$\bold Q$}-{G}orenstein splinter rings of
  characteristic {$p$} are {F}-regular}, Math. Proc. Cambridge Philos. Soc.
  \textbf{127} (1999), no.~2, 201--205. \MR{1735920 (2000j:13006)}

\bibitem[Sin03]{SinghCyclicCoversOfRational}
\bysame, \emph{Cyclic covers of rings with rational singularities}, Trans.
  Amer. Math. Soc. \textbf{355} (2003), no.~3, 1009--1024 (electronic).
  \MR{MR1938743 (2003m:13006)}

\bibitem[Sin05]{SinghFSignatureOfAffineSemigroup}
\bysame, \emph{The {$F$}-signature of an affine semigroup ring}, J. Pure Appl.
  Algebra \textbf{196} (2005), no.~2-3, 313--321. \MR{2110527 (2005m:13010)}

\bibitem[Ska16]{SkalitKoszulFactorization}
C.~Skalit, \emph{Koszul factorization and the {C}ohen--{G}abber theorem},
  Illinois J. Math. \textbf{60} (2016), no.~3-4, 833--844. \MR{3705447}

\bibitem[Smi97]{SmithFRatImpliesRat}
Karen~E. Smith, \emph{{$F$}-rational rings have rational singularities}, Amer.
  J. Math. \textbf{119} (1997), no.~1, 159--180. \MR{1428062 (97k:13004)}

\bibitem[Smo19a]{SmolkinSubadditivity}
Daniel Smolkin, \emph{A new subadditivity formula for test ideals}, Journal of
  Pure and Applied Algebra (2019).

\bibitem[Smo19b]{Smolkinphdthesis}
\bysame, \emph{Subadditivity of test ideals and diagonal ${F}$-regularity},
  Ph.D. thesis, University of Utah, ProQuest LLC, Ann Arbor, MI, 2019, Thesis
  (Ph.D.).

\bibitem[Smo19c]{SmolkinGeneralAdjointIdeals}
\bysame, \emph{Test ideals along closed subschemes},
  \url{http://www-personal.umich.edu/~smolkind/files/tauAlongI.pdf}, 2019.

\bibitem[Spe20]{SpeyerFrobeniusSplit}
David~E. Speyer, \emph{Frobenius split subvarieties pull back in almost all
  characteristics}, J. Commut. Algebra \textbf{12} (2020), no.~4, 573--579.
  \MR{4194942}

\bibitem[ST10]{SchwedeTuckerNumberOfFSplit}
Karl Schwede and Kevin Tucker, \emph{On the number of compatibly {F}robenius
  split subvarieties, prime {$F$}-ideals, and log canonical centers}, Ann.
  Inst. Fourier (Grenoble) \textbf{60} (2010), no.~5, 1515--1531. \MR{2766221}

\bibitem[ST12]{SchwedeTuckerTestIdealSurvey}
\bysame, \emph{A survey of test ideals}, Progress in Commutative Algebra 2.
  Closures, Finiteness and Factorization (C.~Francisco, L.~C. Klinger, S.~M.
  Sather-Wagstaff, and J.~C. Vassilev, eds.), Walter de Gruyter GmbH \& Co. KG,
  Berlin, 2012, pp.~39--99.

\bibitem[ST14]{SchwedeTuckerTestIdealFiniteMaps}
\bysame, \emph{On the behavior of test ideals under finite morphisms}, J.
  Algebraic Geom. \textbf{23} (2014), no.~3, 399--443. \MR{3205587}

\bibitem[ST15]{SchwedeTuckerExplicitlyExtending}
\bysame, \emph{Explicitly extending {F}robenius splittings over finite maps},
  Comm. Algebra \textbf{43} (2015), no.~10, 4070--4079. \MR{3366560}

\bibitem[St{\"a}17]{staeblerunitftestmodules}
Axel St{\"a}bler, \emph{Test module filtrations for unit {$F$}-modules}, J.
  Algebra \textbf{477} (2017), 435--471. \MR{3614158}

\bibitem[Tak08]{TakagiPLTAdjoint}
Shunsuke Takagi, \emph{A characteristic {$p$} analogue of plt singularities and
  adjoint ideals}, Math. Z. \textbf{259} (2008), no.~2, 321--341. \MR{2390084
  (2009b:13004)}

\bibitem[Tak10]{TakagiHigherDimensionalAdjoint}
\bysame, \emph{Adjoint ideals along closed subvarieties of higher codimension},
  J. Reine Angew. Math. \textbf{641} (2010), 145--162. \MR{2643928
  (2011f:14032)}

\bibitem[{The}22]{stacks-project}
{The Stacks Project Authors}, \emph{\textit{Stacks Project}},
  \url{http://stacks.math.columbia.edu}, 2022.

\bibitem[Tuc12]{TuckerFSigExists}
Kevin Tucker, \emph{{$F$}-signature exists}, Invent. Math. \textbf{190} (2012),
  no.~3, 743--765. \MR{2995185}

\bibitem[V{\'e}l95]{VelezOpennessOfTheFRationalLocus}
Juan~D. V{\'e}lez, \emph{Openness of the {F}-rational locus and smooth base
  change}, J. Algebra \textbf{172} (1995), no.~2, 425--453. \MR{MR1322412
  (96g:13003)}

\bibitem[VK12]{VonKorffFSigOfAffineToric}
Michael~R. Von~Korff, \emph{The {F}-{S}ignature of {T}oric {V}arieties},
  ProQuest LLC, Ann Arbor, MI, 2012, Thesis (Ph.D.)--University of Michigan.
  \MR{3093997}

\bibitem[WY04]{WatanabeYoshidaMinimalRelative}
Kei-ichi Watanabe and Ken-ichi Yoshida, \emph{Minimal relative
  {H}ilbert--{K}unz multiplicity}, Illinois J. Math. \textbf{48} (2004), no.~1,
  273--294. \MR{2048225 (2005b:13033)}

\bibitem[Yao05]{YaoModulesWithFFRT}
Yongwei Yao, \emph{Modules with finite {$F$}-representation type}, J. London
  Math. Soc. (2) \textbf{72} (2005), no.~1, 53--72. \MR{2145728 (2006b:13012)}

\end{thebibliography}

\end{document}